\documentclass{amsart}
\usepackage{epsfig}
\usepackage{graphicx}
\usepackage{color}
\usepackage{amssymb}
\usepackage{array}
\usepackage{longtable}
\usepackage{arydshln}
\usepackage[labelfont=sc]{caption}

\def\R{\mathbb R}

\def\N{\mathbb N}

\def\romega{{r_{\phantom{a}\!\!\!\!_\Omega}}}
\def\rgamma{{r_{\phantom{a}\!\!\!\!_\Gamma}}}
\def\somega{{s_{\phantom{a}\!\!\!\!_\Omega}}}
\def\oversomega{{\overline{s}_{\phantom{a}\!\!\!\!_\Omega}}}
\def\sgamma{{s_{\phantom{a}\!\!\!\!_\Gamma}}}
\def\sigmaomega{{\sigma_{\phantom{a}\!\!\!\!_\Omega}}}
\def\sigmagamma{{\sigma_{\phantom{a}\!\!\!\!_\Gamma}}}
\def\lomega{{l_{\phantom{a}\!\!\!\!_\Omega}}}
\def\lgamma{{l_{\phantom{a}\!\!\!\!_\Gamma}}}

\def\loc{{\text{\upshape loc}}}
\def\negquad{\!\!\!\!}

\def\eps{\varepsilon}

\newcommand{\cal}[1]{{\mathcal #1}}
\DeclareMathOperator\essinf{essinf}
\DeclareMathOperator*\essliminf{essliminf}

\numberwithin{equation}{section}

\newtheorem{thm}{Theorem}
\numberwithin{thm}{section}
\newtheorem{cor}{Corollary}
\numberwithin{cor}{section}
\newtheorem{lem}{Lemma}
\numberwithin{lem}{section}
\newtheorem{prop}{Proposition}
\numberwithin{prop}{section}

\numberwithin{summary}{section}
\theoremstyle{definition}
\newtheorem{definition}{Definition}
\numberwithin{definition}{section}
\theoremstyle{remark}
\newtheorem{rem}{Remark}
\numberwithin{rem}{section}

\begin{document}
\title[On the wave equation with hyperbolic...]
{On the  wave equation with hyperbolic dynamical  boundary
conditions, interior and boundary damping and supercritical sources}
\author{Enzo Vitillaro}
\address[E.~Vitillaro]
       {Dipartimento di Matematica ed Informatica, Universit\`a di Perugia\\
       Via Vanvitelli,1 06123 Perugia ITALY}
\email{enzo.vitillaro@unipg.it}
\date{\today}



\thanks{Work done in the framework of the M.I.U.R. project
"Variational and perturbative aspects of nonlinear differential
problems" (Italy)}

\begin{abstract} The aim of this paper is to study  the problem
$$
\begin{cases} u_{tt}-\Delta u+P(x,u_t)=f(x,u) \qquad &\text{in
$(0,\infty)\times\Omega$,}\\
u=0 &\text{on $(0,\infty)\times \Gamma_0$,}\\
u_{tt}+\partial_\nu u-\Delta_\Gamma u+Q(x,u_t)=g(x,u)\qquad
&\text{on
$(0,\infty)\times \Gamma_1$,}\\
u(0,x)=u_0(x),\quad u_t(0,x)=u_1(x) &
 \text{in $\overline{\Omega}$,}
\end{cases}$$
where $\Omega$ is a bounded  open subset of $\R^N$ with $C^1$ boundary ($N\ge 2$),
$\Gamma=\partial\Omega$, $\Gamma_1$ is relatively open on $\Gamma$,
  $\Delta_\Gamma$ denotes the
Laplace--Beltrami operator on $\Gamma$, $\nu$ is the outward normal
to $\Omega$, and the terms $P$ and $Q$ represent nonlinear damping
terms, while $f$ and $g$ are nonlinear perturbations.

In the paper we establish local and global existence, uniqueness and Hadamard
well--posedness results when  source terms can be supercritical or
super-supercritical.
\end{abstract}

\maketitle

\section{Introduction and main results} \label{intro}
\subsection{Presentation of the problem and literature overview}
We deal with the evolution problem consisting of
the
 wave equation posed in a bounded regular open subset  of
$\R^N$, supplied with a second order dynamical boundary condition of
hyperbolic type, in presence of interior and/or boundary damping
terms and sources. More precisely we consider the
initial --and--boundary value problem
\begin{equation}\label{1.1}
\begin{cases} u_{tt}-\Delta u+P(x,u_t)=f(x,u) \qquad &\text{in
$(0,\infty)\times\Omega$,}\\
u=0 &\text{on $(0,\infty)\times \Gamma_0$,}\\
u_{tt}+\partial_\nu u-\Delta_\Gamma u+Q(x,u_t)=g(x,u)\qquad
&\text{on
$(0,\infty)\times \Gamma_1$,}\\
u(0,x)=u_0(x),\quad u_t(0,x)=u_1(x) &
 \text{in $\overline{\Omega}$,}
\end{cases}
\end{equation}
where $\Omega$ is a bounded open subset of $\R^N$ ($N\ge
2$) with $C^1$ boundary (see \cite{grisvard}). We denote  $\Gamma=\partial\Omega$ and we assume
$\Gamma=\Gamma_0\cup\Gamma_1$, $\Gamma_0\cap\Gamma_1=\emptyset$,
$\Gamma_1$ being relatively open on $\Gamma$ (or equivalently $\overline{\Gamma_0}=\Gamma_0$). Moreover,  denoting by
$\sigma$ the standard Lebesgue hypersurface measure on $\Gamma$, we assume that
$\sigma(\overline{\Gamma}_0\cap\overline{\Gamma}_1)=0$.
 These
properties of $\Omega$, $\Gamma_0$ and $\Gamma_1$ will be assumed,
without further comments, throughout the paper.  Moreover $u=u(t,x)$, $t\ge 0$, $x\in\Omega$,
$\Delta=\Delta_x$ denotes the Laplace operator with respect to the space
variable, while $\Delta_\Gamma$ denotes the Laplace--Beltrami
operator on $\Gamma$ and $\nu$ is the outward normal to $\Omega$.

The terms $P$ and $Q$ represent nonlinear damping terms, i.e.
$P(x,v)v\ge 0$, $Q(x,v)v\ge 0$,  the cases $P\equiv 0$ and $Q\equiv
0$ being specifically allowed, while  $f$ and $g$  represent
nonlinear source, or sink, terms. The specific assumptions on them
will be introduced later on.

Problems with kinetic boundary conditions, that is boundary
conditions involving $u_{tt}$ on $\Gamma$, or on a part of it,
naturally arise in several physical applications. A one dimensional
model was studied by several authors to describe transversal small
oscillations of an elastic rod with a tip mass on one endpoint,
while the other one is pinched. See
\cite{andrewskuttlershillor,conradmorgul,darmvanhorssen,guoxu,markuslittman1,markuslittman2,morgulraoconrad} and also
\cite{meurerkugi} were a piezoelectric stack actuator is modeled.

 A two dimensional model introduced in
\cite{goldsteingiselle} deals with a vibrating membrane of surface
density $\mu$, subject to a tension $T$, both taken constant and
normalized here for simplicity.  If $u(t,x)$,
$x\in\Omega\subset\R^2$  denotes the vertical displacement from the
rest state, then (after a standard linear approximation) $u$
satisfies the wave equation $u_{tt}-\Delta u=0$,
$(t,x)\in\R\times\Omega$. Now suppose that a part $\Gamma_0$ of the
boundary  is pinched, while the other part $\Gamma_1$ carries a
constant linear mass density $m>0$ and  it is subject to a linear
tension $\tau$. A practical example of this situation is given by a
drumhead with a hole in the interior having  a thick border, as
common in bass drums. One linearly approximates the force exerted by
the membrane on the boundary with $-\partial_\nu u$. The boundary
condition thus reads as $mu_{tt}+\partial_\nu u-\tau
\Delta_{\Gamma_1}u=0$. In the quoted paper the case
$\Gamma_0=\emptyset$ and $\tau=0$ was studied, while here we
consider the more realistic case $\Gamma_0\not=\emptyset$ and
$\tau>0$, with $\tau$ and $m$ normalized for simplicity. We would
like to mention that this model belongs to a more general class of
models of Lagrangian type involving boundary energies, as introduced
for example in \cite{fagotinreyes}.

A three dimensional model involving kinetic dynamical boundary
conditions comes out from \cite{GGG}, where a gas undergoing small
irrotational perturbations from rest in a domain $\Omega\subset\R^3$
is considered. Normalizing the constant speed of propagation, the
velocity potential $\phi$ of the gas (i.e. $-\nabla \phi$ is the
particle velocity) satisfies the wave equation $\phi_{tt}-\Delta
\phi=0$ in $\R\times\Omega$. Each point $x\in\partial\Omega$ is
assumed to react to the excess pressure of the acoustic wave like a
resistive harmonic oscillator or spring, that is the boundary is
assumed to be locally reacting (see \cite[pp.
259--264]{morseingard}). The normal displacement $\delta$ of the
boundary into the domain then satisfies
$m\delta_{tt}+d\delta_t+k\delta+\rho\phi_t=0$, where $\rho>0$ is the
fluid density  and $m,d,k\in C(\partial\Omega)$, $m,k>0$, $d\ge 0$.
When the boundary is nonporous one has $\delta_t=\partial_{\nu}\phi$
on $\R\times\partial\Omega$, so the boundary condition reads as
$m\delta_{tt}+d\partial_\nu \phi+k\delta+\rho\phi_t=0$. In the
particular case $m=k$ and $d=\rho$  (see  \cite[Theorem~2]{GGG}) one
proves that $\phi_{|\Gamma}=\delta$, so the boundary condition reads
as  $m\phi_{tt}+d\partial_\nu \phi+k\phi+\rho\phi_t=0$, on
$\R\times\partial\Omega$.
 Now, if one consider the case in which the
boundary is not locally reacting, as in \cite{beale}, one adds
a Laplace--Beltrami term  so getting a dynamical
boundary condition like in \eqref{1.1}.

Several papers in the literature deal with the wave equation  with
kinetic boundary conditions. This fact is even more evident  if one
takes into account that, plugging the equation in \eqref{1.1} into
the boundary condition, we can rewrite it as  $\Delta u
+\partial_\nu u-\Delta_\Gamma u+ Q(x,u_t)+P(x,u_t)=f(x,u)+g(x,u)$.
Such a condition is usually called a {\em generalized Wentzell
boundary condition}, at least when nonlinear perturbations are not
present. We refer to \cite{mugnolo2011}, where abstract semigroup techniques are applied to dissipative wave equations, and to
\cite{doroninlarkinsouza,FGGGR,vazvitM3AS,xiaoliang2,xiaoliang1, zhang}.

Here we shall consider this type of kinetic boundary condition in
connection with nonlinear boundary damping and source terms. These
terms have been considered by several authors, but mainly in
connection with first order dynamical boundary conditions. See
\cite{MR2674175,
MR2645989,bociuNAMA,bociulasiecka2,bociulasiecka1,CDCL,CDCM,chueshovellerlasiecka,fisvit2,lastat,global, zuazua}.
The competition between interior damping and source terms is
methodologically related to the competition between boundary damping
and source and it possesses a large literature as well. See
\cite{MR2609953,georgiev,levserr,ps:private,radu1,STV,blowup}.

A linear problem strongly related to \eqref{1.1} has also
been recently studied in \cite{Fourrier, graberlasiecka}, and another one in the recent paper \cite{zahn}, dealing with
holography, a main theme in theoretical high energy physics and
quantum gravity. See also \cite{graber}.

 Problem \eqref{1.1} has been studied by  the author in the recent paper \cite{Dresda1} (see also \cite{AMS})
when source/sink terms are subcritical, that is when the Nemitskii operators associated to $f$ and $g$
\begin{footnote}{defined on $\Gamma_1$ and trivially extended to $\Gamma_0$.}\end{footnote}
are locally Lischitz, respectively from $H^1(\Omega)$ to $L^2(\Omega)$ and from $H^1(\Gamma)$ to $L^2(\Gamma)$.
In particular, using nonlinear semigroup theory, local Hadamard
well--posedness, and hence also local existence and uniqueness, has
been established in the natural energy space related to the problem. Moreover global existence and well--posedness
have been proved when source terms are either sublinear or dominated by corresponding damping terms.

The aim of the present paper is to extend the results of
\cite{Dresda1} to possibly non--subcritical perturbation terms $f$ and $g$.
As a consequence we do not expect
to get new results when all perturbation terms are subcritical, as in the case $N=2$, but we want to cover the case when one term is subcritical
while the other one is non--subcritical. As a byproduct the results in
\cite{Dresda1} will be a subcase of the more general analysis presented in the sequel.

Our main motivation is constituted by the three dimensional case, in which only the term $f$ can be non--subcritical.
Hence our choice to consider also non--subcritical boundary terms $g$ is only of mathematical interest, but it is motivated as follows.
At first in dimensions higher than $3$ both terms can be non--subcritical. At second this extension is costless,
since {\em all estimates} used in the paper will be explicitly proved  only for $f$ (in relation with $P$) and then trivially transposed to $g$.
Finally to consider terms $f$ and $g$ under the same setting allows to give results in which $f$ and $g$ play a symmetric role, as they do in dimension $N=4$.
\subsection{A simplified problem, source classification and main assumptions}\label{sourceclassification}
To best illustrate our results  we shall consider, in this section, the following simplified version
of problem \eqref{1.1}
\begin{equation}\label{1.2}
\begin{cases} u_{tt}-\Delta u+\alpha(x)P_0(u_t)=f_0(u) \qquad &\text{in
$(0,\infty)\times\Omega$,}\\
u=0 &\text{on $(0,\infty)\times \Gamma_0$,}\\
u_{tt}+\partial_\nu u-\Delta_\Gamma u+\beta(x)Q_0(u_t)=g_0(u)\qquad
&\text{on
$(0,\infty)\times \Gamma_1$,}\\
u(0,x)=u_0(x),\quad u_t(0,x)=u_1(x) &
 \text{in $\overline{\Omega}$,}
\end{cases}
\end{equation}
where $\alpha\in L^\infty(\Omega)$, $\beta\in L^\infty(\Gamma_1)$,
$\alpha,\beta\ge0$, and the following assumptions hold:
\renewcommand{\labelenumi}{{(\Roman{enumi})}}
\begin{enumerate}
\item \label{assumptionI}
$P_0$ and $Q_0$ are continuous and monotone increasing in $\R$,
$P_0(0)=Q_0(0)=0$, and there are $m,\mu >1$ such that
\begin{gather*}0<\varliminf_{|v|\to \infty} \frac {|P_0(v)|}{|v|^{m-1}}  \le \varlimsup_{|v|\to \infty} \frac {|P_0(v)|}{|v|^{m-1}}<\infty,\quad
\varliminf_{|v|\to 0} \frac {|P_0(v)|}{|v|^{m-1}}>0,\\
0<\varliminf_{|v|\to \infty} \frac {|Q_0(v)|}{|v|^{\mu-1}}  \le
\varlimsup_{|v|\to \infty} \frac
{|Q_0(v)|}{|v|^{\mu-1}}<\infty,\quad \varliminf_{|v|\to 0} \frac
{|Q_0(v)|}{|v|^{\mu-1}}>0;
\end{gather*}
\item \label{assumptionII} $f_0,g_0\in C^{0,1}_{\text{loc}}(\R)$ and there are exponents
$p,q\ge 2$ such that
$$|f_0'(u)|=O(|u|^{p-2})\qquad\text{and}\quad  |g_0'(u)|=O(|u|^{q-2})\quad\text{as
$|u|\to\infty$.}$$
\end{enumerate}
Our model nonlinearities, trivially satisfying assumptions (I--II),  are given by
\begin{equation}\label{modelli}
\left\{
\begin{alignedat}3 P_0(v)=&P_1(v):=a|v|^{\widetilde{m}-2}v+ |v|^{m-2}v,
 \quad&& 1<\widetilde{m}\le m,\quad &&a\ge 0,\\
Q_0(v)=&Q_1(v):=b|v|^{\widetilde{\mu}-2}v+ |v|^{\mu-2}v, \quad  &&
1<\widetilde{\mu}\le \mu, \quad &&b\ge 0,
\\
f_0(u)=&f_1(u):=\widetilde{\gamma}|u|^{\widetilde{p}-2}u+
\gamma|u|^{p-2}u+c_1,\quad &&2\le\widetilde{p}\le p,\quad
&&\widetilde{\gamma}, \gamma, c_1\in\R,
\\
g_0(u)=&g_1(u):=\widetilde{\delta}|u|^{\widetilde{q}-2}u+
\delta|u|^{q-2}u+c_2,\, \quad &&2\le\widetilde{q}\le q,\quad
&&\widetilde{\delta}, \delta, c_2\in\R.
\end{alignedat}
\right.
\end{equation}
We introduce the
critical exponents $\romega$ and $\rgamma$ of the Sobolev embeddings of $H^1(\Omega)$ and
$H^1(\Gamma)$ into the corresponding Lebesgue spaces, that is
$$\romega=
\begin{cases}
\dfrac {2N}{N-2} &\text{if $N \ge 3$,}\\ \infty &\text{if $N=2$},
\end{cases}
\qquad \rgamma=
\begin{cases}
\dfrac {2(N-1)}{N-3} &\text{if $N \ge 4$,}\\ \infty &\text{if $N=2,
3$}.
\end{cases}
$$
The term $f_0$, or $f_1$, is usually
classified as follows in the literature (see \cite{bociulasiecka2, bociulasiecka1}):
\renewcommand{\labelenumi}{{(\roman{enumi})}}
\begin{enumerate}
\item $f_0$ is {\em  subcritical} if $2\le p\le 1+\romega/2$, when the
Nemitskii operator $\widehat{f_0}$ associated to $f_0$ is locally
Lipschitz from $H^1(\Omega)$ into $L^2(\Omega)$;
\item $f_0$ is {\em supercritical} if $1+\romega/2<p\le \romega$, when  $\widehat{f_0}$ is no longer locally Lipschitz
from $H^1(\Omega)$ into $L^2(\Omega)$ but it still possesses a potential energy in
$H^1(\Omega)$;
\item $f_0$ is {\em super--supercritical} if $p>\romega$, when $\widehat{f_0}$ has no potentials in $H^1(\Omega)$.
\end{enumerate}
The  analogous classification is made for $g_0$,
depending on $q$ and $\rgamma$.

In \cite{Dresda1} the case when (I--II) hold and $f_0$ and $g_0$ are both subcritical  was studied, while here we are concerned with possibly supercritical and super--supercritical terms. To deal with them we need two further main assumptions:
\renewcommand{\labelenumi}{{(\Roman{enumi})}}
\begin{enumerate}
\setcounter{enumi}{2}
\item \label{III}
$\essinf_\Omega \alpha>0$ when $p>1+\romega/2$,
$\essinf_{\Gamma_1}\beta>0$ when $q>1+\rgamma/2$, and
 \begin{equation}\label{R}
2\le p\le 1+\romega/ \overline{m}', \quad 2\le q\le 1+\rgamma/
\overline{\mu}',
\end{equation}
where $\overline{m}=\max\{2,m\}$ and
$\overline{\mu}=\max\{2,\mu\}$;
\item \label{assumptionIV}if  $1+\romega/2<p=1+\romega/m'$ then $N\le 4$,
$f_0\in C^2(\R)$ and
$\varlimsup\limits_{|u|\to\infty} \frac {|f_0''(u)|}{|u|^{p-3}}<\infty$;
\\
 if $1+\rgamma/2<q=1+\rgamma/\mu'$ then $N\le 5$, $g_0\in C^2(\R)$ and
$\varlimsup\limits_{|u|\to\infty} \frac {|g_0''(u)|}{|u|^{q-3}}<\infty$.
\end{enumerate}

\begin{rem}
Clearly assumption (III) can be equally refereed to our model nonlinearities in \eqref{modelli}, while they satisfy (IV)
if
\begin{itemize}
\item[] $N\le 4$ and $\widetilde{p}>3$ or $\widetilde{\gamma}=0$ provided
$1+\romega/2<p=1+\romega/m'$;
\item[] $N\le 5$
and $\widetilde{q}>3$ or $\widetilde{\delta}=0$ provided
$1+\rgamma/2<q=1+\rgamma/\mu'$.
\end{itemize}
\end{rem}
\begin{rem}
\begin{figure}\label{Fig1}
\includegraphics[width=12truecm]{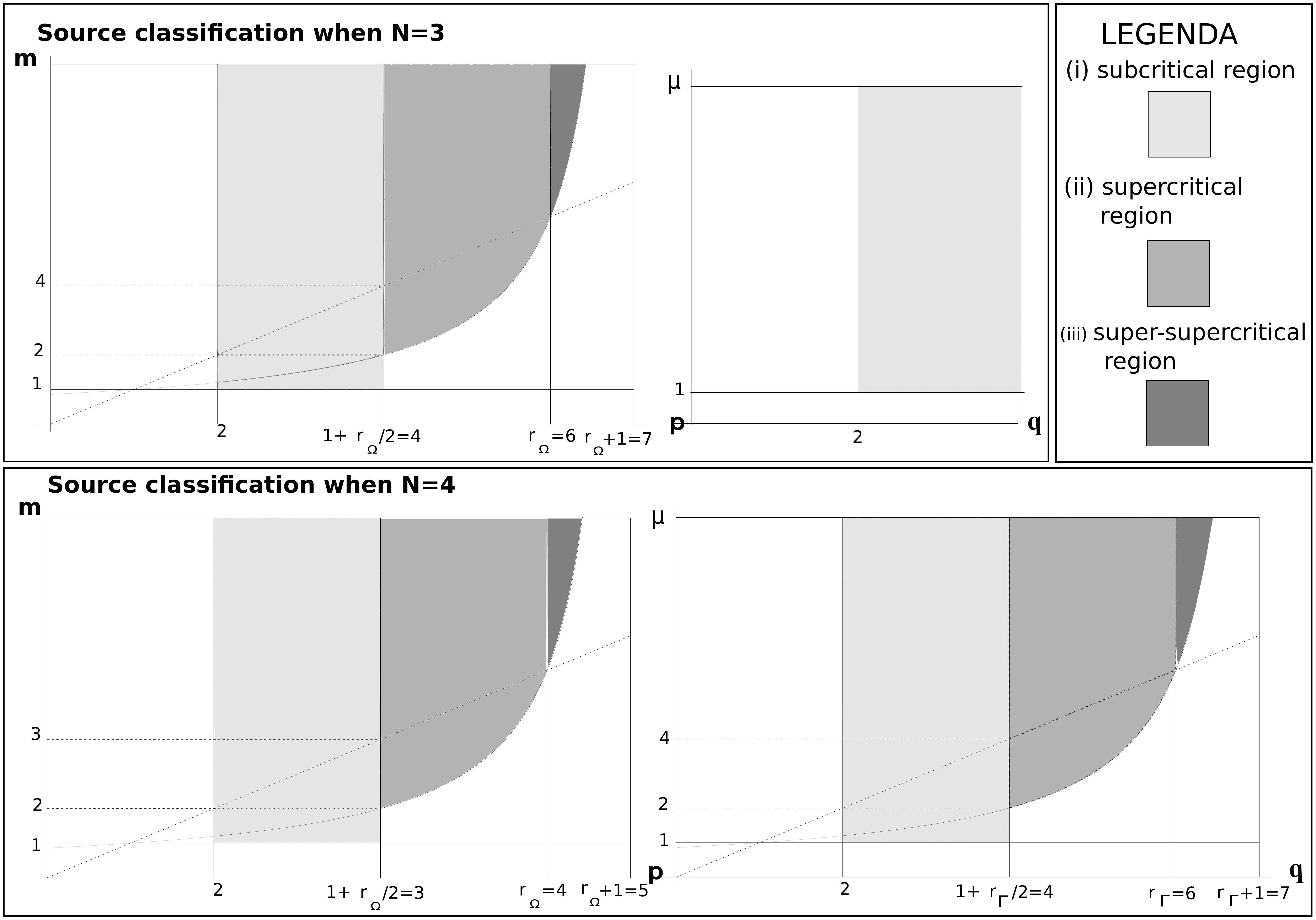}
\caption{The regions covered by \eqref{R} in dimensions $N=3,4$.}
\end{figure}
The sets in the planes $(p,m)$ and $(q,\mu)$, for which \eqref{R}
holds, corresponding to the classification above, are illustrated in
dimensions $N=3,4$ in Figure~1.
The subcritical regions were studied in \cite{Dresda1}, while this paper covers all regions in the picture.
In dimensions $N=3$ and $N=4$, the first one  being of
physical relevance, (IV) is a mild strengthening of (II). When $N=5$ assumption (IV) excludes parameter couples for which $1+\romega/2<p=1+\romega/m'$, while it is a strengthening of (II) for the term $g_0$. When $N\ge
6$ assumption (IV) simply excludes parameter couples for which $1+\romega/2<p=1+\romega/m'$ and $1+\rgamma/2<q=1+\rgamma/\mu'$.
\end{rem}
To present our results it is useful to subclassify supercritical terms $f_0$, by distinguishing between {\em intercritical} ones, when
$1+\romega/2<p<\romega$, and {\em Sobolev critical}  ones when
$p=\romega$.
Moreover terms $f_0$ for which $p<1+\romega/m'$ when $m>2$
will be treated in a different way from those for which
$p=1+\romega/m'$ when $m>2$, according to the compactness or non--compactness of  $\widehat{f_0}$  from $H^1(\Omega)$ to
$L^{m'(p-1)}(\Omega)$. In the latter case we shall say that $f_0$ {\em is on the hyperbola}. We shall also use the term {\em bicritical} for  Sobolev critical
terms on the hyperbola, when $p=m=\romega$. The same terminology will be adopted for $g_0$.
 \subsection{Add--on assumptions}
We now introduce some further assumptions which will be needed only in some results. In uniqueness and well--posedness results assumption (IV) will be
strengthened to the following one:
\renewcommand{\labelenumi}{{(\Roman{enumi})$'$}}
\begin{enumerate}
\setcounter{enumi}{3}
\item \label{assumptionIV'}if  $p>1+\romega/2$ then $N\le 4$,
$f_0\in C^2(\R)$ and
$\varlimsup\limits_{|u|\to\infty} \frac {|f_0''(u)|}{|u|^{p-3}}<\infty$;
\\
 if $q>1+\rgamma/2$ then $N\le 5$, $g_0\in C^2(\R)$ and
$\varlimsup\limits_{|u|\to\infty} \frac {|g_0''(u)|}{|u|^{q-3}}<\infty$.
\end{enumerate}
\begin{rem} The model nonlinearities in \eqref{modelli} satisfy assumption (IV)$'$
if
\begin{itemize}
\item[]$N\le 4$ and $\widetilde{p}>3$ or $\widetilde{\gamma}=0$ provided $p>1+\romega/2$;
\item[] $N\le 5$ and $\widetilde{q}>3$ or $\widetilde{\delta}=0$ provided $q>1+\rgamma/2$.
\end{itemize}
\end{rem}
\begin{rem}
In dimensions $N=3, 4$ assumption (IV)$'$ is a mild strengthening of
(II). When $N=5$ it says that $f_0$ is subcritical while it is a strengthening of (II) for $g_0$.
When $N\ge 6$ it simply says that $f_0$ and $g_0$ are both subcritical.
\end{rem}

When dealing with well--posedness, we shall restrict to
non--bicritical terms $f_0$, $g_0$ and  we shall use one
between the assumptions:
\renewcommand{\labelenumi}{{(\Roman{enumi})}}
\begin{enumerate}
\setcounter{enumi}{4}
\item  \label{assumptionV}if  $p\ge \romega$ then $\essliminf\limits_{|v|\to\infty}\frac {P_0'(v)}{|v|^{m-2}}>0$, if  $q\ge\rgamma $ then $\essliminf\limits_{|v|\to\infty}\frac {Q_0'(v)}{|v|^{\mu-2}}>0$;
\setcounter{enumi}{4}\renewcommand{\labelenumi}{{(\Roman{enumi})$'$}}
\item if  $m>\romega$ then $\essliminf\limits_{|v|\to\infty}\frac {P_0'(v)}{|v|^{m-2}}>0$, if  $\mu>\rgamma$ then $\essliminf\limits_{|v|\to\infty}\frac {Q_0'(v)}{|v|^{\mu-2}}>0$.
\end{enumerate}
\begin{rem}By \eqref{R}, we have $m>\romega$ when $p> \romega$ and
$\mu>\romega$ when $q> \romega$, the same implications being true
for weak inequalities, hence when sources are non--bicritical we
have $m>\romega$ when $p\ge \romega$ and $\mu>\romega$ when $q\ge
\romega$. Hence assumption (V)$'$ is  stronger than (V).
\end{rem}

\begin{rem} The model nonlinearities in \eqref{modelli} trivially satisfy both (V) and (V)$'$.\end{rem}
When looking for global solutions, as shown in \cite{Dresda1}, one has to restrict to perturbation terms which source part has at most linear growth at infinity or, roughly, it is dominated by the corresponding damping term.
Hence, denoting
\begin{equation}\label{23}
\mathfrak{F}_0(u)=\int_0^u f_0(s)\,ds,\qquad\mathfrak{G}_0(u)=\int_0^u g_0(s)\,ds\qquad\text{for all $u\in\R$,}
\end{equation}
we shall use the following specific global existence assumption:
\begin{enumerate}
\setcounter{enumi}5
\item there are $p_1$ and $q_1$ satisfying
\begin{equation}\label{V.2}2\le p_1\le \min\{p,\,\max\{2,m\}\}\quad\text{and}\quad  2\le q_1\le \min\{q,\,\max\{2,\mu\}\}
\end{equation}
and such that
\begin{equation}\label{V.1}
\varlimsup_{|u|\to\infty}\mathfrak{F}_0(u)/|u|^{p_1}<\infty\quad\text{and}\quad \varlimsup_{|u|\to\infty}\mathfrak{G}_0(u)/|u|^{q_1}<\infty.
\end{equation}
Moreover
\begin{equation}\label{V.3}
 \text{$\essinf_\Omega \alpha>0$ if $p_1>2$ and $\essinf_{\Gamma_1}\beta>0$ if $q_1>2$.}
\end{equation}
\end{enumerate}
Since $\mathfrak{F}_0(u)=\int_0^1 f_0(su)u\,ds$ (similarly for $\frak{G}_0$),
 (VI) is a weak version
 \begin{footnote}{ Actually (VI)$'$ is more general than (VI). Indeed, when
$f_0(u)=(m+1)|u|^{m-1}u\cos |u|^{m+1}$ and $g_0(u)=(\mu+1)|u|^{\mu-1}u\cos |u|^{\mu+1}$,
\eqref{Vstrong} holds only for $p_1\ge m+1$, $q_1\ge \mu+1$, while \eqref{V.1} does with $p_1=q_1=2$.
}\end{footnote}
 of the following assumption, which is adequate for most purposes and it is easier to verify:
\renewcommand{\labelenumi}{{(VI)$'$}}
\begin{enumerate}
\item there are $p_1$ and $q_1$ such that  \eqref{V.2} holds with \eqref{V.3} and
\begin{equation}\label{Vstrong}
\varlimsup_{|u|\to\infty}f_0(u)u/|u|^{p_1}<\infty\qquad\text{and}\quad \varlimsup_{|u|\to\infty}g_0(u)u/|u|^{q_1}<\infty.
\end{equation}
\end{enumerate}
\begin{rem}
Assumptions (II) and (VI)$'$ are satified by $f_0$ when it belongs to one among the following classes:
\renewcommand{\labelenumi}{{(\arabic{enumi})}}\begin{enumerate}
\setcounter{enumi}{-1}
\item $f_0$  is constant;
\item $f_0$ satisfies (II) with $p\le \max\{2,m\}$ and
$\essinf_\Omega \alpha>0$ if $p>2$;
\item $f_0$ satisfies (II) and $f_0(u)u\le 0$ in $\R$.
\end{enumerate}
The same remark applies to $g_0$, {\em mutatis mutandis}.
More generally (II) and (VI)$'$ hold when
\begin{equation}\label{centerdot}
f_0=f_0^0+f_0^1+f_0^2,\qquad g_0=g_0^0+g_0^1+g_0^2,
\end{equation}
and $f_0^i$ and $g_0^i$ belong to the class (i) for $i=0,1,2$.
\begin{footnote}{
Actually {\em all} functions verifying (II) and (VI)$'$ are of the form \eqref{centerdot},
where $f_0^1$ are $g_0^1$ are sources (that is $f_0^1(u)u\ge 0$ in $\R$). See Remark \ref{remark6.3ff}.}\end{footnote}
\end{rem}
\begin{rem}\label{remark1.5}
One easily checks that $f_1$ in \eqref{modelli} satisfies (II) and  (VI) if and only if  one among the following cases (the analogous ones applying to  $g_1$) occurs:
\renewcommand{\labelenumi}{{(\roman{enumi})}}
\begin{enumerate}
\item $\gamma>0$, $p\le \max\{2,m\}$ and $\essinf_\Omega\alpha >0$ if $p>2$;
\item $\gamma\le 0$, $\widetilde{\gamma}>0$,  $\widetilde{p}\le \max\{2,m\}$ and
$\essinf_\Omega \alpha>0$ if $\widetilde{p}>2$;
\item $\gamma,\widetilde{\gamma}\le0$.
\end{enumerate}
\end{rem}
\subsection{Function spaces and auxiliary exponents} We shall
identify $L^2(\Gamma_1)$ with its isometric image in $L^2(\Gamma)$,
that is
\begin{equation}\label{ID}
L^2(\Gamma_1)=\{u\in L^2(\Gamma): u=0\,\,\text{a.e. on
}\,\,\Gamma_0\}.
\end{equation}
We set, for $\rho\in [1,\infty)$, the Banach spaces
\begin{alignat*}2
&L^{2,\rho}_\alpha(\Omega)=\{u\in L^2(\Omega): \alpha^{1/\rho}u\in
L^\rho(\Omega)\}, \quad
&&\|\cdot\|_{2,\rho,\alpha}=\|\cdot\|_2+\|\alpha^{1/\rho}
\cdot\|_\rho,\\ &L^{2,\rho}_\beta(\Gamma_1)=\{u\in L^2(\Gamma_1):
\beta^{1/\rho}u\in L^\rho(\Gamma_1)\}, \quad
&&\|\cdot\|_{2,\rho,\beta,\Gamma_1}=\|\cdot\|_{2,\Gamma_1}+\|\beta^{1/\rho}
\cdot\|_{\rho,\Gamma_1},
\end{alignat*}
where
 $\|\cdot\|_\rho:=\|\cdot\|_{L^\rho(\Omega)}$ and
$\|\cdot\|_{\rho,\Gamma_1}:=\|\cdot\|_{L^\rho(\Gamma_1)}$. We denote by $u_{|\Gamma}$ the trace
on $\Gamma$ of $u\in H^1(\Omega)$. We introduce the Hilbert
spaces $H^0 = L^2(\Omega)\times L^2(\Gamma_1)$ and
\begin{equation}\label{H1}
H^1 = \{(u,v)\in H^1(\Omega)\times H^1(\Gamma): v=u_{|\Gamma}, v=0
\,\,\ \text{on $\Gamma_0$}\},
\end{equation}
with the norms inherited from the products. For the sake of
simplicity we shall identify, when useful, $H^1$ with its isomorphic
counterpart $\{u\in H^1(\Omega): u_{|\Gamma}\in H^1(\Gamma)\cap
L^2(\Gamma_1)\}$, through the identification $(u,u_{|\Gamma})\mapsto
u$, so we shall write, without further mention, $u\in H^1$ for
functions defined on $\Omega$. Moreover we shall drop the notation
$u_{|\Gamma}$, when useful, so we shall write $\|u\|_{2,\Gamma}$,
$\int_\Gamma u$, and so on, for $u\in H^1$. We also introduce, for
$\rho,\theta\in [1,\infty)$, the reflexive
spaces
\begin{footnote}{trivially it is possible to extend the definition of the spaces
$H^{1,\rho,\theta}_{\alpha,\beta}$ and $H^{1,\rho,\theta}$ also for $\rho,\theta=\infty$, and actually we shall do in Section \ref{section 2}, but we remark that in the statement of our results presented in this section these values \emph{are not allowed}.
}\end{footnote}
\begin{equation}\label{H1plus}
H^{1,\rho,\theta}_{\alpha,\beta}=H^1\cap
[L^{2,\rho}_\alpha(\Omega)\times L^{2,\theta}_\beta(\Gamma_1)],\,\,
H^{1,\rho,\theta}=H^{1,\rho,\theta}_{1,1}=H^1\cap
[L^\rho(\Omega)\times L^\theta(\Gamma_1)].
\end{equation}
\begin{rem}\label{remark1.1}
By assumption (III) we have
$H^{1,\rho,\theta}_{\alpha,\beta}=H^{1,\rho,\theta}$ when
\begin{equation}
\label{spaziuguali}
\rho\in\begin{cases}
[2,\romega],\quad&\text{if $p\le 1+\romega/2$,}\\
[2,\infty) ,\quad&\text{if $p>1+\romega/2$,}
\end{cases}\quad\text{and}\quad
\theta\in\begin{cases}
[2,\rgamma],\quad&\text{if $q\le 1+\rgamma/2$,}\\
[2,\infty) ,\quad&\text{if $q>1+\rgamma/2$.}
\end{cases}
\end{equation}
\end{rem}

We introduce the auxiliary exponents $\somega=\somega(p,N)$ and $\sgamma=\sgamma(q,N)$ by
\begin{equation}\label{somegagamma}
\somega=\max\left\{2,
\romega(p-2)/(\romega-2)\right\},\quad
\sgamma=\max\left\{2,\rgamma(p-2)/(\rgamma-2)\right\},
\end{equation}
extended by continuity when $\romega,\rgamma=\infty$.
Trivially one has
\begin{alignat}2\label{somegagammarelazioni1} &p\lesseqgtr
\romega \quad\Leftrightarrow\quad \somega\lesseqgtr
\romega,\qquad&& q\lesseqgtr \rgamma
\quad\Leftrightarrow\quad \sgamma\lesseqgtr \rgamma;\\
\label{ellis234}
&p\ge\romega \quad\Rightarrow\quad p\le \somega,\qquad&& q\ge\rgamma \quad\Rightarrow\quad q\le
\sgamma.
\end{alignat}
Moreover, when (IV)$'$ holds, so $\romega\ge 4$ when $p\ge
1+\romega/2$ and $\rgamma\ge 4$ when $q\ge 1+\rgamma/2$, by
\eqref{R} we have
\begin{footnote}{indeed when $p>\romega$ by \eqref{R} we have $m>\romega$ and, since $\romega\ge 4$ we get
$(\romega-2)m^2+\romega(1-\romega)m+\romega^2>0$ or equivalently
$m>\frac \romega{\romega-2}\left(\frac\romega{m'}-1\right)$ so using
\eqref{R} again $m>\somega$. The calculation can be repeated with
weak inequalities and with $p,m,\romega$ replaced by
$q,\mu\,\rgamma$. }\end{footnote}
\begin{equation}\label{somegagammarelazioni2}
p>\romega \quad\Rightarrow\quad \romega< \somega<
m,\qquad\text{and}\quad q>\rgamma \quad\Rightarrow\quad \rgamma<
\sgamma< \mu,
\end{equation}
the same implications being true with weak inequalities. We also set
\begin{gather}\label{sigmaomegagamma}
\sigmaomega=\begin{cases}
\somega &\text{if $\romega<p=1+\romega/m'$},\\
2 & \text{otherwise},
\end{cases}
\qquad\sigmagamma=\begin{cases}
\sgamma &\text{if $\rgamma<q=1+\rgamma/\mu'$},\\
2 & \text{otherwise,}
\end{cases}\\
\label{lomegagamma}
\lomega=\begin{cases}
\somega &\text{if $\romega<p=1+\romega/m'$},\\
p &\text{if $\romega<p<1+\romega/m'$},\\
2 &\text{if $p\le\romega$},
\end{cases}\qquad\lgamma=\begin{cases}
\sgamma &\text{if $\rgamma<q=1+\rgamma/\mu'$},\\
q &\text{if $\rgamma<q<1+\rgamma/\mu'$},\\
2 &\text{if $q\le\rgamma$},
\end{cases}
\end{gather}
so, by \eqref{H1plus}, \eqref{somegagamma}, \eqref{ellis234}, \eqref{sigmaomegagamma}--\eqref{lomegagamma}  and assumption (II),
\begin{equation}\label{lsomegagammaMINORE}
2\le \sigmaomega\le \lomega\le\somega,\qquad 2\le
\sigmagamma\le \lgamma\le\sgamma,
\end{equation}
and $$H^{1,\somega,\sgamma}\hookrightarrow H^{1,\lomega,\lgamma}\hookrightarrow
H^{1,\sigmaomega,\sigmagamma}\hookrightarrow H^1.$$
When (IV)$'$ holds, by \eqref{somegagammarelazioni1}, \eqref{somegagammarelazioni2} and \eqref{lsomegagammaMINORE} we have
\begin{equation}\label{lsomegagamma}
2\le \sigmaomega\le \lomega\le\somega\le\max\{\romega,m\},\quad 2\le
\sigmagamma\le \lgamma\le\sgamma\le \max\{\rgamma,\mu\}.
\end{equation}
\subsection{Local analysis} Our first main result is the following one.
\begin{thm}[\bf Local existence and continuation]  \label{theorem1.1}
Let (I--IV) hold.  Then:
\renewcommand{\labelenumi}{{(\roman{enumi})}}
\begin{enumerate}
\item
for any $(u_0,u_1)\in H^{1,\sigmaomega,\sigmagamma}\times H^0$
\eqref{1.2} has a maximal weak solution, that is
\begin{equation}\label{1.14}
\begin{aligned}
&u=(u,u_{|\Gamma})\in
L^\infty_{\text{loc}}([0,T_{\text{max}});H^1)\cap
W^{1,\infty}_{\text{loc}}([0,T_{\text{max}});H^0), \\
&u'=(u_t,{u_{|\Gamma}}_t)\in
L^m_{\text{loc}}([0,T_{\text{max}});L^{2,m}_\alpha(\Omega))\times
L^\mu_{\text{loc}}([0,T_{\text{max}});L^{2,\mu}_\beta(\Gamma_1)),
\end{aligned}\end{equation}
which satisfies \eqref{1.2} in a distribution sense (to be specified
later);
\item $u$ enjoys the regularity
\begin{equation}\label{1.15}
U=(u,u')\in C([0,T_{\text{max}});H^{1,\sigmaomega,\sigmagamma}\times
H^0)
\end{equation}
and satisfies, for $0\le s\le t<T_{\text{max}}$, the energy identity
\begin{footnote}{here $\nabla_\Gamma$
denotes the Riemannian gradient on $\Gamma$ and $|\cdot|_\Gamma$,
the norm associated to the Riemannian scalar product on the tangent
bundle of $\Gamma$. See Section \ref{section 2}.}\end{footnote}
\begin{equation}\label{energyidentityreduced}
\begin{split}
\tfrac 12 \left.\int_\Omega u_t^2\! + \tfrac 12\int_{\Gamma_1}
{u_{|\Gamma}}_t^2 +\tfrac 12 \int_\Omega |\nabla u|^2 +\tfrac 12
\int_{\Gamma_1} |\nabla_\Gamma  u|_\Gamma ^2\right|_s^t\!\! \!+
\!\int_s^t\!\!\!\int_\Omega \alpha
P_0(u_t)u_t\\+\int_s^t\int_{\Gamma_1} \beta
Q_0({u_{|\Gamma}}_t){u_{|\Gamma}}_t=\int_s^t\int_{\Omega}
f_0(u)u_t+\int_s^t\int_{\Gamma_1} g_0(u){u_{|\Gamma}}_t;
\end{split}
\end{equation}
\item
$\varlimsup_{t\to T^-_{\text{max}}}\|U(t)\|_{H^1\times H^0}=\infty$
provided  $T_{\text{max}}<\infty$.
\end{enumerate}
Finally the property (ii) is enjoyed by any weak maximal solution of
\eqref{1.2}.
\end{thm}

Theorem~\ref{theorem1.1} sharply extends
 the local existence
statement in \cite[Theorem~1.1]{Dresda1}
 with respect to internal sources/sink when $N=3$ and to  internal and
 boundary sources/sink when $N\ge 4$ (see Figure~1). Moreover
 it sharply generalizes all local existence results in the broader literature concerning
wave equation with internal/boundary damping and source terms (see \cite{bociulasiecka1,bociulasiecka2, STV}).

The proof of Theorem~\ref{theorem1.1} presented in the sequel is
based on \cite[Theorem~1.1]{Dresda1} together with a
truncation procedure inspired from \cite{bociulasiecka1}. Moreover
we use a combination of a compactness argument from \cite{STV}, when
$(p,m)$ and $(q,\mu)$ are not on the hyperbola, \begin{footnote}{in this case
 an alternative approach, only using compactness as in
\cite{STV}, is possible.}\end{footnote}with a key estimate, somewhat
simplified and extended to higher dimension, from \cite{bociulasiecka2} (see Lemmas~\ref{lemma4.1}
and \ref{lemma4.4} below) when they are on the hyperbola.

The generality of Theorem~\ref{theorem1.1} is best illustrated by
some of its corollaries, the first of which concerns data in
$H^1\times H^0$ and involves minimal assumptions on the
nonlinearities and no restrictions on $N$ but excludes source/sink terms on
the hyperbola, in the spirit of \cite{STV}.
\begin{cor} \label{corollary1.1}Under assumptions (I--III) and
\begin{equation}\label{R'}
p<1+\romega/m'\quad\text{when $m>2$,}\qquad
q<1+\rgamma/\mu'\quad\text{when $\mu>2$,}
\end{equation}
for any $(u_0,u_1)\in H^1\times H^0$
problem \eqref{1.2} has a maximal weak solution.
\end{cor}
By excluding only super-supercritical sources  on the hyperbola but having restrictions on $N$ on the supercritical part of it we get
\begin{cor}\label{corollary1.2} If assumptions (I--IV) are satisfied and
\begin{equation}\label{R''bis}
p<1+\romega/m'\quad\text{when $m>\romega$,}\qquad
q<1+\rgamma/\mu'\quad\text{when $\mu>\rgamma$,}
\end{equation}
for any $(u_0,u_1)\in H^1\times H^0$
problem \eqref{1.2} has a maximal weak solution.

In particular the same conclusion holds true when
\begin{equation}\label{subSobolevcritical}
2\le p\le \romega\qquad\text{and}\quad 2\le
q\le \rgamma.
\end{equation}
\end{cor}

Theorem~\ref{theorem1.1} and Corollaries~\ref{corollary1.1}--
\ref{corollary1.2}  are stated under  minimal
regularity (or, more properly, integrability) assumptions on $u_0$. When $u_0$
is more regular solutions are more regular, as one sees by a trivial
time-integration, using \eqref{1.14}--\eqref{1.15} and
Remark~\ref{remark1.1}. We explicitly state this remark since it
will be crucial in the sequel.
\begin{cor} \label{corollary1.3} Under assumptions (I--IV) the
conclusions of Theorem~\ref{theorem1.1} hold when
$H^{1,\sigmaomega,\sigmagamma}$ is replaced by
$H^{1,\rho,\theta}_{\alpha,\beta}$, provided $\rho,\theta\in\R$ satisfy
\begin{equation}\label{ranges1}
(\rho,\theta)\in [\sigmaomega,\max\{\romega,m\}]\times
[\sigmagamma,\max\{\rgamma,\mu\}],
\end{equation}
and by $H^{1,\rho,\theta}$, provided $\rho$ and $\theta$ satisfy \eqref{spaziuguali} and \eqref{ranges1}.
 In particular, by
\eqref{lsomegagamma}, it can be
replaced by $H^{1,\somega,\sgamma}$ when also  (IV)$'$ holds.
\end{cor}
Or second main result asserts also uniqueness when $(u_0,u_1)\in
H^{1,\somega,\sgamma}\times H^0$.
\begin{thm}[\bf Local existence--uniqueness and continuation]
\label{theorem1.2} Under assumptions (I--III) and (IV)$'$  the
following conclusions hold:
\renewcommand{\labelenumi}{{(\roman{enumi})}}
\begin{enumerate}
\item
for any $(u_0,u_1)\in H^{1,\somega,\sgamma}\times H^0$  problem
\eqref{1.2} has a unique  maximal weak solution $u$ in
$[0,T_{\text{max}})$;
\item $u$ enjoys the regularity
\begin{equation}\label{1.21}
U=(u,u')\in C([0,T_{\text{max}});H^{1,\somega,\sgamma}\times H^0)
\end{equation}
and satisfies, for $0\le s\le t<T_{\text{max}}$, the energy identity
\eqref{energyidentityreduced};
\item
if  $T_{\text{max}}<\infty$ then
\begin{equation}\label{specialblowup}
\varlimsup_{t\to T^-_{\text{max}}}\|U(t)\|_{H^1\times
H^0}=\lim_{t\to
T^-_{\text{max}}}\|U(t)\|_{H^{1,\somega,\sgamma}\times H^0}=\infty.
\end{equation}
\end{enumerate}
In particular, when $2\le p\le\romega$ and $2\le q\le \rgamma$, we have
 $H^{1,\somega,\sgamma}=H^1$.
\end{thm}
The proof of Theorem~\ref{theorem1.2} is based on the key estimate
recalled above and on standard arguments.  When $u_0$ is more regular, as in Corollary~\ref{corollary1.3}, we have
\begin{cor} \label{corollary1.4} Under assumptions (I--III) and (IV)$'$ the main
conclusions of Theorem~\ref{theorem1.2} hold when the space
$H^{1,\somega,\sgamma}$ is replaced by
$H^{1,\rho,\theta}_{\alpha,\beta}$, provided  $\rho,\theta\in\R$ satisfy
\begin{equation}\label{ranges2}
(\rho,\theta)\in [\somega,\max\{\romega,m\}]\times
[\sgamma,\max\{\rgamma,\mu\}],
\end{equation}
and by $H^{1,\rho,\theta}$, provided $\rho$ and $\theta$ satisfy \eqref{spaziuguali} and \eqref{ranges2}.
\end{cor}
The sequel of our local analysis concerns
local Hadamard well--posedness. Unfortunately it is possible to prove this type of result in the same space used in Theorem~\ref{theorem1.2} {\em only} when
sources are subcritical or intercritical, in this case being
$H^{1,\somega,\sgamma}=H^1.$

When $p\ge \romega$ we have to restrict
to $u_0\in L^{s_1}(\Omega)$ with $s_1\in (\somega,m]$, while when
$q\ge \rgamma$ to ${u_0}_{|\Gamma}\in L^{s_2}(\Gamma_1)$ with
$s_2\in (\sgamma,\mu]$. Since $(\somega,m]=\emptyset$ when
$p=m=\romega$ and $(\sgamma,\mu]=\emptyset$ when $q=\mu=\rgamma$,  we have to exclude these cases by assuming
\begin{equation}\label{R''}
(p,m)\not=(\romega,\romega)\qquad\text{and}\quad
(q,\mu)\not=(\rgamma,\rgamma).
\end{equation}
Consequently we shall consider $u_0\in H^{1,s_1,s_2}\hookrightarrow
H^{1,\somega,\sgamma}$  where
\begin{footnote}{by \eqref{somegagammarelazioni1} conditions $s_1>\somega$ when $p<\romega$  and $s_2>\sgamma$ when $q<\rgamma$ can be trivially skipped.}\end{footnote}
\begin{equation}\label{ranges2ter}
s_1\in
\begin{cases}
(\somega,\romega]&\text{if $p<\romega$},\\
(\somega,\max\{\romega,m\}] &\text{otherwise,}
\end{cases}\quad
s_2\in\begin{cases}
(\sgamma,\rgamma] &\text{if $q<\rgamma$},\\
(\sgamma,\max\{\rgamma,\mu\}]\, &\text{otherwise.}
\end{cases}
\end{equation}
Trivially, by \eqref{somegagammarelazioni1} and
\eqref{somegagammarelazioni2}, when \eqref{R''} holds  we have the implications
\begin{equation}\label{somegagammarelazioni3}
p\geq\romega \quad\Rightarrow\quad \romega\le \somega<
m,\qquad\text{and}\quad q\ge\rgamma \quad\Rightarrow\quad \rgamma\le
\sgamma< \mu.
\end{equation}
After these preliminary considerations we can state
\begin{thm}[\bf Local Hadamard well--posedness I]  \label{theorem1.3}
Let assumptions (I--III), (IV)$'$, (V) and \eqref{R''} hold. Then problem \eqref{1.2} is
locally well--posed in $H^{1,s_1,s_2}\times H^0$ for
$s_1$ and $s_2$ satisfying \eqref{ranges2ter}.

More explicitly,   given
$(u_{0n},u_{1n})\to (u_0,u_1)$ in $H^{1,s_1,s_2}\times H^0$,
respectively denoting by $u^n$ and $u$ the unique weak maximal solution of
\eqref{1.2} in $[0,T^n_{\text{max}})$ and $[0,T_{\text{max}})$
corresponding to initial data $(u_{0n},u_{1n})$ and $(u_0,u_1)$, which exist by Theorem~\ref{theorem1.2},
$U^n=(u^n,\dot{u}^n)$ and $U=(u,\dot{u})$,
 the following conclusions hold:
\renewcommand{\labelenumi}{{(\roman{enumi})}}
\begin{enumerate}
\item $T_{\text{max}}\le \varliminf_n T^n_{\text{max}}$ and
\item $U^n\to U$ in $C([0,T^*];H^{1,s_1,s_2}\times H^0)$ for any
$T^*\in (0,T_{\text{max}})$.
\end{enumerate}
In particular, when $2\le p<\romega$ and $2\le q< \rgamma$, since
(V) has empty content, problem \eqref{1.2} is locally well--posed in
$H^1\times H^0$ under assumptions (I--III), (IV)$'$.
\end{thm}
Theorem~\ref{theorem1.3} covers the supercritical ranges exactly as in
\cite{bociulasiecka1}. Moreover it is the {\em first well--posedness result}, in the author's
knowledge, dealing with internal/boundary super--supercritical term $f_0$, $g_0$. The price paid for this generality
is to work in a Banach space smaller than the natural Hilbert energy space $H^1\times H^0$ and assumption (V).  The proof of Theorem~\ref{theorem1.3} is based on the key estimate recalled
above.

Theorem~\ref{theorem1.3} is aimed to get well--posedness in the largest possible
space, which turns out to be as close as one
likes to $H^{1,\somega,\sgamma}\times H^0$, not including it is some cases. The aim
of the following variant of
Theorem~\ref{theorem1.3}, is to complete the picture made in Corollaries~\ref{corollary1.3}--\ref{corollary1.4} by showing in
which part of the scale of spaces introduced there
 the problem is locally well-posed.
\begin{thm}[\bf Local Hadamard well--posedness II]  \label{theorem1.4}
Let assumptions (I--III), (IV)$'$, (V)$'$, \eqref{R''} hold and  $\rho,\theta\in\R$ satisfy
\begin{equation}\label{ranges3}
(\rho,\theta)\in (\somega,\max\{\romega,m\}]\times
(\sgamma,\max\{\rgamma,\mu\}].
\end{equation}
Then problem \eqref{1.2} is locally well--posed in
$H^{1,\rho,\theta}_{\alpha,\beta}\times H^0$, that is the conclusions
of Theorem~\ref{theorem1.3} hold true, when $H^{1,s_1,s_2}$ is replaced
by $H^{1,\rho,\theta}_{\alpha,\beta}$.
In particular  it is locally--well posed  in
$H^{1,\rho,\theta}\times H^0$ when $\rho,\theta$ satisfy
\eqref{spaziuguali} and \eqref{ranges3}.
\end{thm}

\subsection{Global analysis}
As a main application of the local analysis presented above we now state the global--in--time versions of Theorems~\ref{theorem1.1}--\ref{theorem1.4}
and their corollaries.
When $P_0(v)=Q_0(v)=$,  $f_0(u)=|u|^{p-2}u$ and $g_0(u)=|u|^{q-2}u$, $p,q>2$,  solutions of \eqref{1.2} blow--up in finite time
for suitably chosen initial data, as proved in \cite[Theorem~1.5]{Dresda1}.
Hence in the sequel we shall restrict to terms $f_0$ and $g_0$ satisfying assumption (VI) presented above, which excludes
this case.
\begin{thm}[\bf Global analysis]  \label{theorem1.5}
The following conclusions hold true.
\renewcommand{\labelenumi}{{(\roman{enumi})}}
\begin{enumerate}
\item {\bf (Global existence)} Under assumptions  (I--IV) and (VI), for any $(u_0,u_1)\in H^{1,\lomega,\lgamma}\times H^0$ the weak maximal solution $u$ of problem \eqref{1.2}
found in Theorem~\ref{theorem1.1} is global in time, that is $T_{\text{max}}=\infty$, and $u\in C([0,T_{\text{max}});H^{1,\lomega,\lgamma})$.

In particular, when \eqref{subSobolevcritical} holds, for any $(u_0,u_1)\in H^1\times H^0$ problem \eqref{1.2} has a global weak solution.

\item {\bf (Global existence--uniqueness)} Under assumptions (I--III), (IV)$'$ and (VI), for any $(u_0,u_1)\in H^{1,\somega,\sgamma}\times H^0$ the
unique maximal solution of problem \eqref{1.2} found in  Theorem~\ref{theorem1.2} is global in time, that is
$T_{\text{max}}=\infty$, and $u\in C([0,\infty);H^{1,\somega,\sgamma})$.

In particular, when \eqref{subSobolevcritical} holds, for any $(u_0,u_1)\in H^1\times H^0$ problem \eqref{1.2} has a unique global weak solution.

\item {\bf (Global Hadamard well--posedness)}  Under assumptions (I--III), (IV)$'$, (V--VI)
and \eqref{R''} problem \eqref{1.2} is  globally well--posed in $H^{1,s_1,s_2}\times H^0$ for
$s_1$ and $s_2$ satisfying \eqref{ranges2ter}, that is $T_{\text{max}}=\infty$ in Theorem~\ref{theorem1.3}.

Consequently the semi--flow generated by problem \eqref{1.2} is a dynamical system in $H^{1,s_1,s_2}\times H^0$.

In particular, when $2\le p<\romega$ and $2\le q<\rgamma$ and under assumptions (I--III), (IV)$'$ and (VI), problem \eqref{1.2} is globally well--posed in $H^1\times H^0$, so the semi--flow generated by \eqref{1.2} is a dynamical system in $H^1\times H^0$.
\end{enumerate}
\end{thm}

\begin{rem} Parts (ii) and (iii) of Theorem~\ref{theorem1.5}  simply follow by combining
 Theorem~\ref{theorem1.5}--(i) with, respectively, Theorem~\ref{theorem1.2} and Theorem~\ref{theorem1.3}.
 We include them in Theorem~\ref{theorem1.5} in order to  illustrate in parallel   the global--in--time counterparts of Theorems~\ref{theorem1.1}--\ref{theorem1.3}.
 \end{rem}

By excluding source/sink terms on the hyperbola, so having no restriction on $N$, we get the
global--in--time counterpart of Corollary~\ref{corollary1.1}.
\begin{cor}\label{corollary1.5}
 Under assumptions (I--III), (VI) and \eqref{R'} for any $(u_0,u_1)\in H^{1,p.q}\times H^0$
problem \eqref{1.2} has a global weak solution $u\in C([0,\infty);H^{1,p,q})$.
\end{cor}
By excluding only super-supercritical sources  on the hyperbola but having restrictions on $N$ on the supercritical part of it we get the
global--in--time counterpart of Corollary~\ref{corollary1.2}.
\begin{cor}\label{corollary1.6}
Under assumptions (I--IV), (VI) and \eqref{R''bis}
for any $(u_0,u_1)\in H^{1,p,q}\times H^0$
problem \eqref{1.2} has a global weak solution $u\in C([0,\infty);H^{1,p,q})$.
\end{cor}
The main difference between
Corollaries~\ref{corollary1.1} and  \ref{corollary1.5} (and between
Corollaries~\ref{corollary1.2} and  \ref{corollary1.6}) is that  the former concerns all data $u_0\in H^1$ while the latter restricts to  $u_0\in H^{1,p,q}$.
This restriction originates from the use maid, in the proof of Theorem~\ref{theorem1.5}--(i), of the potentials of $\widehat{f_0}$ and $\widehat{g_0}$. Clearly they are simultaneously defined only in $H^{1,p,q}$.
For the same reason Theorem~\ref{theorem1.1}  concerns  $u_0\in H^{1,\sigmaomega,\sigmagamma}$ while  Theorem~\ref{theorem1.5}--(i) concerns $u_0\in H^{1,\lomega,\lgamma}=H^{1,\sigmaomega,\sigmagamma}\cap H^{1,p,q}$.
Since, by \eqref{ellis234}, $H^{1,\somega,\sgamma}\subset H^{1,p,q}$, this restriction do not effects Theorem~\ref{theorem1.5}--(ii--iii).

We now state, for the reader convenience, the global--in--time version of the more general local analysis made in Corollaries~\ref{corollary1.3}--\ref{corollary1.4} and Theorem~\ref{theorem1.4}, simply obtained
by combining them with Theorem~\ref{theorem1.5}.

\begin{cor}[\bf Global analysis in the scale of spaces]  \label{corollary1.7}
The following conclusions hold true.
\renewcommand{\labelenumi}{{(\roman{enumi})}}
\begin{enumerate}
\item {\bf (Global existence)} Under assumptions  (I--IV) and (VI) the
main conclusion of Theorem~\ref{theorem1.5}--(i) hold when
$H^{1,\lomega,\lgamma}$ is replaced by
$H^{1,\rho,\theta}_{\alpha,\beta}$, provided $\rho,\theta\in\R$ satisfy
\begin{equation}\label{ranges4}
(\rho,\theta)\in [\lomega,\max\{\romega,m\}]\times
[\lgamma,\max\{\rgamma,\mu\}],
\end{equation}
and by $H^{1,\rho,\theta}$, provided $\rho$ and $\theta$ satisfy \eqref{spaziuguali} and \eqref{ranges4}.
\item {\bf (Global existence--uniqueness)} Under assumptions (I--III), (IV)$'$ and (VI) the
main conclusion of Theorem~\ref{theorem1.5}--(ii) holds when the space
$H^{1,\somega,\sgamma}$ is replaced by
$H^{1,\rho,\theta}_{\alpha,\beta}$, provided $\rho,\theta\in\R$ satisfy \eqref{ranges2}
and by $H^{1,\rho,\theta}$, provided $\rho$ and $\theta$ satisfy \eqref{spaziuguali} and \eqref{ranges2}.
\item {\bf (Global Hadamard well--posedness)}  Under assumptions (I--III), (IV)$'$, (V--VI)
and \eqref{R''} problem \eqref{1.2} is locally well--posed in
$H^{1,\rho,\theta}_{\alpha,\beta}\times H^0$ when $\rho,\theta\in\R$ satisfy \eqref{ranges3}, that is the conclusions
of Theorem~\ref{theorem1.5}--(iii) hold true when $H^{1,s_1,s_2}$ is replaced
by $H^{1,\rho,\theta}_{\alpha,\beta}$.
In particular  it is locally--well posed  in
$H^{1,\rho,\theta}\times H^0$ when $\rho,\theta$ satisfy
\eqref{spaziuguali} and \eqref{ranges3}.
\end{enumerate}
\end{cor}
Theorems~\ref{theorem1.1}--\ref{theorem1.5}, with their corollaries,
can be easily extended to more general second order uniformly
elliptic linear operators, both in $\Omega$ and $\Gamma$, under
suitable regularity assumptions on the coefficients. Here we prefer
to deal with the Laplace and Laplace--Beltrami operators for the
sake of clearness.
\subsection{Overall conclusions and paper organization.}
\label{subsection1.7}
The presentation of Theorems~\ref{theorem1.1}--\ref{theorem1.3} and \ref{theorem1.5}, dealing with results for data in the maximal space,
can be simplified and unified by specifing $N$ and slightly strengthening our assumptions set, that is by assuming assumptions (I--II) and
the following ones:
\renewcommand{\labelenumi}{\bf\arabic{enumi})}
\begin{enumerate}
\item when $\essinf_\Omega \alpha>0$ and $\essinf_{\Gamma_1} \beta>0$, the following properties are satisfied:
\renewcommand{\labelenumii}{{($\cal{P}_\arabic{enumii}$)}}
\begin{enumerate}
\item $f_0\in C^2(\R)$  and $\varlimsup_{|u|\to\infty} |f_0''(u)|/|u|^{p-3}<\infty$ when $p>1+\romega/2$;
\item $g_0\in C^2(\R)$  and $\varlimsup_{|u|\to\infty} |g_0''(u)|/|u|^{q-3}<\infty$ when $q>1+\rgamma/2$;
\item $\essliminf_{|v|\to\infty}P_0'(v)/|v|^{m-2}>0$ if $p\ge \romega$;
\item $\essliminf_{|v|\to\infty}Q_0'(v)/|v|^{\mu-2}>0$ if $q\ge \rgamma$.
\end{enumerate}
\item when $\essinf_\Omega \alpha>0=\essinf_{\Gamma_1} \beta$,  properties ($\cal{P}_1$) and ($\cal{P}_3$) are satisfied;
\item when $\essinf_\Omega \alpha=0<\essinf_{\Gamma_1} \beta$,  properties ($\cal{P}_2$) and ($\cal{P}_4$) are satisfied;
\item when $\essinf_\Omega \alpha=\essinf_{\Gamma_1} \beta=0$ no further properties are requested.
\end{enumerate}
Theorems~\ref{theorem1.1}--\ref{theorem1.3} and \ref{theorem1.5} are summarized  in Tables~1--4 in the sequel, respectively dealing with the cases $N=3$, $N=4$, $N=5$ and $N\ge 6$,
to be read according to the following conventions:
\begin{enumerate}
\item[-] boxes separated by continuous lines indicate different cases depending on $p,q,m$ and $\mu$,
while dashed lines separate different results in the same case;
\item[-] depending on $\alpha$ and $\beta$, the following parts of  Tables~1--4 apply:
\renewcommand{\labelenumii}{\bf\arabic{enumii})}
\begin{enumerate}
\item when $\essinf_\Omega \alpha>0$ and $\essinf_{\Gamma_1} \beta>0$ all the tables;
\item when $\essinf_\Omega \alpha>0=\essinf_{\Gamma_1} \beta$ only the first column;
\item when $\essinf_\Omega \alpha=0<\essinf_{\Gamma_1} \beta$ only the first row;
\item when $\essinf_\Omega \alpha=\essinf_{\Gamma_1} \beta=0$ only the first box of the first row;
\end{enumerate}
\item[-] by existence, existence--uniqueness and well--posedness in the Banach space $V$ we indicate the corresponding result among
Theorems~\ref{theorem1.1}--\ref{theorem1.3} and the corresponding one among parts (i-iii) of Theorem~\ref{theorem1.5} which applies for
all $(u_0,u_1)\in V\times H^0$, the latter only under the additional assumption (VI); when the space $V$ is different for
Theorem~\ref{theorem1.1} and part (i) of Theorem~\ref{theorem1.5},  we shall call the former local existence and the latter global existence;
\item[-] since well--posedness yields existence--uniqueness, which in turn yields existence, when two or three results hold
in the same space only the strongest result is explicitly written;
\item[-] $\varepsilon$ denotes a  sufficiently small positive number.
\end{enumerate}
The case $N=2$ is omitted, since in this case, without additional assumptions,
we simply get well--posedness in  $H^1$ for all $p,q\in [2,\infty)$.
\begin{table}
\begin{center}
\setlength{\tabcolsep}{28pt}
\renewcommand{\arraystretch}{1.3}
\setlength{\arrayrulewidth}{0.7pt}
\begin{tabular}{|c|c|}
\multicolumn{2}{>{\normalsize}c}{\sc Table 1. \rm Main results when $N=3$}\\
\hline
                     &$2\le q<\infty$ \\
\hline
$2\le p\le 4$        &   \\
\cline{1-1}
$4<p<6$        &  well--posedness in $H^1$ \\
\hline
$6=p<m$           &  existence--uniqueness in $H^1$ \\
\cdashline{2-2}[1pt/1pt]
                  &  well--posedness  in $H^{1,6+\eps,2}$\\
\hline
$6=p=m$           &  existence--uniqueness in $H^1$\\
\hline
$6<p<1+6/m'$      &  local existence in $H^1$ \\
\cdashline{2-2}[1pt/1pt]
                  &  global existence in $H^{1,p,2}$ \\
\cdashline{2-2}[1pt/1pt]
                  &  existence--uniqueness in $H^{1,3(p-2)/2,2}$\\
\cdashline{2-2}[1pt/1pt]
                  &  well--posedness  in $H^{1,3(p-2)/2+\eps,2}$\\
\hline
$6<p=1+6/m'$      & existence--uniqueness in $H^{1,3(p-2)/2,2}$ \\
\cdashline{2-2}[1pt/1pt]
                  &  well--posedness  in $H^{1,3(p-2)/2+\eps,2}$\\
\hline
\end{tabular}

\bigskip

\setlength{\tabcolsep}{0pt}
\renewcommand{\arraystretch}{0.5}
\setlength{\arrayrulewidth}{0.7pt}
\begin{tabular}{|>{\tiny}c|>{\tiny}c|>{\tiny}c|>{\tiny}c|>{\tiny}c|>{\tiny}c|}
\multicolumn{6}{>{\normalsize}c}{\sc Table 2. \rm Main results when $N=4$}\\
\multicolumn{6}{>{\normalsize}c}{}\\
\hline
                                 &\raisebox{1.5ex}[12pt]{\tiny{$2\le q\le 4$}}\vline\raisebox{1.5ex}[12pt]{\tiny{$4<q<6$}}   &  \raisebox{1.5ex}[12pt]{\tiny{$6=q<\mu$}}                            & \raisebox{1.5ex}[12pt]{\tiny{$6=q=\mu$}}                 & \raisebox{1.5ex}[12pt]{\tiny{$6<q<1+ 6/\mu'$}}                                  & \raisebox{1.5ex}[12pt]{\tiny{$6<q=1+6/\mu'$}}            \\
\hline
                                 &                                                  &                                                    &                                  &local  existence                                          &                                               \\
                                 &                                                  &                                                    &                                  &in  $H^1$                                                 &                                               \\
\cdashline{5-5}[1pt/1pt]
$2\le p\le 3$                    &                                                  &                                                    &                                  &global existence                                          &                                               \\
                                 &                                                  &                                                    &                                  & in $H^{1,2,q}$                                           &                                               \\
\cline{1-1}\cdashline{5-5}[1pt/1pt]
                                 &                                                  &existence--                                         &                                  &existence--                                               &existence--                                    \\
                                 &                                                  &uniqueness in                                       &                                  &uniqueness in                                             &uniqueness in                                      \\
$3< p<4$                         &                                                  &$H^1$                                               &                                  &$H^{1,2,3(q-2)/2}$                                        &$H^{1,2,\frac32(q-2)}$,                       \\
\cdashline{3-3}[1pt/1pt]\cdashline{5-6}[1pt/1pt]
                                 & well--posedness                                  & well-- posedness                                   &                                  &well--posedness                                           &well-- posedness                                        \\
                                 &  in                                              &  in                                                &                                  & in                                                       & in                                        \\
                                 & $H^1$                                            &$H^{1,2,6+\varepsilon}$                             &                                  &$H^{1,2,\frac32(q-2)+\varepsilon}$                        &$H^{1,2,\frac32(q-2)+\varepsilon}$            \\
\cline{1-3}\cline{5-6}
                                 &                                                  &                                                    &                                  &local existence                                           &                                               \\
                                 &                                                  &                                                    &                                  & in  $H^1$                                                &                                               \\
\cdashline{5-5}[1pt/1pt]
                                 &                                                  &                                                    &                                  &global                                                    &                                               \\
                                 &                                                  &                                                    &                                  &existence in $H^{1,2,q}$                                  &                                               \\
\cdashline{5-5}[1pt/1pt]
                                 &existence--                                       &existence--                                         &                                  &existence--                                               &existence--                                            \\
$4=p      <m$                    &uniqueness in                                     &uniqueness in                                       &                                  &uniqueness in                                             &uniqueness in                                            \\
                                 &$H^1$                                             &$H^1$                                               &                                  &$H^{1,2,3(q-2)/2}$                                        &$H^{1,2,3(q-2)/2}$                          \\
\cdashline{2-3}[1pt/1pt]\cdashline{5-6}[1pt/1pt]
                                 &well--                                            &well--                                              &                                  &well--                                                    &well--                                                 \\
                                 &posedness in                                      &posedness in                                        &                                  &posedness in                                              &posedness in                                             \\
                                 &$H^{1,4+\varepsilon,2}$                           &$H^{1,4+\varepsilon,6+\varepsilon}$                 &                                  &$H^{1,4+\varepsilon,3(q-2)/2+\varepsilon}$                &$H^{1,4+\varepsilon,3(q-2)/2+\varepsilon}$      \\
\cline{1-3}\cline{5-6}
                                  &  \multicolumn{3}{>{\tiny}c|}{}                                                                                                          &local existence                                           &                                                \\
                                  &  \multicolumn{3}{>{\tiny}c|}{}                                                                                                          & in $H^1$                                                 &                                                \\
\cdashline{5-5}[1pt/1pt]
                                  &  \multicolumn{3}{>{\tiny}c|}{}                                                                                                          &global                                                    &                                                \\
$4=p=m$                           &  \multicolumn{3}{>{\tiny}c|}{}                                                                                                          &existence in $H^{1,2,q}$                                  &                                                \\
\cdashline{5-5}[1pt/1pt]
                                  &  \multicolumn{3}{>{\tiny}c|}{\qquad \qquad\qquad \qquad \qquad \qquad \qquad \quad existence--}                                         &existence--                                               &existence--                                     \\
                                  &  \multicolumn{3}{>{\tiny}c|}{\qquad \qquad\qquad \qquad \qquad \qquad \qquad \quad uniqueness in}                                       &uniqueness in                                             &uniqueness in                                     \\
                                  &  \multicolumn{3}{>{\tiny}c|}{\qquad \qquad\qquad \qquad \qquad \qquad \qquad \quad  $H^1$}                                              &$H^{1,2,3(q-2)/2}$                                        &$H^{1,2,3(q-2)/2}$                          \\
\hline
                                 &local  existence                                   &local existence                                    &local  existence                 &local existence                                                    &local  existence                                                   \\
                                 &in  $H^1$                                          &in  $H^1$                                          &in  $H^1$                        &in  $H^1$                                                          &in  $H^{1,2,3(q-2)/2}$                                  \\
\cdashline{2-6}[1pt/1pt]
                                 &global existence                                   &global existence                                   &global existence                 &global existence                                                   &global existence                                                   \\
$4<p<$                           &in $H^{1,p,2}$                                     &in $H^{1,p,2}$                                     &in $H^{1,p,2}$                   &in $H^{1,p,q}$                                                     &in $H^{1,p,3(q-2)/2}$                                      \\
\cdashline{2-6}[1pt/1pt]
$1+4/m'$                         &existence--                                        &existence--                                        &existence--                      &existence--                                                        &existence--                                              \\
                                 &uniqueness in                                      &uniqueness in                                      &uniqueness in                    &uniqueness in                                                      &uniqueness in                                             \\
                                 &$H^{1,2(p-2),2}$                                   &$H^{1,2(p-2),2}$                                   &$H^{1,2(p-2),2}$                 &$H^{1,2(p-2),3(q-2)/2}$                                            &$H^{1,2(p-2),3(q-2)/2}$                                \\
\cdashline{2-6}[1pt/1pt]
                                 &well--                                             &well--                                             &                                 &well--                                                             &well--                                                       \\
                                 &posedness in                                       &posedness in                                       &                                 &posedness in                                                       &posedness in                                              \\
                                 &$H^{1,2(p-2)+\varepsilon,2}$                       &$H^{1,2(p-2)+\varepsilon,6+\varepsilon}$           &                                 &$H^{1,2(p-2)+\varepsilon,3(q-2)/2+\varepsilon}$                    &$H^{1,2(p-2)+\varepsilon,3(q-2)/2+\varepsilon}$  \\
\hline
                                 &                                                   &                                                   &                                  &local existence                                                    &                                                    \\
                                 &                                                   &                                                   &                                  &in $H^{1,2(p-2),2}$                                                &                                                    \\
 \cdashline{5-5}[1pt/1pt]
                                 &                                                   &                                                   &                                  &local existence                                                    &                                                    \\
$4<p=$                           &                                                   &                                                   &                                  &in $H^{1,2(p-2),q}$                                                &                                                    \\
 \cdashline{5-5}[1pt/1pt]
$1+4/m'$                         &existence--                                        &existence--                                        &existence--                      &existence--                                                        &existence--                                              \\
                                 &uniqueness in                                      &uniqueness in                                      &uniqueness in                    &uniqueness in                                                      &uniqueness in                                             \\
                                 &$H^{1,2(p-2),2}$                                   &$H^{1,2(p-2),2}$                                   &$H^{1,2(p-2),2}$                 &$H^{1,2(p-2),3(q-2)/2}$                                            &$H^{1,2(p-2),3(q-2)/2}$                                \\
\cdashline{2-3}[1pt/1pt]\cdashline{5-6}[1pt/1pt]
                                 &well--                                             &well--                                             &                                 &well--                                                             &well--                                                       \\
                                 &posedness in                                       &posedness in                                       &                                 &posedness in                                                       &posedness in                                              \\
                                 &$H^{1,2(p-2)+\varepsilon,2}$                       &$H^{1,2(p-2)+\varepsilon,6+\varepsilon}$           &                                 &$H^{1,2(p-2)+\varepsilon,3(q-2)/2+\varepsilon}$                    &$H^{1,2(p-2)+\varepsilon,3(q-2)/2+\varepsilon}$  \\

\hline
\end{tabular}
\end{center}
\end{table}
\begin{table}
\begin{center}
\setlength{\tabcolsep}{2pt}
\renewcommand{\arraystretch}{0.5}
\setlength{\arrayrulewidth}{0.7pt}
\begin{tabular}{|>{\tiny}c|>{\tiny}c|>{\tiny}c|>{\tiny}c|>{\tiny}c|>{\tiny}c|}
\multicolumn{6}{>{\normalsize}c}{\sc Table 3. \rm Main results when $N=5$}\\
\multicolumn{6}{>{\normalsize}c}{}\\
\hline
                                 &\raisebox{1.5ex}[12pt]{\tiny{$2\le q<4$}}\,\vline\,\raisebox{1.5ex}[12pt]{\tiny{$2\le q<4$}}   &  \raisebox{1.5ex}[12pt]{\tiny{$4=q<\mu$}}                            & \raisebox{1.5ex}[12pt]{\tiny{$4=q=\mu$}}                 & \raisebox{1.5ex}[12pt]{\tiny{$4<q<1+ 4/\mu'$}}                                  & \raisebox{1.5ex}[12pt]{\tiny{$4<q=1+4/\mu'$}}            \\
\hline
                                 &                                                  &                                                    &                                  &local  existence                                          &                                               \\
                                 &                                                  &                                                    &                                  &in  $H^1$                                                 &                                               \\
\cdashline{5-5}[1pt/1pt]
                                 &                                                  &                                                    &                                  &global existence                                          &                                               \\
                                 &                                                  &                                                    &                                  & in $H^{1,2,q}$                                           &                                               \\
\cdashline{5-5}[1pt/1pt]
                                 &                                                  &existence--                                         &                                  &existence--                                               &existence--                                    \\
$2\le p\le\frac 83$                 &                                                  &uniqueness in                                       &                                  &uniqueness in                                            &uniqueness in                                      \\
                                 &                                                  &$H^1$                                               &                                  &$H^{1,2,2(q-2)}$                                          &$H^{1,2,2(q-2)}$                       \\
\cdashline{3-3}[1pt/1pt]\cdashline{5-6}[1pt/1pt]
                                 & well--posedness                                  & well-- posedness                                   &                                  &well--posedness                                           &well-- posedness                                        \\
                                 &  in                                              &  in                                                &                                  & in                                                       & in                                        \\
                                 & $H^1$                                            &$H^{1,2,4+\varepsilon}$                             &                                  &$H^{1,2,2(q-2)+\varepsilon}$                              &$H^{1,2,2(q-2)+\varepsilon}$            \\
\cline{1-3}\cline{5-6}
                                  &  \multicolumn{3}{>{\tiny}c|}{}                                                                                                          &local existence                                           &                      \\
$\frac 83<p \le\frac{10}3$        &  \multicolumn{3}{>{\tiny}c|}{\qquad \qquad\qquad \qquad \qquad \qquad existence}                                                                                                          & in $H^1$                       &                                               \\
\cdashline{5-5}[1pt/1pt]
                                  &\multicolumn{3}{>{\tiny}c|}{\qquad \qquad\qquad \qquad \qquad \qquad in $H^1$}                                                                                                             &global                          &                                              \\
                                  &\multicolumn{3}{>{\tiny}c|}{}                                                                                                             &existence in $H^{1,2,q}$                                                         &existence                                               \\
\cline{1-5}
                                 &\multicolumn{3}{>{\tiny}c|}{local existence}                                                                                                            &local existence                                                    &in $H^{1,p,2(q-2)}$                                                \\
$\frac{10}3<p<1+\frac{10}{3m'}$  &\multicolumn{3}{>{\tiny}c|}{in $H^1$}                                                                                                            &in  $H^1$                                                                 &                                 \\
\cdashline{2-5}[1pt/1pt]
                                 &\multicolumn{3}{>{\tiny}c|}{global existence}                                                                                                            &global existence                                                 &                                                  \\
                                 &\multicolumn{3}{>{\tiny}c|}{in $H^{1,p,2}$}                                                                                                            &in $H^{1,p,q}$                                                     &                                      \\
\hline
&\multicolumn{5}{>{\tiny}c|}{} \\
$\frac{10}3<p=1+\frac{10}{3m'}$&\multicolumn{5}{>{\tiny}c|}{no results} \\
&\multicolumn{5}{>{\tiny}c|}{} \\
\hline
\end{tabular}

\bigskip

\setlength{\tabcolsep}{1pt}
\renewcommand{\arraystretch}{1.2}
\setlength{\arrayrulewidth}{0.7pt}
\begin{tabular}{|>{\tiny}c|>{\tiny}c|>{\tiny}c|>{\tiny}c|>{\tiny}c|}
\multicolumn{5}{>{\normalsize}c}{\sc Table 4. \rm Main results when $N\ge 6$}\\
\multicolumn{5}{>{\normalsize}c}{}\\
\hline
                                 &\raisebox{1.5ex}[12pt]{\tiny{$2\le q\le 1+\rgamma/2$}}   &  \raisebox{1.5ex}[12pt]{\tiny{$1+\rgamma/2<q\le\rgamma$}}                            & \raisebox{1.5ex}[12pt]{\tiny{$\rgamma<q<1+\rgamma/\mu'$}}                 & \raisebox{1.5ex}[12pt]{\tiny{$\rgamma<q=1+\rgamma/\mu'$}}                                            \\
\hline
$2\le p\le 1+\romega/2$          &well--posedness in $H^1$                                   &                                      &  local  existence  in  $H^1$                                 &                                                                                \\
\cline{1-2} \cdashline{4-4}[1pt/1pt]
$1+\romega/2<p\le\romega$        &\multicolumn{2}{>{\tiny}c|}{\qquad \qquad\qquad \qquad \qquad existence in $H^1$}                        &global existence   in $H^{1,2,q}$                              &                                                                                    \\
\cline{1-4}
$\romega<p<1+\romega/m'$                                 &  \multicolumn{2}{>{\tiny}c|}{local existence in $H^1$}                                                        &local existence  in $H^1$                 &                                                               \\
\cdashline{2-4}[1pt/1pt]
                                 &\multicolumn{2}{>{\tiny}c|}{global existence in $H^{1,p,2}$}                                                   &global   existence   in $H^{1,p,q}$                      &          \\
\cline{1-4}
$\romega<p=1+\romega/m'$&\multicolumn{4}{>{\tiny}c|}{\qquad\qquad \qquad\qquad\qquad\qquad \qquad\qquad \qquad\qquad \qquad\qquad \qquad\qquad\qquad no results} \\
\hline
\end{tabular}
\end{center}
\end{table}

Tables 1--4 show that when $N\ge 3$ and $\essinf_\Omega \alpha>0$ or $\essinf_{\Gamma_1}\beta>0$ the analysis made in the present paper essentially extends
the results in \cite{Dresda1}.  In dimension $N=3$ we get new results only when $\essinf_\Omega \alpha>0$,
as it is natural due to the essential role played by the damping term, which is even more clear in dimensions $N\ge 4$.
Moreover Tables 1--4  suggest that, in dimensions $N=3,4$ and partially in dimension $N=5$, in presence of a couple of effective damping terms,
the standard source classification presented above is mainly of technical nature, while Sobolev--criticality and belonging to the hyperbola are essential.

The outcomes of the analysis in the full scale of spaces, contained in Corollaries~\ref{corollary1.3}--\ref{corollary1.4},
Theorem~\ref{theorem1.4} and Corollary~\ref{corollary1.7}, are summarized (for simplicity when $\essinf_\Omega \alpha>0, \essinf_{\Gamma_1}\beta>0$)
in Tables 5--8, p. \pageref{Table5}--\pageref{Tables7-8}. In them we follow the same conventions presented above and we denote
$x\vee y=\max\{x,y,\}$ for $x,y\in\R$.

The paper is organized as follows:
\begin{itemize}
\item[-] in Sections~\ref{section 2}--\ref{section3} we recall some background material from
\cite{Dresda1}, we state our assumptions on $f$, $g$, $P$ $Q$  and we give some preliminary results;
\item[-] in
Section~\ref{section4} we give a key estimate to be used in the sequel;
\item[-]  Section~\ref{section5} is devoted to local
existence for problem \eqref{1.1}, including the proofs of Theorem~\ref{theorem1.1} and
Corollaries~\ref{corollary1.1}--\ref{corollary1.3};
\item[-]  Section~\ref{section6}
deals with uniqueness and local well--posedness, including the proofs of Theorems~\ref{theorem1.2}--\ref{theorem1.4} and
Corollary~\ref{corollary1.4};
\item[-]  Section~\ref{section7} is
devoted to our global analysis, including the proofs of Theorem~\ref{theorem1.5} and
Corollaries~\ref{corollary1.5}--\ref{corollary1.7};
\item[-] in Appendix~\ref{appendixB} we prove a density result used in the
paper.
\end{itemize}

\section{Background}\label{section 2}
\subsection{Notation.} We shall adopt the standard notation for
(real) Lebesgue and Sobolev spaces in $\Omega$ (see \cite{adams})
and $\Gamma$ (see \cite{grisvard}). As usual  $\rho'$ is the
H\"{o}lder conjugate of $\rho$, i.e. $1/\rho+1/\rho'=1$. Given a
Banach space $X$ and its dual $X'$ we shall denote by $\langle
\cdot,\cdot\rangle_X$ the duality product between them. Finally, we
shall use the standard notation for vector valued Lebesgue and
Sobolev spaces in a real interval, with the exception that the
derivative of $u$, a time derivative, will be denoted by $\dot{u}$.

 Given $\alpha\in
L^\infty(\Omega)$, $\beta\in L^\infty(\Gamma)$, $\alpha,\beta\ge 0$
and $\rho\in [1,\infty]$ we shall respectively denote by
$(L^\rho(\Omega),\|\cdot\|_\rho)$,
$(L^\rho(\Gamma),\|\cdot\|_{\rho,\Gamma})$,
$(L^\rho(\Gamma_1),\|\cdot\|_{\rho,\Gamma_1})$,
$(L^\rho(\Omega;\lambda_\alpha),\|\cdot\|_{\rho,\alpha})$,
$(L^\rho(\Gamma;\lambda_\beta),\|\cdot\|_{\rho,\beta,\Gamma})$ and
$(L^\rho(\Gamma_1;\lambda_\beta),\|\cdot\|_{\rho,\beta,\Gamma_1})$
the  Lebesgue spaces (and norms) with respect to the following
measures: the standard Lebesgue one in $\Omega$, the hypersurface
measure $\sigma$ on $\Gamma$ and $\Gamma_1$, $\lambda_\alpha$ in
$\Omega$ defined  by $d\lambda_\alpha=\lambda_\alpha\,dx$,
 $\lambda_\beta$ on $\Gamma$ and $\Gamma_1$ defined  by
$d\lambda_\beta=\lambda_\beta\,d\sigma$. The equivalence classes
with respect to the measures $\lambda_\alpha$ and $\lambda_\beta$
will be respectively denoted by $[\cdot]_\alpha$ and
$[\cdot]_\beta$.

We recall some well--known preliminaries on the Riemannian gradient,
where only the fact that $\Gamma$ is a $C^1$  compact manifold
endowed with a $C^0$ Riemannian metric is used. We refer to
\cite{taylor} for more details and proofs, given there for smooth
manifolds, and to \cite{sternberg} for a general background on
differential geometry on $C^k$  manifolds.  We  denote
by $(\cdot,\cdot)_\Gamma$ the metric inherited from $\R^N$,
  given in local coordinates
$(y_1,\ldots,y_{N-1})$ by $(g_{ij})_{i,j=1,\ldots,N-1}$,
$|\cdot|_\Gamma^2=(\cdot,\cdot)_\Gamma$, by $d\sigma$ the natural
volume element on $\Gamma$, given by $\sqrt{\widetilde{g}}
\,\,dy_1\wedge\ldots\wedge dy_{N-1}$, where
$\widetilde{g}=\operatorname{det} (g_{ij})$.  We denote by
$(\cdot|\cdot)_\Gamma$ the Riemannian (real) inner product on
$1$-forms on $\Gamma$ associated to the metric, given in local
coordinates by $(g^{ij})=(g_{ij})^{-1}$,   by $d_\Gamma$ the total
differential on $\Gamma$ and by $\nabla_\Gamma$ the Riemannian
gradient, given in local coordinates by $\nabla_\Gamma
u=g^{ij}\,\partial_j u\,\,\partial_i$ for any $u\in H^1(\Gamma)$. It
is then clear that $(d_\Gamma u|d_\Gamma v)_\Gamma=(\nabla_\Gamma u,
\nabla_\Gamma v)_\Gamma$ for $u,v\in H^1(\Gamma)$, so the use of
vectors or forms in the sequel is optional. It is well--known (see
\cite{taylor} in the smooth setting, and \cite{quarteroni} in the
$C^1$ setting) that the norm
$\|u\|_{H^1(\Gamma)}^2=\|u\|_{2,\Gamma}^2+\|\nabla_\Gamma
u\|_{2,\Gamma}^2$,  where $\|\nabla_\Gamma
u\|_{2,\Gamma}^2:=\int_\Gamma |\nabla_\Gamma u|_\Gamma^2$, is
equivalent in $H^1(\Gamma)$  to the standard one. In the sequel, the
notation $d\sigma$ will be dropped from the boundary integrals.

\subsection{Functional setting and weak solutions for a linear problem}
We start by recalling some facts about the spaces $L^{2,\rho}_\alpha(\Omega)$ and
$L^{2,\rho}_\beta(\Gamma_1)$, refereing  to \cite{Dresda1} for more details and proofs.
They are reflexive and, making the standard
identifications
\begin{equation}\label{dualiLebesgue}
[L^\rho(\Omega)]'\simeq L^{\rho'}(\Omega),\qquad\text{and}\quad
[L^\rho(\Gamma_1)]'\simeq L^{\rho'}(\Gamma_1),
\end{equation}
when $\rho\in [2,\infty)$ we have the two chains of embedding
\begin{footnote}{by \eqref{dualiLebesgue} we \emph{can not} identify $[L^\rho(\Omega,\lambda_\alpha)]'$ and
$[L^\rho(\Gamma_1,\lambda_\beta)]'$  with
$L^{\rho'}(\Omega,\lambda_\alpha)$ and
$L^{\rho'}(\Gamma_1,\lambda_\beta)$.}\end{footnote}
\begin{equation}\label{2.2}
[L^\rho(\Omega,\lambda_\alpha)]' \hookrightarrow [L^{2,\rho}_\alpha
(\Omega)]' \hookrightarrow L^{\rho'}(\Omega),\,
[L^\rho(\Gamma_1,\lambda_\beta)]' \hookrightarrow [L^{2,\rho}_\beta
(\Gamma_1)]' \hookrightarrow L^{\rho'}(\Gamma_1).
\end{equation}
 Next, given $\rho,\theta\in [2,\infty)$ and
$-\infty\le a<b\le\infty$ we introduce the reflexive space
$$L^{2,\rho,\theta}_{\alpha,\beta}(a,b)=L^\rho(a,b\, ;
L^{2,\rho}_\alpha (\Omega))\times L^\theta(a,b\, ;
L^{2,\theta}_\beta (\Gamma_1)),$$ with its dual
\begin{equation}\label{2.3}
[L^{2,\rho,\theta}_{\alpha,\beta}(a,b)]'\simeq L^{\rho'}(a,b\, ;
[L^{2,\rho}_\alpha (\Omega)]')\times L^{\theta'}(a,b\, ;
[L^{2,\theta}_\beta (\Gamma_1)]').
\end{equation}
By \eqref{2.2}--\eqref{2.3} we have the embedding
\begin{equation}\label{2.4}
L^{\rho'}(a,b\, ; [L^\rho(\Omega,\lambda_\alpha)]')\times
L^{\theta'}(a,b\, ;
[L^\theta(\Gamma_1,\lambda_\beta)]')\hookrightarrow
[L^{2,\rho,\theta}_{\alpha,\beta}(a,b)]'.
\end{equation}
The space $H^1$ introduced in \eqref{H1} is endowed with the norm
\begin{equation}\label{2.5} \|u\|_{H^1}^2=\int_\Omega |\nabla
u|^2+\int_{\Gamma_1} |\nabla_\Gamma
u|_\Gamma^2+\int_{\Gamma_1}|u|^2,
\end{equation}
equivalent to the one inherited from the product. The definition of the space $H^{1,\rho,\theta}_{\alpha,\beta}$ given in
\eqref{H1plus} can be extended also for $\rho,\theta=\infty$, loosing reflexivity, and clearly $H^{1,\rho,\theta}_{\alpha,\beta}\hookrightarrow H^1 \hookrightarrow
H^0$ and $H^{1,\rho,\theta}_{\alpha,\beta}\hookrightarrow
L^{2,\rho}_\alpha(\Omega)\times
L^{2,\theta}_\beta(\Gamma_1)\hookrightarrow H^0$, which are dense
thanks to \cite[Lemma 2.1]{Dresda1}.

Finally we introduce the phase spaces for problem \eqref{1.1}, that
is
\begin{equation}\label{varicalH}
\cal{H}=H^1\times H^0 \qquad\text{and}\quad
\cal{H}^{\rho,\theta}=H^{1,\rho,\theta}\times H^0\quad\text{for
$\rho,\theta\in [2,\infty)$.}\end{equation}
 We consider the linear
evolution boundary value problem
\begin{equation}\label{L}
\begin{cases} u_{tt}-\Delta u=\xi \qquad &\text{in
$(0,T)\times\Omega$,}\\
u=0 &\text{on $(0,T)\times \Gamma_0$,}\\
u_{tt}+\partial_\nu u-\Delta_\Gamma u=\eta\qquad &\text{on
$(0,T)\times \Gamma_1$,}
\end{cases}
\end{equation}
where $0<T<\infty$ and $\xi=\xi(t,x)$, $\eta=\eta(t,x)$ are given
forcing terms of the form
\begin{equation}\label{2.7}
\left\{
\begin{alignedat}3
&\xi=\xi_1+\alpha \xi_2,\qquad && \xi_1\in L^1(0,T;L^2(\Omega)), \quad && \xi_2\in L^{\rho'}(0,T;L^{\rho'}(\Omega, \lambda_\alpha)), \\
&\eta=\eta_1+\beta \eta_2,\qquad && \eta_1\in
L^1(0,T;L^2(\Gamma_1)), \quad && \eta_2\in
L^{\theta'}(0,T;L^{\theta'}(\Gamma_1, \lambda_\beta)),
\end{alignedat}\right.
\end{equation}
where $\alpha\in L^\infty(\Omega)$, $\beta\in L^\infty(\Gamma_1)$,
$\alpha, \beta\ge 0$ and $\rho,\theta\in [2,\infty)$.

By a {\em weak solution} of \eqref{L} in $[0,T]$   we mean
\begin{equation}\label{2.8}
u\in L^\infty(0,T;H^1)\cap W^{1,\infty}(0,T;H^0),\qquad \dot{u}\in
L^{2,\rho,\theta}_{\alpha,\beta}(0,T),
\end{equation}
such that the distribution identity
\begin{equation}\label{2.9}
\int_0^T\left[-(\dot{u},\dot{\phi})_{H^0}+\int_\Omega \nabla u\nabla
\phi+\int_{\Gamma_1} (\nabla_\Gamma
u,\nabla_\Gamma\phi)_\Gamma-\int_\Omega
\xi\phi-\int_{\Gamma_1}\eta\phi\right]=0
\end{equation}
holds for all $\phi\in C_c((0,T);H^1)\cap C^1_c((0,T);H^0)\cap
L^{2,\rho,\theta}_{\alpha,\beta}(0,T)$.

When dealing with $u$ satisfying \eqref{2.8}, we shall {\em
systematically denote in the paper}
$$\dot{u}=(u_t,{u_{|\Gamma}}_t)\qquad\text{and}\quad U=(u,\dot{u})\in
L^\infty([0,T];\cal{H}).$$ We shall also denote
$A\in\cal{L}(H^1,(H^1)')$ defined by
\begin{equation}\label{2.10}
\langle Au,v\rangle_{H^1}=\int_\Omega \nabla u\nabla
v+\int_{\Gamma_1} (\nabla_\Gamma u,\nabla_\Gamma v)_\Gamma,\qquad
\text{for all $u,v\in H^1$}.
\end{equation}
We recall the following result (see \cite[Lemma 2.2]{Dresda1}).
\begin{lem}\label{lemma2.1}
Suppose that \eqref{2.7} holds. Then any weak solution $u$ of
\eqref{L} enjoys the further regularity $U\in C([0,T];\cal{H})$ and
satisfies the energy identity
$$\frac 12\|\dot{u}\|_{H^0}^2+\frac 12
\langle Au, u\rangle _{H^1}\Big|_s^t=\int_s^t\int_\Omega \xi
u_t+\int_{\Gamma_1}\eta {u_{|\Gamma}}_t\,d\tau  $$ for $0\le s\le
t\le T$. Moreover \eqref{2.9} holds in the generalized form
$$(\dot{u},\phi)_{H^0}\Big|_0^T+
\int_0^T\left[-(\dot{u},\dot{\phi})_{H^0}+\langle
Au,\phi\rangle_{H^1}-\int_\Omega
\xi\phi-\int_{\Gamma_1}\eta\phi\right]=0$$ for all $\phi\in
C([0,T];H^1)\cap C^1([0,T];H^0)\cap
L^{2,\rho,\theta}_{\alpha,\beta}(0,T)$.
\end{lem}

\section{Preliminaries}
\label{section3}
\subsection{Main assumptions}
With reference to problem \eqref{1.1} we suppose that
\renewcommand{\labelenumi}{{(PQ\arabic{enumi})}}
\begin{enumerate}
\item $P$ and $Q$ are Carath\'eodory functions, respectively in
$\Omega\times\R$ and $\Gamma_1\times\R$, and  there are $\alpha\in
L^\infty(\Omega)$, $\beta\in L^\infty(\Gamma_1)$, $\alpha,\beta\ge
0$, and $m,\mu>1$, $c_m,c_\mu> 0$,  such that
\begin{alignat}2\label{Pgrowth}
|&P(x,v)|\le c_m\alpha(x) (1+|v|^{m-1})\,\,&&\text{for a.a.
$x\in\Omega$, all $v\in\R$;}\\\label{Qgrowth} |&Q(x,v)|\le
c_\mu\beta(x) (1+|v|^{\mu-1})\,\,&&\text{for a.a. $x\in\Gamma_1$,
all $v\in\R$;}
\end{alignat}
\item $P$ (respectively $Q$)  is monotone increasing in $v$ for a.a. $x\in\Omega$ ($x\in\Gamma_1$);
\item $P$ and $Q$ are coercive, that is there are constants $c'_m,c'_\mu>
0$ such that
\begin{alignat}2\label{Plow}
&P(x,v)v\ge c'_m\alpha(x)|v|^m \qquad&&\text{for a.a. $x\in\Omega$, all $v\in\R$;}\\
&Q(x,v)v\ge c'_\mu\beta(x)|v|^\mu\qquad&&\text{for a.a.
$x\in\Gamma_1$, all $v\in\R$.}
\end{alignat}
\end{enumerate}
\begin{rem}\label{remark3.1} Trivially (PQ1--3) yield
$P(\cdot,0)\equiv 0$ and $Q(\cdot,0)\equiv0$. Moreover, when
$P(x,v)=\alpha(x)P_0(v)$ and  $Q(x,v)=\beta(x)Q_0(v)$ with
$\alpha\in L^\infty(\Omega)$ and $\beta\in L^\infty(\Gamma_1)$,
$\alpha,\beta\ge 0$, (PQ1--3) reduce
 to assumption (I), p. \pageref{assumptionI}.
\end{rem}
We denote $\overline{m}=\max\{2,m\}$,
$\overline{\mu}=\max\{2,\mu\}$, and, for $-\infty\le a<b\le\infty$,
\begin{equation}\label{3.5}
Y=L^{2,\overline{m}}_\alpha(\Omega)\times
L^{2,\overline{\mu}}_\beta(\Gamma_1),\quad
X=H^{1,\overline{m},\overline{\mu}}_{\alpha,\beta}, \quad
Z(a,b)=L^{2,\overline{m},\overline{\mu}}_{\alpha,\beta}(a,b).
\end{equation}
By (PQ1) the Nemitskii operators $\widehat{P}$ and $\widehat{Q}$
(respectively) associated to $P$ and $Q$ are continuous from
$L^{\overline{m}}(\Omega)$ to $L^{\overline{m}'}(\Omega))\simeq
[L^{\overline{m}}(\Omega))]'$ and from
$L^{\overline{\mu}}(\Gamma_1)$ to
$L^{\overline{\mu}'}(\Gamma_1))\simeq
[L^{\overline{\mu}}(\Gamma_1))]'$, and they can be uniquely extended to
$\widehat{P}: L^{2,\overline{m}}_\alpha (\Omega)\to
[L^{\overline{m}}(\Omega,\lambda_\alpha)]'$  and $\widehat{Q}:
L^{2,\overline{\mu}}_\beta (\Gamma_1)\to
[L^{\overline{\mu}}(\Gamma_1,\lambda_\beta)]'$. We denote
$$B=(\widehat{P},\widehat{Q}):Y\to
[L^{\overline{m}}(\Omega,\lambda_\alpha)]'\times
[L^{\overline{\mu}}(\Gamma_1,\lambda_\beta)]'.$$
We recall (see \cite{Dresda1})
\begin{lem}\label{lemma3.1}
Let (PQ1---2) hold and $(a,b)\subset\R$ is bounded. Then
\renewcommand{\labelenumi}{{(\roman{enumi})}}
\begin{enumerate}
\item B is continuous and bounded from $Y$ to $[L^{\overline{m}}(\Omega,\lambda_\alpha)]'\times
[L^{\overline{\mu}}(\Gamma_1,\lambda_\beta)]'$ and hence, by
\eqref{2.3}, to $Y'$;
\item $B$ acts boundedly and continuously from $Z(a,b)$ to $L^{\overline{m}'}(a,b\, ; [L^{\overline{m}}(\Omega,\lambda_\alpha)]')\times
L^{\overline{\mu}'}(a,b\, ;
[L^{\overline{\mu}}(\Gamma_1,\lambda_\beta)]')$ and hence, by
\eqref{2.4}, to $Z'(a,b)$;
\item $B$ is monotone in $Y$ and in $Z(a,b)$.
\end{enumerate}
\end{lem}
Our main assumption on $f$ and $g$ is the following one:
\renewcommand{\labelenumi}{{(FG\arabic{enumi})\phantom{a}}}
\begin{enumerate}\item
\begin{enumerate}
\renewcommand{\labelenumii}{{(F\arabic{enumii})}}
\item $f$ is a Carath\'eodory function in
$\Omega\times\R$ and  there are an exponent $p\ge 2$ and constants
$c_p,c_p'\ge 0$  such that, for a.a. $x\in\Omega$ and all $u,v\in\R$,
\begin{align}
\label{3.6}&|f(x,u)|\le c_p(1+|u|^{p-1}),\\
\label{3.7}&|f(x,u)-f(x,v)| \le c'_p |u-v|(1+|u|^{p-2}+|v|^{p-2});
\end{align}
\renewcommand{\labelenumii}{{(G\arabic{enumii})}}
\setcounter{enumii}{0}
\item $g$ is a  Carath\'eodory function in
$\Gamma_1\times\R$, and  there are an exponent $q\ge 2$ and constants
$c_q,c_q'\ge 0$  such that, for a.a. $x\in\Gamma_1$ and all $u,v\in\R$,
\begin{align}
&|g(x,u)|\le c_q(1+|u|^{q-1}) \\
&|g(x,u)-g(x,v)|\le c'_q |u-v|(1+|u|^{q-2}+|v|^{q-2})
\end{align}
\end{enumerate}
\end{enumerate}
\begin{rem}\label{remark3.2}
 Assumption (FG1) can be equivalently formulated as follows:
\renewcommand{\labelenumi}{{(FG\arabic{enumi})$'$\phantom{a}}}
\renewcommand{\labelenumii}{{(F\arabic{enumii})$'$}}
\begin{enumerate}\item
\begin{enumerate}
\item $f$ is a Carath\'eodory function in
$\Omega\times\R$, $f(x,\cdot)\in
C^{0,1}_{\text{loc}}(\R)$ for a.a. $x\in\Omega$, and there are an exponent $p\ge 2$ and
constants $\widetilde{c_p},\widetilde{c_p}'\ge 0$ such that
\begin{alignat}2
\label{3.10}&|f(x,0)|\le \widetilde{c_p},\quad &&\text{for
a. a. $x\in\Omega$,}\\
\label{3.11}&|f_u(x,u)|\le \widetilde{c_p}'(1+|u|^{p-2}),\qquad&&\text{for
a.a. $(x,u)\in\Omega\times\R$, }
\end{alignat}
\renewcommand{\labelenumii}{{(G\arabic{enumii})$'$}}
\setcounter{enumii}0
\item $g$ is a Carath\'eodory function in
$\Omega\times\Gamma_1$, $g(x,\cdot)\in
C^{0,1}_{\text{loc}}(\R)$ for a.a. $x\in\Gamma_1$, and there are an exponent $q\ge 2$ and
constants $\widetilde{c_q},\widetilde{c_q}'\ge 0$ such that
\begin{alignat}2
&|g(x,0)|\le \widetilde{c_q},\quad &&\text{for
a. a. $x\in\Gamma_1$,}\\
&\label{3.13}|g_u(x,u)|\le \widetilde{c_q}'(1+|u|^{q-2}),\qquad&&\text{for
a.a. $(x,u)\in\Gamma_1\times\R$, }
\end{alignat}
\end{enumerate}
\end{enumerate}
Indeed by \eqref{3.6} we immediately get \eqref{3.10} with
$\widetilde{c_p}=c_p$ and by \eqref{3.7} we have $f(x,\cdot)\in
C^{0,1}_{\text{loc}}(\R)$ for a.a. $x\in\Omega$, hence $f_u$  exist a.e.
\begin{footnote}{the fact that measurable functions in an open set, which are locally absolutely continuous with respect to a variable,
 possess a.e. the partial derivative with respect to that variable is classical, as stated for example in \cite[p.297]{mm}. However
 the sceptical reader can prove it by repeating \cite[Proof~of~ Proposition~2.1~p.~173]{DiBenedettorealanalysis}
 for Carath\'eodory functions, so getting the measurability of the four Dini derivatives. Hence the set where the
 derivative does not exist is measurable and finally  it has zero measure by
 Fubini's theorem.}\end{footnote}
and \eqref{3.11} follows from \eqref{3.7}, with $\widetilde{c_p}'=2c'_p$.
In the same way from (G1) we get (G1)$'$ with $\widetilde{c_q}=c_q$ and $\widetilde{c_q}'=2c'_q$.
Conversely by (FG1)$'$, integrating
\eqref{3.11} and \eqref{3.13} with respect to the second variable in the convenient
interval, one gets (FG1), with
$c_p=\widetilde{c_p}+2\widetilde{c_p}'$,
$c_q=\widetilde{c_q}+2\widetilde{c_q}'$, $c'_p=\widetilde{c_p}'$,
$c'_q=\widetilde{c_q}'$.
Consequently when $f(x,u)=f_0(u)$ and
$g(x,u)=g_0(u)$ assumption (FG1) reduces to (II), p. \pageref{assumptionII}. Other relevant
examples of functions $f$ and $g$ satisfying (FG1) are given by
\begin{equation}\label{modellisupplementari}
\begin{alignedat}3
f_2(x,u)&=\gamma_1(x)|u|^{\widetilde{p}-2}u+
\gamma_2(x)|u|^{p-2}u+\gamma_3(x),\quad &&2\le\widetilde{p}\le
p,\,\, && \gamma_i\in L^\infty(\Omega),
\\
g_2(x,u)&=\delta_1(x)|u|^{\widetilde{q}-2}u+
\delta_2(x)|u|^{q-2}u+\delta_3(x),\, \quad &&2\le\widetilde{q}\le
q,\,\, &&\delta_i\in L^\infty(\Gamma_1),
\end{alignedat}
\end{equation}
and by
\begin{equation}\label{f3g3}
f_3(x,u)=\gamma(x) f_0(u),\quad g_3(x,u)=\delta(x) g_0(u),\quad
\gamma\in L^\infty(\Omega),\quad \delta\in L^\infty(\Gamma_1),
\end{equation}
where
$f_0$ and $g_0$ satisfy (II).
\end{rem}

Beside the structural assumptions (PQ1--3) and (FG1) we introduce
the following assumption relating $f$ with $P$ and $g$ with $Q$
\renewcommand{\labelenumi}{{(FGQP1)}}
\begin{enumerate}
\item $p$, $q$, $m$, $\mu$, $\alpha$ and $\beta$ in (PQ1--3) and (FG1) satisfy  (III),
p. \pageref{III}.
\end{enumerate}
\begin{rem} \label{remark3.3} To simplify several estimates in the sequel we
remark that, when $p>1+\romega/2$, so $\alpha_0:=\essinf_\Omega
\alpha>0$, by replacing $\alpha$ with $\alpha/\alpha_0$ and
consequently $c_m$ with $\alpha_0c_m$ in \eqref{Pgrowth} and $c'_m$
with $\alpha_0c'_m$ in \eqref{Plow} we can normalize $\alpha_0=1$.
For the same reason when $q>1+\rgamma/2$ we shall assume without
restriction that $\essinf_{\Gamma_1}\beta=1$.
\end{rem}

\begin{rem} Trivially, by (FGQP1), $m=\overline{m}>2$ when
$p>1+\romega/2$ and $\mu=\overline{\mu}>2$ when $q>1+\rgamma/2$.
Moreover, when $p\le1+\romega/2$ and  $q\le1+\rgamma/2$ assumption
(FGQP1) can be skipped.
\end{rem}

We now introduce the auxiliary exponents
\begin{equation}\label{mpmuq}
m_p=\begin{cases}
2 &\text{if $p\le 1+\romega/2$},\\
m=\overline{m} &\text{otherwise,}
\end{cases}\quad
\mu_q=\begin{cases}
2 &\text{if $q\le 1+\rgamma/2$},\\
\mu=\overline{\mu}, &\text{otherwise,}
\end{cases}
\end{equation}
so by \eqref{R} we have
\begin{equation}\label{mpmuqbound}
m_p'\le \romega/(p-1)\qquad\text{and}\quad \mu_q'\le\rgamma/(q-1).
\end{equation}
Moreover, by \eqref{3.5}, \eqref{mpmuq} and assumption (FGQP1), for
any $T>0$
\begin{equation}\label{ZLMP}
\|w\|_{L^{m_p}((0,T)\times\Omega)\times
L^{\mu_q}((0,T)\times\Gamma_1)}\le \|w\|_{Z(0,T)}\quad\text{for all
$w\in Z(0,T)$.}
\end{equation}
 The following lemma points out some
easy consequences of (PQ1--3), (FG1), \eqref{R}.
\begin{lem}\label{Nemitskii} If $\mathfrak{f},\mathfrak{g}$ satisfy (FG1) with constants
$\mathfrak{c_p}$,$\mathfrak{c_p}'$, $\mathfrak{c_q}$,$\mathfrak{c_q}'$,
and $\rho\in
[p-1,\infty)$, $\theta\in [q-1,\infty)$, the Nemitskii operators
$\widehat{\mathfrak{f}}: L^\rho(\Omega)\to L^{\rho/(p-1)}(\Omega)$ and
$\widehat{\mathfrak{g}}: L^\theta(\Gamma_1)\to L^{\theta/(q-1)}(\Gamma_1)$ associated to them are
locally Lipschitz and bounded, and  there are $k_1,k_2>0$,
depending only on $\Omega$, such that for any $R\ge 0$
\begin{alignat*}4
&\|\widehat{\mathfrak{f}}(u)\|_{\frac \rho {p-1}}&&\le  \mathfrak{c_p} k_1
(1+R^{p-1}),\quad
&&\|\widehat{\mathfrak{f}}(u)-\widehat{\mathfrak{f}}(v)\|_{\frac \rho {p-1}}&&\le  \mathfrak{c_p}'k_1(1+R^{p-2})\|u-v\|_\rho,\\
&\|\widehat{\mathfrak{g}}(\widetilde{u})\|_{\frac \theta {q-1},\Gamma_1}&&\le
\mathfrak{c_q} k_2 (1+R^{q-1}),\quad
&&\|\widehat{\mathfrak{g}}(\widetilde{u})-\widehat{\mathfrak{g}}(\widetilde{v})\|_{\frac
\theta {q-1},\Gamma_1}&&\le \mathfrak{c_q}'k_2
(1+R^{q-2})\|\widetilde{u}-\widetilde{v}\|_{\theta,\Gamma_1},
\end{alignat*}
provided $\|u\|_\rho, \|v\|_\rho, \|\widetilde{u}\|_{\theta,\Gamma_1},
\|\widetilde{v}\|_{\theta,\Gamma_1}\le R$.
Moreover, if also (PQ1--3) and \eqref{R} hold then $\widehat{\mathfrak{f}}:
H^1(\Omega)\to L^{m_p'}(\Omega)$ and $\widehat{\mathfrak{g}}: H^1(\Gamma)\cap
L^2(\Gamma_1)\to L^{\mu_q'}(\Gamma_1)$ enjoy the same properties and
there is $k_3>0$, depending only on $\Omega$, such that, for any
$R\ge 0$
\begin{alignat*}4
&\|\widehat{\mathfrak{f}}(u)\|_{m_p'}&&\le \mathfrak{c_p} k_3 (1+R^{p-1}),\quad
&&\|\widehat{\mathfrak{f}}(u)-\widehat{\mathfrak{f}}(v)\|_{m_p'}&&\le \mathfrak{c_p}'k_3 (1+R^{p-2})\|u-v\|_{H^1(\Omega)},\\
&\|\widehat{\mathfrak{g}}(\widetilde{u})\|_{\mu_q',\Gamma_1}&&\le \mathfrak{c_q} k_3
(1+R^{q-1}),\quad
&&\|\widehat{\mathfrak{g}}(\widetilde{u})-\widehat{\mathfrak{g}}(\widetilde{v})\|_{\mu_q',\Gamma_1}&&\le
\mathfrak{c_q}' k_3(1+R^{q-2})\|\widetilde{u}-\widetilde{v}\|_{H^1(\Gamma)},
\end{alignat*}
provided $\|u\|_{H^1(\Omega)}, \|v\|_{H^1(\Omega)},
\|\widetilde{u}\|_{H^1(\Gamma)}, \|\widetilde{v}\|_{H^1(\Gamma)}\le
R$.
\end{lem}
\subsection{Weak solutions}
We note that by Lemma~\ref{Nemitskii} and \eqref{mpmuqbound}, for
any $u$ satisfying \eqref{2.8}, we have
$$\widehat{f}(u)\in
L^\infty(0,T;L^{m_p'}(\Omega))\quad\text{and}\quad
\widehat{g}(u_{|\Gamma})\in L^\infty(0,T;L^{\mu_q'}(\Gamma_1)).$$
Hence, when $2\le p\le1+\romega/2$ we get $\widehat{f}(u)\in
L^1(0,T;L^2(\Omega))$, while when $p>1+\romega/2$ we get
$\widehat{f}(u)\in L^{m'}(0,T;L^{m'}(\Omega))$ and by
(FGQP1) we thus have $\widehat{f}(u)\in
L^{\overline{m}'}(0,T;[L^{\overline{m}}(\Omega;\lambda_\alpha)]')$.
In conclusion in both cases we can write $\widehat{f}(u)$ in the
form \eqref{2.7} with $\rho=\overline{m}$. Similar arguments show
that $\widehat{g}(u)$ can be written in the form \eqref{2.7} with
$\theta=\overline{\mu}$. Moreover, by Lemma~\ref{lemma3.1},
$\widehat{P}(u_t)\in L^{\overline{m}'}(0,T;[L^{\overline{m}}(\Omega,
\lambda_\alpha)]')$ and $\widehat{Q}({u_{|\Gamma}}_t)\in
L^{\overline{\mu}'}(0,T;[L^{\overline{\mu}}(\Gamma_1,
\lambda_\alpha)]')$, so they can be written in the same form.
By previous considerations and Lemma~\ref{lemma2.1} the following
definition makes sense.
\begin{definition}\label{definition3.1} Let (PQ1--3), (FG1),
(FGQP1) hold and $U_0=(u_0,u_1)\in\cal{H}$. A weak solution of
problem \eqref{1.1} in $[0,T]$, $0<T<\infty$,  is a weak solution of
\eqref{L} with
\begin{equation}\label{xieta}
\xi=\widehat{f}(u)-\widehat{P}(u_t),\quad
\eta=\widehat{g}(u_{|\Gamma})-\widehat{Q}({u_{|\Gamma}}_t), \quad
\rho=\overline{m}\quad\text{and}\quad \theta=\overline{\mu},
\end{equation} such that $u(0)=u_0$ and $\dot{u}(0)=u_1$.
A weak solution of \eqref{1.1} in $[0,T)$, $0<T\le\infty$, is $u\in
L^\infty_{\text{loc}}([0,T);H^1)$  which is a weak solution of
\eqref{1.1} in $[0,T']$ for any $T'\in (0,T)$. Such a solution is
called maximal if it has no proper extensions.
\end{definition}
Weak solutions enjoy good properties, as shown in the next result.
\begin{lem}\label{lemma3.3}
Let $u$ be a weak solution of \eqref{1.1} in $\text{dom }u=[0,T]$ or
$\text{dom }u=[0,T)$. Then
\renewcommand{\labelenumi}{{(\roman{enumi})}}
\begin{enumerate}
\item  $u\in C(\text{dom }u;H^1)\cap
C^1(\text{dom }u;H^0)$, it satisfies the energy identity
\begin{multline}\label{energyidentity}
\tfrac 12 \left[\int_\Omega u_t^2+\int_{\Gamma_1}
{u_{|\Gamma}}_t^2+\int_\Omega |\nabla
u|^2+\int_{\Gamma_1}|\nabla_\Gamma
u|_\Gamma^2\right]_s^t+\int_s^t\int_\Omega
P(\cdot,u_t)u_t\\
+\int_s^t\left[\int_{\Gamma_1}Q(\cdot,{u_{|\Gamma}}_t){u_{|\Gamma}}_t
-\int_\Omega f(\cdot,u)u_t-\int_{\Gamma_1}g(\cdot,
u){u_{|\Gamma}}_t\right]=0
\end{multline}
for all $s,t\in \text{dom }u$, and the distribution identity
\begin{multline}\label{generdistribution}
\left[\int_\Omega
u_t\phi+\int_{\Gamma_1}{u_{|\Gamma}}_t\phi\right]_0^{T'}
+\int_0^{T'}\left[-\int_\Omega
u_t\phi_t-\int_{\Gamma_1}{u_{|\Gamma}}_t{\phi_{|\Gamma}}_t+\int_\Omega
\nabla u\nabla \phi\right.\\
+\left.\int_{\Gamma_1} (\nabla_\Gamma
u,\nabla_\Gamma\phi)_\Gamma+\int_\Omega
P(\cdot,u_t)\phi+\int_{\Gamma_1}
Q(\cdot,{u_{|\Gamma}}_t)\phi-\!\int_\Omega
f(\cdot,u)\phi-\!\!\int_{\Gamma_1}  g(\cdot,u)\phi\right]=0
\end{multline}
for all $T'\in \text{dom }u$ and $\phi\in C([0,T'];H^1)\cap
C^1([0,T'];H^0)\cap Z(0,T')$;
\item if $u$ is a weak solution of \eqref{1.1} in $[0,T_u]$, $v$ is
a weak solution of \eqref{1.1} with initial data $v_0,v_1$ in
$[0,T_v]$ and $u(T_u)=v_0, \dot{u}(T_u)=v_1$ then $w$ defined  by
$w(t)=u(t)$ when $t\in [0,T_u]$, $w(t)=v(t-T_u)$ when $t\in
[T_u,T_u+T_v]$, is a weak solution of \eqref{1.1} in $[0,T_u+T_v]$;
\item for any  $(\rho,\theta)\in
[2,\max\{\romega,m\}]\times[2,\max\{\rgamma,\mu\}]\cap\R^2$
\begin{equation}\label{3.14}
u_0\in H^{1,\rho,\theta}_{\alpha,\beta}\Rightarrow u\in C(\text{dom
}u; H^{1,\rho,\theta}_{\alpha,\beta});
\end{equation}
\item if $\text{dom }u=[0,T)$, $T<\infty$ and  $U\in
L^\infty(0,T;\cal{H})$ then $u$ is a weak solution in $[0,T]$ and
$U\in C([0,T];\cal{H})$.
\end{enumerate}
\end{lem}
\begin{proof}
Clearly  (i) follows from Lemma~\ref{lemma2.1} while (ii) follows by
\eqref{generdistribution}
since \eqref{1.1} is autonomous. To prove (iii) let us take $u_0\in
H^{1,\rho,\theta}_{\alpha,\beta}$. When $m\le \romega$ then, by the
trivial embedding $L^\rho(\Omega)\hookrightarrow
L^{2,\rho}_\alpha(\Omega)$ and Sobolev embedding we get $u\in
C(\text{dom }u; L^{2,\rho}_\alpha(\Omega))$, while when $m>\romega$,
so $\rho\le m$, as $\dot{u}\in Z(0,T')$ for all $T'\in \text{dom }u$ we
get $u_t\in L^\rho(0,T';L^{2,\rho}_\alpha(\Omega))$, and hence $u\in
W^{1,\rho}(0,T';L^{2,\rho}_\alpha(\Omega))\hookrightarrow
C([0,T'];L^{2,\rho}_\alpha(\Omega))$ since $u_0\in
L^{2,\rho}_\alpha(\Omega)$ and $u(t)=u_0+\int_0^t
u_t(\tau,\cdot)\,d\tau$ in $L^2(\Omega)$. Then $u\in C(\text{dom }u;
H^{1,\rho,2}_{\alpha,1})$. Since the same arguments show that $u\in
C(\text{dom }u; H^{1,2,\theta}_{1,\beta})$ and
$H^{1,\rho,\theta}_{\alpha,\beta}=H^{1,\rho,2}_{\alpha,1}\cap
H^{1,2,\theta}_{1,\beta}$ we get \eqref{3.14}.

To prove (iv), thanks to (i), we just have to prove that if $\text{dom }u=[0,T)$,
$T<\infty$ and  $U\in L^\infty(0,T;\cal{H})$ then $\dot{u}\in
Z(0,T)$. Set $S=\|U\|_{L^\infty(0,T;\cal{H})}<\infty$. By the energy identity
\eqref{energyidentity} and assumption (PQ3)
\begin{multline}\label{3.22}
c_m'\int_0^t \|[u_t]_{\alpha}\|_{m,\alpha}^m+c_\mu'\int_0^t
\|[{u_{|\Gamma}}_t]_{\beta}\|_{\mu,\beta,\Gamma_1}^\mu\le \tfrac 12
\|U_0\|_{\cal{H}}^2\\+\int_0^t \int_\Omega\widehat{f}(u)u_t+\int_0^t
\int_{\Gamma_1}\widehat{g}(u){u_{|\Gamma}}_t\qquad\text{for $t\in
[0,T)$.}
\end{multline}
By H\"{o}lder and weighted Young inequalities together with
Lemma~\ref{Nemitskii} we have
\begin{equation}\label{3.23}
\int_0^t \int_\Omega\widehat{f}(u)u_t\le\int_0^t
c_pk_3(1+S^{p-1})\|u_t\|_{m_p}\le \tfrac
{c_m'}{m_p}\int_0^t\|u_t\|_{m_p}^{m_p}+K_1,
\end{equation}
where $K_1=(c_m')^{-1/m_p}[c_pk_3(1+S^{p-1})]^{m_p'}T$. We now
distinguish between the cases $2\le p\le 1+\romega/2$ and $p>
1+\romega/2$. In the first one, by \eqref{mpmuq}, we have $m_p=2$ so
$\int_0^t\|u_t\|_{m_p}^{m_p}\le S^2T$, while in the second one
$m_p=m$ and, by assumption (FGQP1) and Remark~\ref{remark3.3}, we
have $\int_0^t\|u_t\|_{m_p}^{m_p}\le \int_0^t
\|[u_t]_{\alpha}\|_{m,\alpha}^m$. Hence
\begin{equation}\label{3.24}
\int_0^t \int_\Omega\widehat{f}(u)u_t\le \tfrac {c_m'}m\int_0^t
\|[u_t]_{\alpha}\|_{m,\alpha}^m+K_2\quad\text{for all $t\in [0,T)$},
\end{equation}
where $K_2=\{(c_m')^{-1/m_p}[c_pk_3(1+S^{p-1})]^{m_p'}+c_m'S^2\}T$.
 Using the same arguments
\begin{equation}\label{3.25}
\int_0^t \int_{\Gamma_1}\widehat{g}(u){u_{|\Gamma}}_t\le \tfrac
{c_\mu'}\mu\int_0^t
\|[{u_{|\Gamma}}_t]_{\beta}\|_{\mu,\beta,\Gamma_1}^\mu+K_3\quad\text{for
all $t\in [0,T)$},
\end{equation}
where
$K_3=\{(c_\mu')^{-1/\mu_q}[c_qk_3(1+S^{q-1})]^{\mu_q'}+c_\mu'S^2\}T$.
 Plugging \eqref{3.24}--\eqref{3.25} in \eqref{3.22} we get
$$\tfrac{c_m'}{m'}\int_0^t
\|[u_t]_{\alpha}\|_{m,\alpha}^m+\tfrac{c_\mu'}{\mu'}\int_0^t
\|[{u_{|\Gamma}}_t]_{\beta}\|_{\mu,\beta,\Gamma_1}^\mu\le\tfrac 12
\|U_0\|_{\cal{H}}^2+K_2+K_3\quad\text{for all $t\in [0,T)$,}$$ from
which, since $\dot{u}\in L^\infty(0,T:H^0)$, we get $\dot{u}\in Z(0,T)$,
concluding he proof.
\end{proof}

\subsection{Additional assumptions}
The following properties of $f$ and $g$ will be
assumed only in connection with (FG1) and for some values of
$p,q>3$ to be precised:
\renewcommand{\labelenumi}{{(F2)}}
\begin{enumerate}
\item $f(x,\cdot)\in C^2(\R)$ for a.a. $x\in\Omega$
and there is a constant $c''_p\ge 0$ such that
$$|f_u(x,u)-f_u(x,v)|\le c''_p |u-v|(1+|u|^{p-3}+|v|^{p-3})\quad\text{for a.a. $x\in\Omega$ and all $u,v\in\R$;
}$$
\renewcommand{\labelenumi}{{(G2)}}
\item  $g(x,\cdot)\in C^2(\R)$ for a.a. $x\in\Gamma_1$
and there is a constant $c''_q\ge 0$ such that
$$|g_u(x,u)-g_u(x,v)|\le c''_q |u-v|(1+|u|^{q-3}+|v|^{q-3})\quad\text{for a.a. $x\in\Gamma_1$ and all $u,v\in\R$.}$$
\end{enumerate}

\begin{rem} \label{remark3.7} We point out that (F2) and (G2) are respectively equivalent to
\renewcommand{\labelenumi}{{(F2)$'$}}
\begin{enumerate}
\item $f(x,\cdot)\in C^2(\R)$ for a.a. $x\in\Omega$
and there is a constant $\widetilde{c_p}''\ge 0$ such that
$$|f_{uu}(x,u)|\le \widetilde{c_p}'' (1+|u|^{p-3})\quad\text{for a.a. $x\in\Omega$ and all $u\in\R$;
}$$
\renewcommand{\labelenumi}{{(G2)$'$}}
\item $g(x,\cdot)\in C^2(\R)$ for a.a. $x\in\Gamma_1$
and there is a constant $\widetilde{c_q}''\ge 0$ such that
$$|g_{uu}(x,u)|\le \widetilde{c_q}''(1+|u|^{q-3})\quad\text{for a.a. $x\in\Gamma_1$ and all $u\in\R$.
}$$
\end{enumerate}
Moreover (F2) implies (F2)$'$ with $\widetilde{c_p}''=2c_p''$,  (G2)
implies (G2)$'$ with $\widetilde{c_q}''=2c_q''$ and conversely (F2)$'$
implies (F2) with $\widetilde{c_p}''=c_p''$,  (G2)$'$ implies (G2)
with $\widetilde{c_q}''=c_q''$. Finally, in the case considered in
problem \eqref{1.2}, that is $f(x,u)=f_0(u)$ and $g(x,u)=g_0(u)$,
(F2) means that $f_0\in\cal{F}_p$, while (G2) that
$g_0\in\cal{F}_q$.
\end{rem}
In  the sequel we shall use one between the following
two assumptions, the latter being trivially stronger than the
former:
\begin{enumerate}
\renewcommand{\labelenumi}{{(FG2)}}
\item $N\le 4$ and (F2) holds when $1+\romega /2<p=1+\romega /m'$,  \\ $N\le 5$ and (G2)
holds when $1+\rgamma /2<q=1+\rgamma/ {\mu'}$;
\renewcommand{\labelenumi}{{(FG2)$'$}}
\item $N\le 4$ and (F2) holds when $p>1+\romega /2$,  \\ $N\le 5$ and (G2)
holds when $q>1+\rgamma /2$.
\end{enumerate}
\begin{rem} \label{remark3.6} By previous remark when $f(x,u)=f_0(u)$ and $g(x,u)=g_0(u)$
clearly (FG2) and (FG2)$'$ reduce to (IV) and (IV)$'$, pp. \pageref{assumptionIV}--\pageref{assumptionIV'}.
\end{rem}

The following properties of $P$ and $Q$ will be assumed only for
some values of $p,q,m$ and $\mu$ to be specified later on:
\begin{enumerate}
\renewcommand{\labelenumi}{{(P4)}}
\item there are constants $c_m'', M_m>0$  such that
\begin{equation}\label{P4}
P_v(x,v)\ge c_m''\, \alpha(x)|v|^{m-2}\quad\text{for a.a. $(x,v)\in
\Omega\times(\R\setminus (-M_m,M_m))$,}
\end{equation}
\renewcommand{\labelenumi}{{(Q4)}}
\item
there are constants $c_\mu'', M_\mu>0$  such that
\begin{equation}\label{Q4}
Q_v(x,v)\ge c_\mu''\, \beta(x)|v|^{\mu-2}\quad\text{for a.a.
$(x,v)\in \Gamma_1\times(\R\setminus (-M_\mu,M_\mu))$.}
\end{equation}

\end{enumerate}

\begin{rem}\label{remark4.2bis}
Since by (PQ1--2)  the partial derivatives $P_v$ and $Q_v$ exist
almost everywhere (see \cite{DiBenedettorealanalysis}) and are
nonnegative, \eqref{P4}--\eqref{Q4} always hold if one allows
$c_m''$ and $c_\mu''$ to vanish, and  (P4) and (Q4) respectively
reduce to ask that  there is $M_m>0$ such that one can take
$c_m''>0$ in \eqref{P4} and that there is $M_\mu>0$ such that one
can take $c_\mu''>0$ in \eqref{Q4}.
\end{rem}

In particular in our well--posedness result we shall use one between
the following two assumptions:
\begin{enumerate}
\renewcommand{\labelenumi}{{(PQ4)}}
\item if $p\ge \romega$\,\, then (P4) holds, if  $q\ge \rgamma$ \,then (Q4)
holds;
\renewcommand{\labelenumi}{{(PQ4)$'$}}
\item if $m>\romega$ then (P4) holds, if  $\mu>\rgamma$ then (Q4)
holds.
\end{enumerate}
Clearly, when  \eqref{R''} holds, (PQ4)$'$ is stronger than (PQ4) by
\eqref{somegagammarelazioni3}.

\begin{rem}\label{remark3.8} We remark some trivial consequences of assumptions (PQ1--4) and
(PQ1--3)--(PQ4)$'$. Setting $c_m''=0$ when (P4) is not assumed to hold
and $c_\mu''=0$ when (Q4) is not assumed to hold, since $P_v,Q_v\ge
0$ a.e., from (P4) and (Q4) (when they are assumed) we have
\begin{alignat}2\label{3.17}
P_v(x,v)\ge
&\alpha(x)\left[c_m''|v|^{m-2}-c_m'''\right]\quad&&\text{for a.a. $(x,v)\in \Omega\times\R$,}\\
\label{3.18} Q_v(x,v)\ge
&\beta(x)\left[c_\mu''|v|^{\mu-2}-c_\mu'''\right]\quad&&\text{for
a.a. $(x,v)\in \Gamma_1\times\R$,}
\end{alignat}
where $c_m'''=c_m''M_m^{m-2}$, $c_\mu'''=c_\mu''M_\mu^{\mu-2}$. Then, by
(PQ2), integrating \eqref{3.17} we get, for a.a. $x\in\Omega$
and all $v<w$,
\begin{equation}\label{3.19}
P(x,w)-P(x,v)\ge  \alpha(x)\left[\frac
{c_m''}{m-1}\left(|w|^{m-2}w-|v|^{m-2}v\right)-c_m'''(w-v)\right].\end{equation}
Consequently, using the elementary inequality
$$\left(|w|^{m-2}w-|v|^{m-2}v\right)(w-v)\ge \widetilde{c_m}|w-v|^m\qquad\text{for all $v,w\in\R$,}$$
where $\widetilde{c_m}$ is a positive constant, setting
$\widetilde{c_m}''=c_m''\widetilde{c_m}/(m-1)$, from \eqref{3.19} we
get
\begin{equation}\label{3.20}
\widetilde{c_m}''\alpha(x)|v-w|^{\overline{m}}\le c_m'''
\alpha(x)|v-w|^2+ (P(x,w)-P(x,v))(w-v)\end{equation} for a.a.
$x\in\Omega$ and all $v,w\in\R$, with $\widetilde{c_m}''>0$ when
$p\ge \romega$ if (PQ4) is assumed and when $m>\romega$ if (PQ4)$'$ is
assumed. Using the same arguments we get from \eqref{3.18} the
existence of $\widetilde{c_\mu}''\ge0$ such that
\begin{equation}\label{3.21}
\widetilde{c_\mu}''\beta(x)|v-w|^{\overline{\mu}}\le c_\mu'''
\beta(x)|v-w|^2+ (Q(x,w)-Q(x,v))(w-v)\end{equation} for a.a.
$x\in\Gamma_1$ and all $v,w\in\R$, with $\widetilde{c_\mu}''>0$ when
$q\ge \rgamma$ if (PQ4) is assumed and when $\mu>\rgamma$ if (PQ4)$'$
is assumed.
\end{rem}
\begin{rem}\label{remark3.9} When $P(x,v)=\alpha(x)P_0(v)$ and
$Q(x,v)=\beta(x)Q_0(v)$ with $\alpha\in L^\infty(\Omega)$, $\beta\in
L^\infty(\Gamma_1)$, $\alpha,\beta\ge 0$, (PQ4) and (PQ4)$'$
reduce to (V) and (V)$'$, p. \pageref{assumptionV}.
\end{rem}
\begin{rem}
In the paper  we shall introduce several positive constants
depending on $\Omega$, $P$, $Q$, $f$ and $g$, and on the various constants appearing in the assumptions. Since they are fixed
we shall denote these
constants by $k_i$, $i\in\N$.  We shall denote positive constants (possibly) depending
on other objects $\Upsilon_1,\ldots,\Upsilon_n$ by
$K_i=K_i(\Upsilon_1,\ldots,\Upsilon_n)$, $i\in\N$.
\end{rem}
\section{A key estimate}\label{section4}
This section is devoted to give the key estimate which will be used
when (FG2) or (FG2)$'$ hold,  the $\Omega$ -- version of which
constitutes the content of the following
\begin{lem}\label{lemma4.1} Suppose that
$\mathfrak{f}$ satisfies (F1) with constants  $\mathfrak{c_p}$, $\mathfrak{c_p}'$,  that
\eqref{R} holds and
that either $2\le p\le 1+\romega/2$ or $p>1+\romega/2$, $N\le 4$ and  $\mathfrak{f}$
satisfies (F2) with constant $\mathfrak{c_p}''$. Let $T\in (0,1]$, $U=(u,\dot{u}),
V=(v,\dot{v})\in C([0,T];\cal{H})$, $\dot{u},\dot{v}\in Z(0,T)$ and denote
$W=U-V=(w,\dot{w})$, $U_0=U(0)=(u_0,u_1)$, $V_0=V(0)=(v_0,v_1)$,
$W_0=U_0-V_0=(w_0,w_1)$, $\boldsymbol{\mathfrak{c_p}}=(\mathfrak{c_p},\mathfrak{c_p}')$,
$\boldsymbol{\mathfrak{c_p}'}=(\mathfrak{c_p},\mathfrak{c_p}',\mathfrak{c_p}'')$.
 Suppose moreover that $u_0,v_0\in
L^\somega(\Omega)$ and take $R\ge 0$ such that
\begin{equation}\label{4.1}
\|U\|_{C([0,T];\cal{H})},\, \|V\|_{C([0,T];\cal{H})},\,
\|\dot{u}\|_{Z(0,T)},\,\|\dot{v}\|_{Z(0,T)},\,\|u_0\|_\somega,\,
\|v_0\|_\somega\le R.
\end{equation}
Then given any $\eps>0$ there are
$$K_4=K_4(R,\boldsymbol{\mathfrak{c_p}})\qquad\text{and}\quad K_5=K_5(\eps,R,u_0,v_0,\boldsymbol{\mathfrak{c_p}'}),$$
independent on $\mathfrak{f}$ and increasing in $R$, such that for all $t\in
[0,T]$
\begin{multline}\label{4.3}
I_\mathfrak{f}(t):=\int_0^t\int_\Omega[\widehat{\mathfrak{f}}(u)-\widehat{\mathfrak{f}}(v)]w_t  \le
K_4(\eps+t)\|W(t)\|^2_\cal{H} \\+
K_5\left[\|W_0\|^2_\cal{H}+\int_0^t(1+\|u_t\|_{m_p}+\|v_t\|_{m_p})\|W(\tau)\|^2_\cal{H}\,d\tau\right].
\end{multline}
Moreover, if $u_0,v_0\in L^{s_1}(\Omega)$ for some $s_1>\somega$
and, in addition to \eqref{4.1}, we have
\begin{equation}\label{4.4}
\|u_0\|_{s_1},\quad \|v_0\|_{s_1}\le R,
\end{equation}
then $K_5$ is  independent on $u_0$ and $v_0$, that is
$K_5(\eps,R,u_0,v_0,\boldsymbol{\mathfrak{c_p}}')=K_6(\eps,R,\boldsymbol{\mathfrak{c_p}}')$.
\end{lem}
To prove Lemma~\ref{lemma4.1} we shall use the following well--known
abstract version of the Leibnitz formula, which can be proved as in \cite[Theorem~2,
p.~477]{dautraylionsvol5}.
\begin{lem}\label{lemma4.2} Let $T>0$, $X_1,X_2$ be Banach spaces,
$X_1\hookrightarrow X_2$ with dense embedding,  $u\in L^2(0,T;X_1)$,
$v\in L^2(0,T;X_2')$, $\dot{u}\in L^2(0,T;X_2)$ and $\dot{v}\in
L^2(0,T;X_1')$. Then $\langle v,u\rangle_{X_1}\in W^{1,1}(0,T)$ and
$\langle v,u\rangle'_{X_1}=\langle v,\dot{u}\rangle_{X_2}+\langle
\dot{v},u\rangle_{X_1}$.
\end{lem}
\begin{proof}[Proof of Lemma~\ref{lemma4.1}] We distinguish between
the cases $2\le p\le 1+\romega/2$ and $p>1+\romega/2$ since the
first one is trivial.  Indeed, as $m_p=2$, by Lemma~\ref{Nemitskii},
H\"{o}lder and Young inequalities and \eqref{4.1} we immediately get
\begin{equation}\label{4.5}
\begin{aligned}
I_{\mathfrak{f}}(t)\le & \int_0^t \|\widehat{{\mathfrak{f}}}(u)-\widehat{{\mathfrak{f}}}(v)\|_2\|w_t\|_2\le
c_p'k_3(1+R^{p-2})\int_0^t\|w\|_{H^1(\Omega)}\|w_t\|_2\\
\le & \tfrac 12 c_p'k_3(1+R^{p-2})\int_0^t
\|W(\tau)\|^2_{\cal{H}}\,d\tau
\end{aligned}
\end{equation}
so getting \eqref{4.3} with $K_4=1$ and $K_5=\frac 12
c_p'k_3(1+R^{p-2})$.

The case $p>1+\romega/2$ (so $\romega<\infty$) is much more
involved, and assumption $N\le 4$ and property (F2) will be essentially used. We set
$\oversomega=\max\{\romega,\somega\}<\infty$ and we note that, when
$p> \romega$, by \eqref{somegagamma}--\eqref{somegagammarelazioni1}
we have $\somega=\romega(p-2)/(\romega-2)>p$. Hence, as $\romega\ge
4$, for any value of $p$ we have $3<p\le \oversomega$. Since $m=\overline{m}$ by
\eqref{R}, we have
\begin{equation}\label{4.6}
[\min\{m,\oversomega\}]'=\max\{m',\oversomega'\}\le
\frac{\oversomega}{p-1}<\frac{\oversomega}{p-2},
\end{equation}
hence we can set $l_1>1$ by
\begin{equation}\label{4.7}
\frac 1{l_1}:=\frac{p-2}{\oversomega}+\frac 1m.
\end{equation}
Moreover we note that, when $p>\romega$ then, by
\eqref{somegagammarelazioni2} (which holds when $\romega\ge 4$),
$\romega<\somega<m$ so, since $u_0,v_0\in L^{\somega}(\Omega)$,
$u(t)=u_0+\int_0^t u_t$, $v(t)=v_0+\int_0^t v_t$ in $L^2(\Omega)$
and $u_t,v_t\in L^m(0,T;L^m(\Omega))$ by (III), we have $u,v\in
C([0,T];L^{\somega}(\Omega))$. Then, using Sobolev embedding when
$p\le\romega$, in general we get
\begin{equation}\label{4.8}
u,v\in C([0,T];L^{\oversomega}(\Omega)).
\end{equation}
From  (FG1)$'$ and (F2)$'$ and well--known
continuity results for Nemitskii operators (see
\cite[Theorem~2.2,~p.16]{ambrosettiprodi}), we then get
\begin{equation}\label{4.9}
\begin{gathered}
\widehat{{\mathfrak{f}}}(u),\, \widehat{{\mathfrak{f}}}(v) \in
C([0,T];L^{\oversomega/(p-1)}(\Omega)),\quad
\widehat{{\mathfrak{f}_u}}(u),\, \widehat{{\mathfrak{f}_u}}(v) \in
C([0,T];L^{\oversomega/(p-2)}(\Omega)),\\
\widehat{{\mathfrak{f}_{uu}}}(u),\,\widehat{{\mathfrak{f}_{uu}}}(v) \in
C([0,T];L^{\oversomega/(p-3)}(\Omega)).
\end{gathered}
\end{equation}
Moreover, being $u,v\in H^1((0,T)\times\Omega)$ we have
$[\widehat{{\mathfrak{f}}}(u)]_t=\widehat{{\mathfrak{f}}}_u(u)u_t$ and
$[\widehat{{\mathfrak{f}}}(v)]_t=\widehat{{\mathfrak{f}}}_u(v)v_t$ a.e. in
$(0,T)\times\Omega$, so being $u_t,v_t\in L^m(0,T;L^m(\Omega))$ by
\eqref{4.7} and \eqref{4.9} we have $[\widehat{{\mathfrak{f}}}(u)]_t,
[\widehat{{\mathfrak{f}}}(v)]_t\in L^m(0,T; L^{l_1}(\Omega))$. By
\eqref{4.6}--\eqref{4.7} we have $\oversomega'\le l_1$, so
$[\widehat{{\mathfrak{f}}}(u)]_t, [\widehat{{\mathfrak{f}}}(v)]_t\in L^m(0,T;
L^{\oversomega'}(\Omega))$. Hence, setting
$X_1=L^{\oversomega}(\Omega)$ and
$X_2=L^{\min\{m,\oversomega\}}(\Omega)$, by \eqref{4.6} we have
$\widehat{{\mathfrak{f}}}(u),\widehat{{\mathfrak{f}}}(v)\in L^2(0,T; X_2')$, $w\in
L^2(0,T;X_1)$, $[\widehat{{\mathfrak{f}}}(u)]_t, [\widehat{{\mathfrak{f}}}(v)]_t\in L^m(0,T;
X_1')$ and $w_t\in L^2(0,T;X_2)$, so by  Lemma~\ref{lemma4.2} we can
integrate by parts with respect to $t$  in \eqref{4.3} to get, for
$t\in [0,T]$,
\begin{equation}\label{4.10}
I_{\mathfrak{f}}(t)=\left.\int_\Omega[\widehat{{\mathfrak{f}}}(u)-\widehat{{\mathfrak{f}}}(v)]w\right|_0^t-\int_0^t\int_\Omega
\widehat{{\mathfrak{f}}}_u(u)w w_t-\int_0^t\int_\Omega
[\widehat{{\mathfrak{f}}}_u(u)-\widehat{{\mathfrak{f}}}_u(v)]wv_t.
\end{equation}
Now, since $m,\oversomega>2$ we can set $l>1$  by
\begin{equation}\label{4.10bis}
\frac 1l=\frac 1 \oversomega+\frac 1m.
\end{equation}
Since $p\le \oversomega$ and, by \eqref{R}, $p\le 1+\oversomega/m'$,
we have
\begin{equation}\label{4.11}
\left[\min\left\{\oversomega/2,l
\right\}\right]'=\max\left\{\left(\oversomega /
2\right)',l'\right\}\le \oversomega/(p-2).
\end{equation}
Moreover, using \eqref{4.6} again, we can set $l_2>1$ by
\begin{equation}\label{4.12}
\frac 1{l_2}=\frac {p-3}\oversomega+\frac 1m.
\end{equation}
Since $[\widehat{{\mathfrak{f}}}_u(u)]_t=\widehat{{\mathfrak{f}}}_{uu}(u)u_t$ and
$[\widehat{{\mathfrak{f}}}_u(v)]_t=\widehat{{\mathfrak{f}}}_{uu}(v)v_t$ a.e. in
$(0,T)\times\Omega$, by \eqref{4.9} and \eqref{4.12} we have
$[\widehat{{\mathfrak{f}}}_u(u)]_t, [\widehat{{\mathfrak{f}}}_u(v)]_t\in L^m(0,T;
L^{l_2}(\Omega))$. Since $p\le 1+\oversomega/m'$ we also have
$(p-3)/\oversomega+1/m\le (\oversomega-2)/\oversomega$, hence by
\eqref{4.12} we get $(\oversomega/2)'=\oversomega/(\oversomega-2)\le
l_2$ and consequently $[\widehat{{\mathfrak{f}}}_u(u)]_t, [\widehat{{\mathfrak{f}}}_u(v)]_t\in
L^m(0,T; L^{(\oversomega/2)'}(\Omega))$. Moreover, by \eqref{4.9}
and \eqref{4.11} we have $\widehat{{\mathfrak{f}}}_u(u), \widehat{{\mathfrak{f}}}_u(v)\in
C([0,T]; L^{(\min\{\oversomega/2,l\})'}(\Omega))$. By Sobolev
embedding $w^2\in C([0,T];L^{\oversomega/2}(\Omega))$ and, as
$u_t,v_t\in L^m(0,T;L^m(\Omega))$, $(w^2)_t=2ww_t\in
L^m(0,T;L^l(\Omega) )\hookrightarrow L^m(0,T;
L^{\min\{l,\oversomega/2\}}(\Omega))$ by \eqref{4.10bis}. From
previous considerations we can apply Lemma~\ref{lemma4.2} with
$X_1=L^{\oversomega/2}(\Omega)$ and
$X_2=L^{\min\{l,\oversomega/2\}}(\Omega)$ and integrate by parts
with respect to $t$ once again in the second addendum of
\eqref{4.10} to get the final form of $I_{\mathfrak{f}}(t)$ suitable for our
estimate, that is, for $t\in [0,T]$,
\begin{multline*}
I_{\mathfrak{f}}(t)=\left.\int_\Omega[\widehat{{\mathfrak{f}}}(u)-\widehat{{\mathfrak{f}}}(v)]w\right|_0^t-\left.\frac
12 \int_\Omega \widehat{{\mathfrak{f}}}_u(u)w^2\right|_0^t\\+\frac
12\int_0^t\int_\Omega \widehat{{\mathfrak{f}}}_{uu}(u)w^2u_t-\int_0^t\int_\Omega
[\widehat{{\mathfrak{f}}}_u(u)-\widehat{{\mathfrak{f}}}_u(v)]wv_t.
\end{multline*}
By (FG1), (FG1)$'$, (F2), (F2)$'$ we then derive
 the preliminary estimate
\begin{equation}\label{4.14}
\begin{aligned}
I_{\mathfrak{f}}(t)\le
&2\mathfrak{c_p}'\int_\Omega(1+|u_0|^{p-2}+\!|v_0|^{p-2})w_0^2+2\mathfrak{c_p}'\|w(t)\|_2^2\\+&2\mathfrak{c_p}'\int_\Omega\!\!(|u(t)|^{p-2}+|v(t)|^{p-2})w^2(t)
+\mathfrak{c_p}''\int_0^t\int_\Omega
(|u_t|+|v_t|)w^2\\+&\mathfrak{c_p}''\int_0^t\int_\Omega
(|u|^{p-3}+|v|^{p-3})(|u_t|+|v_t|)w^2
\end{aligned}
\end{equation}
for $t\in [0,T]$. In the sequel we shall estimate the addenda in the
right--hand side of \eqref{4.14}, denoting by $I^i_{\mathfrak{f}}(t)$ the $i$--th
among them, for $i=1,\ldots,5$.

\noindent{\bf Estimate of $\boldsymbol{I_{\mathfrak{f}}^1(t)}$.} Denoting
$\nu=\oversomega/(\oversomega-p+2)>1$, so $\nu'=\oversomega/(p-2)$,
by H\"{o}lder inequality we have $\int_\Omega|u_0|^{p-2}w_0^2\le
\|u_0\|_{\oversomega}^{p-2}\|w_0\|_{2\nu}^2$. Since
$\oversomega=\max\{\romega,\somega\}$, by \eqref{somegagamma} we
have $\nu'\ge \romega/(\romega-2)$, hence $2\nu\le \romega$ so as
$\romega<\infty$ by Sobolev embedding and \eqref{4.1}
$\int_\Omega|u_0|^{p-2}w_0^2\le k_4 R^{p-2}\|W_0\|^2_\cal{H}$.
Estimating in the same way $\int_\Omega|v_0|^{p-2}w_0^2$
\begin{equation}\label{4.15}
I_{\mathfrak{f}}^1(t)\le k_5 \mathfrak{c_p}'(1+R^{p-2})\|W_0\|^2_\cal{H}.
\end{equation}
\noindent{\bf Estimate of $\boldsymbol{I_{\mathfrak{f}}^2(t)}$.} By the trivial
estimate
\begin{equation}\label{4.15bis}
\|w(t)\|_2^2\le 2\|w_0\|_2^2+2t\int_0^t \|w_t\|_2^2\le
2\|w_0\|_2^2+2\int_0^t \|w_t\|_2^2
\end{equation}
where $0\le t\le T\le 1$ was used, we get
\begin{equation}\label{4.16}
I_{\mathfrak{f}}^2(t)\le k_6 \mathfrak{c_p}'\left(\|W_0\|^2_\cal{H}+\int_0^t
\|W(\tau)\|_{\cal{H}}^2\,d\tau\right).
\end{equation}
\noindent{\bf Estimate of $\boldsymbol{I_{\mathfrak{f}}^4(t)}$.} Since
$\romega\ge 4$ so $\romega/(\romega-2)\le 2$, by
H\"{o}lder inequality, Sobolev embedding and \eqref{4.1} we
immediately get
\begin{equation}\label{4.17}
I_{\mathfrak{f}}^4(t)\le \mathfrak{c_p}''\int_0^t
\|w\|_{\romega}^2\left(\|u_t\|_{\frac\romega{\romega-2}}+\|v_t\|_{\frac\romega{\romega-2}}\right)\le
k_7\mathfrak{c_p}''R\int_0^t \|W(\tau)\|^2_{\cal{H}}\,d\tau.
\end{equation}
\noindent{\bf Estimate of $\boldsymbol{I_{\mathfrak{f}}^5(t)}$.} By
H\"{o}lder inequality with exponents
$\romega/(p-3)$, $\romega/(\romega-p+1)$, $\romega/2$,
 and Sobolev embedding, since by \eqref{R} we have $m=m_p\ge \romega/(\romega-p+1)$,
\begin{equation}\label{4.18}
\begin{aligned} I_{\mathfrak{f}}^5(t)\le &\mathfrak{c_p}''\int_0^t
(\|u\|_{\romega}^{p-3}+\|v\|_{\romega}^{p-3})\left(\|u_t\|_{\frac\romega{\romega-p+1}}+\|v_t\|_{\frac\romega{\romega-p+1}}\right)\|w\|_\romega^2\,d\tau\\
\le & k_8\mathfrak{c_p}''R^{p-3}\int_0^t
\left(\|u_t\|_{m_p}+\|v_t\|_{m_p}\right) \|W(\tau)\|_{\cal{H}}^2\,d\tau.
\end{aligned}
\end{equation}
\noindent{\bf Estimate of $\boldsymbol{I_{\mathfrak{f}}^3(t)}$.} We shall
distinguish between the two subcases $p<\romega$ and $p\ge \romega$.
In the first one, by H\"{o}lder inequality with conjugate exponents
$\romega/(p-2)$ and $\romega/(\romega-p+2)$ and \eqref{4.1} we get
$$\int_\Omega|u(t)|^{p-2}w^2(t)\le
\|u(t)\|_{\romega}^{p-2}\|w(t)\|^2_{2\romega/(\romega-p+2)}\le
k_9R^{p-2}\|w(t)\|^2_{2\romega/(\romega-p+2)}.$$ Since $3<p<\romega$ we
have $2\romega/(\romega-p+2)\in (2,\romega)$ and then, by
interpolation and weighted Young inequalities we get, for any
$\eps>0$,
$$\int_\Omega|u(t)|^{p-2}w^2(t)\le
k_9R^{p-2}\|w(t)\|_2^{2\theta_1}\|w(t)\|_{\romega}^{2(1-\theta_1)}\le
k_9R^{p-2}\left(\frac 1\eps\|w(t)\|_2^2+\eps\|w(t)\|_{\romega}^2\right),$$
where $\theta_1\in (0,1)$ is given by
$\frac{\romega-p+2}{2\romega}=\frac {\theta_1} 2+\frac
{1-\theta_1}{\romega}$,
 and consequently, using \eqref{4.15bis},
$$\int_\Omega|u(t)|^{p-2}w^2(t)\le \frac 14 k_{10}R^{p-2}\left(\frac 1\eps \|W_0\|_{\cal{H}}
+\frac 1\eps \int_0^t \|W(\tau)\|_{\cal{H}}^2\,d\tau+\eps
\|W(t)\|_{\cal{H}}^2\right).$$ By estimating the
term $\int_\Omega|v|^{p-2}w^2$ in the same way we get our estimate
for $I_{\mathfrak{f}}^3(t)$ in case $1+\romega/2<p<\romega$, that is
\begin{equation}\label{4.20}
I_{\mathfrak{f}}^3(t)\le k_{10} \mathfrak{c_p}'R^{p-2}\left(\frac 1\eps \|W_0\|^2_{\cal{H}}
+\frac 1\eps \int_0^t \|W(\tau)\|_{\cal{H}}^2\,d\tau+\eps
\|W(t)\|_{\cal{H}}^2\right).
\end{equation}
We now consider the subcase $p\ge\romega$.  Trivially
\begin{equation}\label{4.21}
\int_\Omega|u(t)|^{p-2}w^2(t)\le2^{p-3}\left(\int_\Omega|u(t)-u_0|^{p-2}w^2(t)+\int_\Omega|u_0|^{p-2}w^2(t)\right).
\end{equation}
We shall estimate separately the two addenda inside brackets in
\eqref{4.21}. For the first one we use H\"{o}lder inequality with
conjugate exponents $\romega/(\romega-2)$ and $\romega/2$ to get,
recalling that $\somega=\romega(p-2)/(\romega-2)$ in this case,
$$\int_\Omega|u(t)-u_0|^{p-2}w^2\le
\|u(t)-u_0\|_\somega^{p-2}\|w(t)\|_\romega^2\le
k_{11}\left(\int_0^t\|u_t\|_\somega\right)^{p-2}\|W(t)\|^2_{\cal{H}}.
$$
Consequently, since by \eqref{somegagammarelazioni2} we have
$\somega\le m$,
$$\int_\Omega|u(t)-u_0|^{p-2}w^2(t)\le
k_{12}\left(\int_0^t\|u_t\|_m\right)^{p-2}\|W(t)\|^2_{\cal{H}}.
$$
Then using H\"{o}lder inequality in time, \eqref{R} and
Remark~\ref{remark3.3} we get
$$
\int_\Omega|u(t)-u_0|^{p-2}w^2\le
k_{12}t^{(p-2)(m-1)/m}\left(\int_0^t\|[u_t]_\alpha\|_{\overline{m},\alpha}^{\overline{m}}\right)^{(p-2)/m}\|W(t)\|^2_{\cal{H}}.
$$
Since $0\le t\le T\le 1$,  $p\ge \romega\ge 4$ and $m>2$, by
\eqref{4.1} we get
\begin{equation}\label{4.22}
\int_\Omega|u(t)-u_0|^{p-2}w^2\le k_{12}R^{p-2}t\,\|W(t)\|_{\cal{H}}^2.
\end{equation}
We now estimate the second addendum inside brackets in \eqref{4.21}
when $u_0\in L^{s_1}(\Omega)$,  $s_1>\somega$ and \eqref{4.4} holds.
Using H\"{o}lder inequality with conjugate exponents $s_1/(p-2)$,
$s_1/(s_1-p+2)$ and  \eqref{4.4}
$$\int_\Omega|u_0|^{p-2}w^2\le
\|u_0\|_{s_1}^{p-2}\|w(t)\|^2_{2s_1/(s_1-p+2)}\le
R^{p-2}\|w(t)\|^2_{2s_1/(s_1-p+2)}.$$ Since
$s_1/(p-2)>\somega/(p-2)=\romega/(\romega-2)$ we have
$2s_1/(s_1-p+2)\in (2,\romega)$, hence by interpolation and weighted
Young inequalities, for $\eps>0$ we get
$$\int_\Omega|u_0|^{p-2}w^2(t)\le R^{p-2}\|w(t)\|_2^{2\theta_2}\|w(t)\|_{\romega}^{2(1-\theta_2)}
\le R^{p-2}\left(\frac
1\eps\|w(t)\|_2^2+\eps\|w(t)\|_{\romega}^2\right),$$ where $\theta_2\in
(0,1)$ is given by $\frac{s_1-p+2}{2s_1}=\frac {\theta_2} 2+\frac
{1-\theta_2}{\romega}$. Consequently by \eqref{4.15bis} we get
\begin{equation}\label{4.23}
\begin{aligned}
\int_\Omega|u_0|^{p-2}w^2(t)\le & R^{p-2}\left[\frac 2\eps
\left(\|w_0\|_2^2+\int_0^t
\|w_t\|_2^2\right)+\eps\|w(t)\|_\romega^2\right]\\
\le & k_{13}R^{p-2}\left( \frac 1\eps \|W_0\|_\cal{H}^2+\frac 1\eps
\int_0^t \|W(\tau)\|_\cal{H}^2\,d\tau +\eps \|W(t)\|_{\cal{H}}^2\right).
\end{aligned}
\end{equation}
In the general case $u_0\in L^\somega(\Omega)$ a different argument
is needed. Since $L^\infty(\Omega)$ is dense in $L^\somega(\Omega)$,
in correspondence to $\eps>0$ there is
$\overline{u}_0=\overline{u_0}(\eps,u_0)\in L^\infty(\Omega)$ such
that $\|\overline{u_0}-u_0\|_\somega\le \eps^{1/(p-2)}$. Then, using
H\"{o}lder inequality with conjugate exponents
$\romega/(\romega-2)$, $\romega/2$ and \eqref{4.15bis} we get
\begin{equation}\label{4.24}
\begin{aligned}
\int_\Omega|u_0|^{p-2}w^2(t) \le
&2^{p-3}\left(\int_\Omega|u_0-\overline{u_0}|^{p-2}w^2(t)+\int_\Omega|\overline{u_0}|^{p-2}w^2(t)\right)\\
\le &
2^{p-3}\left(\|\overline{u_0}-u_0\|_\somega^{p-2}\|w(t)\|^2_\romega+\|\overline{u_0}\|_\infty^{p-2}\|w(t)\|_2^2\right)\\
\le &
2^{p-3}\left[\eps\|w(t)\|^2_\romega+\|\overline{u_0}\|_\infty^{p-2}\left(2\|w_0\|_2^2+2\int_0^t\|w_t\|_2^2\right)\right]\\
\le &
K_7(\eps,u_0)\left(\|W_0\|^2_\cal{H}+\int_0^t\|W(\tau)\|_\cal{H}^2\,d\tau\right)+k_{14}\eps
\|W(t)\|_\cal{H}^2.
\end{aligned}
\end{equation}
 Comparing \eqref{4.23} and \eqref{4.24}
there are $K_8=K_8(\eps,R,u_0)$ and  $K_9=K_9(R)$,
increasing in $R$, such that
\begin{equation}\label{4.25}
\int_\Omega|u_0|^{p-2}w^2(t) \le K_8\left(\|W_0\|_\cal{H}^2+\int_0^t
\|W(\tau)\|_\cal{H}^2\,d\tau\right)+\eps K_9\|W(t)\|_\cal{H}^2,
\end{equation}
with $K_8$ independent on $u_0$ when $u_0\in L^{s_1}(\Omega)$,
$s_1>\somega$ and \eqref{4.4} holds. Plugging \eqref{4.22} and
\eqref{4.25} in \eqref{4.21} and using exactly the same arguments to
estimate the term $\int_\Omega |v|^{p-2}w^2$ we then get our final
estimate for $I^3_{\mathfrak{f}}(t)$ when $p\ge \romega$, that is
\begin{equation}\label{4.26}
I_{\mathfrak{f}}^3(t)\le \mathfrak{c_p}' K_{10}\left(\|W_0\|_\cal{H}^2+\int_0^t
\|W(\tau)\|_\cal{H}^2\,d\tau\right)+\mathfrak{c_p}'K_{11}(\eps+t)\|W(t)\|_\cal{H}^2
\end{equation}
with $K_{10}=K_{10}(\eps,R,u_0,v_0)$, $K_{11}=K_{11}(R)$ increasing
in $R$, $K_{10}$ being independent on $u_0$, $v_0$ when $u_0,v_0\in
L^{s_1}(\Omega)$, $s_1>\somega$ and \eqref{4.4} holds. Comparing
\eqref{4.26} with \eqref{4.20} and possibly changing the values of
$K_{10}$ and $K_{11}$ we get that actually \eqref{4.26} is our final
estimate for $I^3_{\mathfrak{f}}(t)$ for any $p>1+\romega/2$.

By plugging \eqref{4.15}--\eqref{4.18} and \eqref{4.26} in
\eqref{4.14} we get \eqref{4.3} when
$p>1+\romega/2$.
\end{proof}

A trivial transposition of the arguments used in the proof of
Lemma~\ref{lemma4.1} allows to prove the following $\Gamma$ --
version of the estimate.
\begin{lem}\label{lemma4.4} Suppose that
$\mathfrak{g}$ satisfies (G1) with constants  $\mathfrak{c_q}$, $\mathfrak{c_q}'$, that
\eqref{R} holds and
that either $2\le q\le 1+\rgamma/2$ or $q>1+\rgamma/2$, $N\le 5$ and  $\mathfrak{g}$
satisfies (F2) with constant $\mathfrak{c_p}''$. Let $T\in (0,1]$, $U=(u,\dot{u}),
V=(v,\dot{v})\in C([0,T];\cal{H})$, $\dot{u},\dot{v}\in Z(0,T)$ and denote
$W=U-V=(w,\dot{w})$, $U_0=U(0)=(u_0,u_1)$, $V_0=V(0)=(v_0,v_1)$,
$W_0=U_0-V_0=(w_0,w_1)$, $\boldsymbol{\mathfrak{c_q}}=(\mathfrak{c_q},\mathfrak{c_q}')$,
$\boldsymbol{\mathfrak{c_q}'}=(\mathfrak{c_q},\mathfrak{c_q}',\mathfrak{c_q}'')$. Suppose moreover that
$u_0,v_0\in L^\sgamma(\Gamma_1)$ and take $R\ge 0$ such that
\begin{equation}\label{4.1bis}
\|U\|_{C([0,T];\cal{H})}, \|V\|_{C([0,T];\cal{H})},
\|\dot{u}\|_{Z(0,T)},\|\dot{v}\|_{Z(0,T)},\,\|u_0\|_{\sgamma,\Gamma_1}
\|v_0\|_{\sgamma,\Gamma_1}\le R.
\end{equation}
Then given any $\eps>0$ there are
$$K_{12}=K_{12}(R,\boldsymbol{\mathfrak{c_q}}),\qquad\text{and}\quad K_{13}=K_{13}(\eps,R,u_0,v_0,\boldsymbol{\mathfrak{c_q}}'),$$
independent on $\mathfrak{g}$ and increasing in $R$, such that for all $t\in
[0,T]$
\begin{multline}\label{4.3bis}
\int_0^t\int_{\Gamma_1}[\widehat{\mathfrak{g}}(u)-\widehat{\mathfrak{g}}(v)]{w_{|\Gamma}}_t
\le K_{12}(\eps+t)\|W(t)\|^2_\cal{H} \\+
K_{13}\left[\|W_0\|^2_\cal{H}+\int_0^t(1+\|{u_{|\Gamma}}_t\|_{\mu_q}+\|{v_{|\Gamma}}_t\|_{\mu_q})\|W(\tau)\|^2_\cal{H}\,d\tau\right].
\end{multline}
Moreover, if $u_0,v_0\in L^{s_2}(\Gamma_1)$ for some $s_2>\sgamma$
and, in addition to \eqref{4.1bis}, we have
\begin{equation}\label{4.4bis}
\|{u_0}_{|\Gamma}\|_{s_2,\Gamma_1},\quad
\|{v_0}_{|\Gamma}\|_{s_2,\Gamma_1}\le R,
\end{equation}
then $K_{13}$ is  independent on $u_0$ and $v_0$, that is
$K_{13}(\eps,R,u_0,v_0,\boldsymbol{\mathfrak{c_q}}')=K_{14}(\eps,R,\boldsymbol{\mathfrak{c_q}}')$.
\end{lem}
\section{Local existence}\label{section5}
This section is devoted to our local existence result for
problem \eqref{1.1}, that is
\begin{thm}[\bf Local existence]\label{theorem5.1} Suppose that (PQ1--3), (FG1--2) and
(FGQP1) hold. Then all conclusions of Theorem~\ref{theorem1.1} hold
true when problem \eqref{1.2} is generalized to problem \eqref{1.1},
provided the energy identity \eqref{energyidentityreduced} is
generalized to \eqref{energyidentity}.
\end{thm}
To prove Theorem~\ref{theorem5.1} we approximate, following a
procedure from \cite{bociulasiecka1}, problem \eqref{1.1} with a
sequence of problems involving subcritical sources.
We start by introducing a suitable cut--off sequence. At first we fix
\begin{footnote}{$\eta_1$ is easily built as follows. Let
$\eta_0\in W^{1,\infty}(\R)$ be defined by $\eta_0(\tau)=1$ for
$|\tau|\le 5/4$, $\eta_0(\tau)=0$ for $|\tau|\ge 7/4$, linear for
$5/4\le |\tau|\le 7/4$, and $(\rho_n)_n$ the standard mollifying
sequence in $\R$ defined at  \cite[p. 108]{brezis2}. Then
$\eta_1=\rho_4*\eta_0$ satisfies the required properties. Indeed
$\|\eta_0'\|_\infty=2$ so $\|\eta_1'\|_\infty\le 2$. Moreover
$\rho_4(x)=4\rho_1(4x)/\int_{-1}^1\rho_1$, where
$\rho_1(x)=e^{1/(x^2-1)}$ in $(-1,1)$, vanishing outside,
$\|\rho_1'\|_\infty\le 2\max_{y\ge 0}y^2e^{-y}=8e^{-2}$ and
$\int_{-1}^1\rho_1\ge 2\int_0^{1/\sqrt{2}}\rho_1>e^{-2}$ so
$\|\rho_4'\|_\infty\le 2^7$ and consequently $\|\eta_1''\|_\infty\le
\|\rho_4'\|_\infty\|\eta_0'\|_1=2\|\rho_4'\|_\infty\le
2^8$.}\end{footnote} $\eta_1\in C^\infty_c(\R)$ such that
$\eta_1\equiv 1$ in $[-1,1]$, $\text{supp }\eta_1\subset [-2,2]$,
$0\le\eta_1\le 1$, $\|\eta_1'\|_\infty\le 2$ and
$\|\eta_1''\|_\infty\le 2^8$. Then we set the sequence $(\eta_n)_n$
by $\eta_n(x)=\eta_1(x/n)$. For all $n\in\N$ we have
\begin{equation}\label{5.1}
\begin{alignedat}3
& \eta_n\in C^\infty_c(\R), \quad &&\eta_n\equiv
1\,\text{in}\,[-n,n],\quad &&\text{supp }\eta_n\subset [-2n,2n],\\
&0\le\eta_n\le 1,\quad &&\|\eta_n'\|_\infty\le 2/n,\quad
&&\|\eta_n''\|_\infty\le 2^8/n^2.
\end{alignedat}
\end{equation}
We then define, for $f$ and $g$ satisfying (FG1) and $n\in\N$, the
truncated nonlinearities $f^n$ and $g^n$ by setting, for all
$u\in\R$, a.a. $x\in\Omega$ and $y\in\Gamma_1$,
\begin{equation}\label{5.2}
f^n(x,u)=\eta_n(u)f(x,u), \qquad g^n(y,u)=\eta_n(u)g(y,u).
\end{equation}
By (FG1), Remark~\ref{remark3.2} and \eqref{5.1} for each $n\in\N$
we trivially have
\begin{equation}\label{5.3}
\begin{alignedat}2
&|f^n(\cdot,0)|\le c_p,\quad |f^n_u(\cdot,u)|\le
&&\left[2c_p'(1+|u|^{p-2})+2c_p(1+|u|^{p-1})/n\right]\chi_{A_n}(u)\\
&|g^n(\cdot,0)|\le c_q,\quad |g^n_u(\cdot,u)|\le
&&\left[2c_q'(1+|u|^{q-2})+2c_q(1+|u|^{q-1})/n\right]\chi_{A_n}(u)
\end{alignedat}
\end{equation}
where $\chi_{A_n}$ denotes the characteristic function of
$A_n=[-2n,2n]$, hence $f^n$ and $g^n$ satisfy assumptions (FG1)
with exponents $p=q=2$ and  constants  dependent on $n$. Then, by
\cite[Theorem~3.1]{Dresda1} for each $U_0\in\cal{H}$ and $n\in\N$
the approximating problem
\begin{equation}\label{1.1n}
\begin{cases} u^n_{tt}-\Delta u^n+P(x,u^n_t)=f^n(x,u^n) \qquad &\text{in
$(0,\infty)\times\Omega$,}\\
u^n=0 &\text{on $(0,\infty)\times \Gamma_0$,}\\
u^n_{tt}+\partial_\nu u^n-\Delta_\Gamma
u^n+Q(x,u^n_t)=g^n(x,u^n)\qquad &\text{on
$(0,\infty)\times \Gamma_1$,}\\
u(0,x)=u_0(x),\quad u_t(0,x)=u_1(x) &
 \text{in $\overline{\Omega}$}
\end{cases}
\end{equation}
has a unique weak solution $u^n$  with $U^n\in
C([0,\infty);\cal{H})$.
The strategy of the proof of Theorem~\ref{theorem5.1} is to pass to
the limit as $n\to\infty$ in \eqref{1.1n}. With this aim we point
out the following uniform estimates on $f^n$,  $g^n$ and  the
Nemitskii operators $\widehat{f^n}$ and $\widehat{g^n}$ associated
with them.
\begin{lem}\label{lemma5.1}
Let (PQ1--3), (FG1) and  (FGQP1) hold. Then:
\renewcommand{\labelenumi}{{(\roman{enumi})}}
\begin{enumerate}
\item for all $n\in\N$ the couples $(f^n,g^n)$ satisfy (FG1)
with constants
$$
\overline{\boldsymbol{c}_p}=(\overline{c_p},\overline{c_p}')=(c_p,
4(c_p+c_p'))\quad
\text{and}\quad\overline{\boldsymbol{c}_q}=(\overline{c_q},\overline{c_q}')=(c_q,
4(c_q+c_q'));
$$
\item for any  $\rho\in [p-1,\infty)$ and  $\theta\in
[q-1,\infty)$, $\widehat{f^n}: L^\rho(\Omega)\to
L^{\rho/(p-1)}(\Omega)$, $\widehat{f^n}: H^1(\Omega)\to
L^{m_p'}(\Omega)$, $\widehat{g^n}: L^\theta(\Gamma_1)\to
L^{\theta/(q-1)}(\Gamma_1)$ and $\widehat{g^n}: H^1(\Gamma)\cap
L^2(\Gamma_1)\to L^{\mu_q'}(\Gamma_1)$ are  locally Lipschitz and
bounded, uniformly in $n\in\N$, and for any $R\ge 0$ we
have
\begin{equation}\label{5.9}
\begin{alignedat}2
&\|\widehat{f^n}(u)\|_{\rho/(p-1)} &&\le\overline{c_p} k_1
(1+R^{p-1}),\\
&\|\widehat{g}(\widetilde{u})\|_{\theta/(q-1),\Gamma_1}&&\le\overline{c_q}
k_2
(1+R^{q-1}),\\
&\|\widehat{f^n}(u)-\widehat{f^n}(v)\|_{\rho/(p-1)}
&&\le\overline{c_p'}k_1(1+R^{p-2})\|u-v\|_\rho,\\
&\|\widehat{g^n}(\widetilde{u})-\widehat{g^n}(\widetilde{v})\|_{\theta/
(q-1),\Gamma_1} && \le\overline{c_q'}k_2
(1+R^{q-2})\|\widetilde{u}-\widetilde{v}\|_{\theta,\Gamma_1}
\end{alignedat}
\end{equation}
provided $\|u\|_\rho, \|v\|_\rho,
\|\widetilde{u}\|_{\theta,\Gamma_1},
\|\widetilde{v}\|_{\theta,\Gamma_1}\le R$, and
\begin{equation}\label{5.10}
\begin{alignedat}2
&\|\widehat{f^n}(u)\|_{m_p'}&&\le \overline{c_p} k_3 (1+R^{p-1}),\\
&\|\widehat{g^n}(\widetilde{u})\|_{\mu_q',\Gamma_1}&&\le
\overline{c_q}
k_3 (1+R^{q-1}),\\
&\|\widehat{f^n}(u)-\widehat{f^n}(v)\|_{m_p'}&&\le \overline{c_p'}k_3 (1+R^{p-2})\|u-v\|_{H^1(\Omega)},\\
&\|\widehat{g^n}(\widetilde{u})-\widehat{g^n}(\widetilde{v})\|_{\mu_q',\Gamma_1}&&\le
\overline{c_q'}
k_3(1+R^{q-2})\|\widetilde{u}-\widetilde{v}\|_{H^1(\Gamma)},
\end{alignedat}
\end{equation}
provided $\|u\|_{H^1(\Omega)}, \|v\|_{H^1(\Omega)},
\|\widetilde{u}\|_{H^1(\Gamma)}, \|\widetilde{v}\|_{H^1(\Gamma)}\le
R$;
\item when $f$ and $g$ satisfy also  (FG2) or (FG2)$'$
then $f^n$ and $g^n$ satisfy the same assumption with constants
\begin{equation}\label{5.7}
\overline{c_p}''=2^{10} c_p+16c_p'+2c_p'',\quad \text{and}\quad
\overline{c_q}''=2^{10} c_q+16c_q'+2c_q'',
\end{equation}
hence $(f^n,g^n)$ satisfy (FG1--2) or (FG1)--(FG2)$'$ with constants
$\overline{\boldsymbol{c}_p}'=(\overline{\boldsymbol{c}_p},\overline{c_p}'')$
and
$\overline{\boldsymbol{c}_q}'=(\overline{\boldsymbol{c}_q},\overline{c_q}'')$,
independent on $n\in\N$.
\end{enumerate}
\end{lem}
\begin{proof} We first note that since  $\chi_{A_n}(u)|u|\le 2n$, from
\eqref{5.3} we also get $|f^n_u(\cdot,u)|\le
2(c_p+c_p')+4(c_p+c_p')|u|^{p-2}$ and  $|g^n_u(\cdot,u)|\le
2(c_q+c_q')+4(c_q+c_q')|u|^{q-2}$, and by combining them with
\eqref{5.3} and Remark~\ref{remark3.2} we  complete the proof of
(i). By combining Lemma~\ref{Nemitskii} with (i) we immediately
derive (ii). To prove (iii) we note that, when (F2)  holds
true we have, by \eqref{5.1},
$$|f^n_{uu}(\cdot,u)|\le \left[\tfrac{2^8}{n^2}
|f(\cdot,u)|+\tfrac 4n
|f_u(\cdot,u)|+|f_{uu}(\cdot,u)|\right]\chi_{A_n}(u)$$ and
consequently, using (FG1)$'$ and (F2)$'$, see Remarks~\ref{remark3.2}
and \ref{remark3.7}, we get
$$|f^n_{uu}(\cdot,u)|\le \left[\tfrac{2^8}{n^2}
c_p(1+|u^{p-1})+\tfrac 8n c_p'(1+|u|^{p-2})
+2c_p''(1+|u|^{p-3})\right]\chi_{A_n}(u),
$$
and then (F2) follows since $\chi_{A_n}(u)|u|\le 2n$, using
Remark~\ref{remark3.7} again. By the same arguments we get (G2).
\end{proof}
Our first main estimate on the sequence $(u^n)_n$ is  the following
one.
\begin{lem}\label{lemma5.3} If (PQ1--3), (FG1) and (FGQP1) hold
there are a decreasing function $T_1:[0,\infty)\to (0,1]$ and an
increasing one $\kappa:[0,\infty)\to [0,\infty)$, such that
\begin{alignat}3\label{5.11}
&\|U^n\|_{C([0,T_1(\|U_0\|_{\cal{H}})];\cal{H})}\le
&&1+2\|U_0\|_{\cal{H}},\quad&&\text{for all
$n\in\N$,}\\
\label{5.12}&\|\dot{u}^n\|_{Z(0,T_1(\|U_0\|_{\cal{H}}))}\le
&&\kappa(\|U_0\|_{\cal{H}}),\quad&&\text{for all $n\in\N$.}
\end{alignat}
\end{lem}
\begin{proof} We denote
$R=1+2\|U_0\|_{\cal{H}}$. Since, as already noted, $f^n$ and $g^n$
satisfy assumptions (FG1) with exponents $p=q=2$  then, by Lemma~\ref{Nemitskii}, the
Nemitskii operators $\widehat{f^n}: H^1(\Omega)\to L^2(\Omega)$ and
$\widehat{g^n}: H^1(\Gamma)\cap L^2(\Gamma_1)\to L^2(\Gamma_1)$ are,
possibly not uniformly in $n$, locally Lipschitz. Hence we can
introduce, as in \cite{bociulasiecka1,chueshovellerlasiecka}, their
globally Lipschitz truncations $\widehat{f^n_R}: H^1(\Omega)\to
L^2(\Omega)$ and $\widehat{g^n_R}: H^1(\Gamma)\cap L^2(\Gamma_1)\to
L^2(\Gamma_1)$, given by $\widehat{f^n_R}=\widehat{f^n}\cdot
\Pi_{H^1(\Omega),R}$ and $\widehat{g^n_R}=\widehat{g^n}\cdot
\Pi_{H^1(\Gamma)\cap L^2(\Gamma_1),R}$, where in any Hilbert space
$H$ we denote by $\Pi_{H,R}:H\to B_R(H)$ the projection onto the
ball $B_R(H)$ of radius $R$ centered at $0$ in $H$, given for $u\in
H$ (see \cite[Theorem~5.2]{brezis2}) by
$$\Pi_{H,R}(u)=\begin{cases}
u\,&\text{if $\|u\|_H\le R$,}\\
Ru/\|u\|_H\,&\text{otherwise.}
\end{cases}$$
Then the operator couple $F^n_R=(\widehat{f^n_R},\widehat{g^n_R})$
is globally Lipschitz from $H^1$ to $H^0$, and consequently by
applying \cite[Theorem~3.2]{Dresda1} the abstract Cauchy problem
\begin{equation}\label{5.13}
\begin{cases} \ddot{v}^n+Av^n+B(\dot{v}^n)=F^n_R(v^n)\qquad\text{in $X'$,}\\
v^n(0)=u_0,\quad \dot{v}^n(0)=u_1
\end{cases}
\end{equation}
 has a unique global weak solution
$v^n\in C([0,\infty);H^1)\cap C^1([0,\infty);H^0)$, which (see
\cite[Remark~3.6]{Dresda1}) satisfies the energy identity
\begin{multline}\label{5.14}
\left.\tfrac 12\|\dot{v}^n\|_{H^0}^2+\tfrac
12\|v^n\|_{H^1}^2\right|_0^t+\int_0^t\langle
B(\dot{v}^n,\dot{v}^n\rangle_Y\\=\left.\tfrac
12\|v^n\|_{2,\Gamma_1}^2\right|_0^t +\int_0^t \int_\Omega
\widehat{f^n_R}(v^n)v^n_t+\int_0^t \int_{\Gamma_1}
\widehat{g^n_R}(v^n){v^n_{|\Gamma}}_t,\qquad t\in [0,\infty).
\end{multline}
Setting $\eps_1=\min\{1,c_m',c_\mu'\}>0$, by weighted Young
inequality we have
$$\int_0^t \int_\Omega
\widehat{f^n_R}(v^n)v^n_t\le
\tfrac{\eps_1}{m_p}\int_0^t\|v^n_t\|_{m_p}^{m_p}+\tfrac
{\eps_1^{1-m_p'}}{m_p'}\int_0^t\|\widehat{f}_R^n(v^n)\|_{m_p'}^{m_p'}$$
and then, using the definition on $\widehat{f^n_R}$, \eqref{5.10},
Lemma~\ref{lemma5.1} and the fact that $m_p\ge 2$ and $R\ge 1$, we get
$$\int_0^t \int_\Omega
\widehat{f^n_R}(v^n)v^n_t\le
\tfrac{\eps_1}2\int_0^t\|v^n_t\|_{m_p}^{m_p}+k_{14}c_p^{m_p'}(1+R^{2p})\,t,\qquad
\text{for all $t\in [0,\infty)$.}$$ Using the same arguments to
estimate the term $\int_0^t \int_{\Gamma_1}
\widehat{g^n_R}(v^n){v_{|\Gamma}^n}_t$ we obtain
\begin{multline}\label{5.15}
\int_0^t \int_\Omega \widehat{f^n_R}(v^n)v^n_t+\int_0^t
\int_{\Gamma_1} \widehat{g^n_R}(v^n){v_{|\Gamma}^n}_t\\\le
\tfrac{\eps_1}2\int_0^t\left(\|v^n_t\|_{m_p}^{m_p}+\|{v_{|\Gamma}^n}_t\|_{\mu_q}^{\mu_q}\right)+k_{15}(1+R^{2(p+q)})\,t.
\end{multline}
By Young and H\"{o}lder inequality
in time we have
$$\tfrac 12 \|v^n(t)\|_{2,\Gamma_1}^2\!\!\le \tfrac 12
\|v^n(t)\|_{H^0}^2\!\!\le \tfrac 12
\left(\|u_0\|_{H^0}+\!\!\int_0^t\|\dot{v}^n\|_{H^0}\right)^2\!\!\le
\|u_0\|_{H^0}^2+t\int_0^t\|\dot{v}^n(\tau)\|_{H^0}^2\,d\tau,
$$
so, plugging it and \eqref{5.15} in \eqref{5.14} and denoting
$V^n=(v^n, \dot{v}^n)$, we get
\begin{multline*}\tfrac 12
\|V^n(t)\|_{\cal{H}}^2+\int_0^t\langle B(\dot{v}^n),\dot{v}^n\rangle_Y\le
\tfrac 32
\|U_0\|_{\cal{H}}^2+\int_0^t\|V^n(\tau)\|_{\cal{H}}^2\,d\tau\\+\tfrac{\eps_1}2\int_0^t\left(\|v^n_t\|_{m_p}^{m_p}+\|{v_{|\Gamma}^n}_t\|_{\mu_q}^{\mu_q}\right)+k_{15}(1+R^{2(p+q)})\,t
\qquad\text{for $t \in [0,1]$.}
\end{multline*}
 Denoting
\begin{equation}\label{5.16}
\begin{aligned}
\cal{I}^n_f(t)=&c_m'\int_0^t\|[v^n_t]_{\alpha}\|_{m,\alpha}^m+\tfrac
12 \|v^n_t(t)\|_2^2-\tfrac{\eps_1}2\|v^n_t(t)\|_{m_p}^{m_p},\\
\cal{I}^n_g(t)=&c_\mu'\int_0^t\|[{v^n_{|\Gamma}}_t(t)]_{\beta}\|_{\mu,\beta,\Gamma_1}^\mu+\tfrac
12
\|{v^n_{|\Gamma}}_t(t)\|_{2,\Gamma_1}^2-\tfrac{\eps_1}2\|{v^n_{|\Gamma}}_t(t)\|_{\mu_q}^{\mu_q},
\end{aligned}
\end{equation}
and using assumption (PQ3) in previous formula, we get
\begin{equation}\label{5.17}
\tfrac 12 \|V^n(t)\|_{\cal{H}}^2+\cal{I}^n_f(t)+\cal{I}^n_g(t)\le
\tfrac 32 \|U_0\|_{\cal{H}}^2+\tfrac
32\int_0^t\|V^n\|_{\cal{H}}^2\,d\tau+k_{15}(1+R^{2(p+q)})\,t
\end{equation}
for $t \in [0,1]$. Now we remark that, by \eqref{mpmuq}, when $p\le
1+\romega/2$ we have $m_p=2$ so, as $\eps_1\le 1$,
$\cal{I}^n_f(t)\ge c_m'\int_0^t \|[v^n_t]_{\alpha}\|_{m,\alpha}^m$.
When $p> 1+\romega/2$ we have $m_p=m>2$ and, by assumption (FGQP1)
and Remark~\ref{remark3.3}, we have $\|v^n_t\|_m^m\le
\|[v^n_t]_\alpha\|_{m,\alpha}^m$ and then, since $\eps_1\le c_m'$,
$\cal{I}^n_f(t)\ge \frac 12 c_m'\int_0^t
\|[v^n_t]_{\alpha}\|_{m,\alpha}^m$. Using the same arguments to
estimate from below $\cal{I}_g^n(t)$ we then get from \eqref{5.17}
\begin{multline}\label{5.18}
\tfrac 12 \|V^n(t)\|_{\cal{H}}^2+\tfrac 12 c_m'\int_0^t
\|[v^n_t]_{\alpha}\|_{m,\alpha}^m+\tfrac 12 c_\mu'\int_0^t
\|[(v_{|\Gamma_1}^n)_t]_{\beta}\|_{\mu,\beta,\Gamma_1}^\mu \\\le
\tfrac 32 \|U_0\|_{\cal{H}}^2+\tfrac
32\int_0^t\|V^n(\tau)\|_{\cal{H}}^2\,d\tau+k_{15}(1+R^{2(p+q)})\,t\qquad\text{for
$t \in [0,1]$.}
\end{multline}
Disregarding the second and third terms in the left--hand side of
\eqref{5.18} and using Gronwall inequality (see \cite[Lemma~4.2,~p.
179]{showalter}) we get
$$
\|V^n(t)\|_{\cal{H}}^2\le
3\|U_0\|_{\cal{H}}^2e^{3t}+2e^{3}k_{15}(1+R^{2(p+q)})\,t\qquad\text{for
$t \in [0,1]$.}
$$
Consequently
\begin{equation}\label{5.19}
\|V^n(t)\|_{\cal{H}}^2\le 1+4\|U_0\|_{\cal{H}}^2\le
(1+2\|U_0\|_{\cal{H}})^2=R^2,
\end{equation}
 provided $3e^{3t}\le 4$ and $2e^3k_{15}
(1+R^{2(p+q)})t\le 1$, that is $t\in [0,T_1]$, where
$$T_1=T_1(\|U_0\|_{\cal{H}}):=\min\left\{\tfrac 13 \log \tfrac
43,[2e^3k_{15}(1+(1+2\|U_0\|_{\cal{H}})^{2(p+q)}]^{-1}\right),$$
which is trivially decreasing. By the definitions of
$\widehat{f^n_R}$, $\widehat{g^n_R}$ and $F^n_R$ then we have
$F^n_R(v^n)(t)=(\widehat{f^n}(v^n)(t),\widehat{g^n}(v^n)(t))$ for
$t\in [0,T_1]$, so $v^n$ is a weak solution of \eqref{1.1n} in $[0,T_1]$.
Since weak solutions of \eqref{1.1n} are unique we get
$v^n=u^n$ in $[0,T_1]$, so \eqref{5.11} is nothing but \eqref{5.19}.
To prove \eqref{5.12} we note that plugging \eqref{5.11} in
\eqref{5.18}, since $v^n=u^n$ in $[0,T_1]$ and $2(p+q)\ge 2$, we get
\begin{equation}\label{5.20}
c_m'\int_0^{T_1} \|[u^n_t]_{\alpha}\|_{m,\alpha}^m+
c_\mu'\int_0^{T_1}
\|[(u_{|\Gamma_1}^n)_t]_{\beta}\|_{\mu,\beta,\Gamma_1}^\mu \le
k_{16}\left[1+\|U_0\|_{\cal{H}}^{2(p+q)}\right].
\end{equation}
Recalling that  $\overline{m}=\max\{2,m\}$
 and $\overline{\mu}=\max\{2,\mu\}$ we have
 \begin{alignat*}2
&\|[u^n_t]_\alpha\|_{\overline{m},\alpha}&&\le\|\alpha\|_\infty \|u^n_t\|_2+\|[u^n_t]_\alpha\|_{m,\alpha},\\
&\|[{u^n_{|\Gamma}}_t]_\beta\|_{\overline{\mu},\beta,\Gamma_1}
&&\le\|\beta\|_{\infty,\Gamma_1}\|{u^n_{|\Gamma}}_t\|_{2,\Gamma_1}+\|[{u^n_{|\Gamma}}_t]_\beta\|_{\mu,\beta,\Gamma_1}
\end{alignat*}
so by \eqref{5.11} and \eqref{5.20} we immediately get \eqref{5.12},
completing the proof.
\end{proof}

To pass to the limit as $n\to\infty$ we shall use the following
density result, which  is proved in Appendix~\ref{appendixB} for the
reader's convenience.
\begin{lem}\label{lemma5.4}
Let $0<T<\infty$. Then $C^1_c((0,T);H^{1,\infty,\infty})$ is dense
in $C_c((0,T);H^1)\cap C^1_c((0,T);H^0)\cap Z(0,T)$ with
the norm of $C([0,T];H^1)\cap C^1([0,T];H^0)\cap Z(0,T)$.
\end{lem}
We  set
\begin{equation}\label{sigmaomegagammatilda}
\widetilde{\sigmaomega}=\begin{cases}
\somega &\text{if $1+\frac\romega 2<p=1+\frac\romega{m'}$},\\
2 & \text{otherwise},
\end{cases}
\text{and}\,\,\widetilde{\sigmagamma}=\begin{cases}
\sgamma &\text{if $1+\frac\rgamma 2<q=1+\frac\rgamma{\mu'}$},\\
2 & \text{otherwise,}
\end{cases}
\end{equation}
so, by \eqref{somegagammarelazioni1}, we have
$H^{1,\widetilde{\sigmaomega},\widetilde{\sigmagamma}}=H^{1,\sigmaomega,\sigmagamma}$
and consequently
$\cal{H}^{\widetilde{\sigmaomega},\widetilde{\sigmagamma}}=\cal{H}^{\sigmaomega,\sigmagamma}$.

The following result is a main step in the proof of
Theorem~\ref{theorem5.1}.
\begin{prop}\label{proposition5.1} Under the assumptions of
Theorem~\ref{theorem5.1} there is a decreasing function
$T_2:[0,\infty)\to (0,1]$, $T_2\le T_1$, such that for any $U_0\in
\cal{H}^{\sigmaomega,\sigmagamma}$ problem \eqref{1.1} has a weak
solution $u$ in
$[0,T_2(\|U_0\|_{\cal{H}^{\widetilde{\sigmaomega},\widetilde{\sigmagamma}}})]$.
Moreover
\begin{alignat}2\label{5.11ripetuta}
&\|U\|_{C([0,T_2(\|U_0\|_{\cal{H}})];\cal{H})}
&&\le 1+2\|U_0\|_{\cal{H}},\\
\label{5.12ripetuta}&\|\dot{u}\|_{Z(0,T_2(\|U_0\|_{\cal{H}}))}
&&\le\kappa(\|U_0\|_{\cal{H}}).
\end{alignat}
\end{prop}
\begin{proof} By Lemmas~~\ref{lemma3.1} and \ref{lemma5.3} the sequences $(u^n)_n$,
$(\dot{u}^n)_n$ and $(B(\dot{u}^n))_n$ are bounded, respectively, in
$L^\infty(0,T_1;H^1)$, $L^\infty(0,T_1;H^0)\cap Z(0,T_1)$ and $L^{\overline{m}'}(0,T_1;
[L^{\overline{m}}(\Omega,\lambda_\alpha)]')\times
L^{\overline{\mu}'}(0,T_1;
[L^{\overline{\mu}}(\Gamma_1,\lambda_\beta)]')$. Moreover by
\eqref{R} we have $p<1+\romega$ and $q<1+\rgamma$ so by
Rellich--Kondrachov theorem (see \cite[Theorem~9.16,~p.
285]{brezis2} and \cite[Theorem~2.9,~p. 39]{hebey}) the embeddings
$H^1(\Omega)\hookrightarrow L^{p-1}(\Omega)$ and
$H^1(\Gamma)\hookrightarrow L^{q-1}(\Gamma)$ are compact. Then,
using Simon's compactness results (see \cite[Corollary~5,~p.
86]{simon}) we get that, up to a subsequence,
\begin{equation}\label{convergenze}
\left\{\begin{alignedat}2 &u^n\rightharpoonup^*  u &&\text{ in
$L^\infty(0,T_1;H^1)$},\\
& \dot{u}^n\rightharpoonup^* \dot{u} &&\text{ in
$L^\infty(0,T_1;H^0)$,}\\
&u^n\to u&&\text{ in
$C([0,T_1];L^{p-1}(\Omega)\times L^{q-1}(\Gamma_1))$,}\\
&\dot{u}^n\rightharpoonup \dot{u} &&\text{ in $Z(0,T_1)$,}\\
&B(\dot{u}^n)\rightharpoonup\chi &&\text{ in $L^{\overline{m}'}(0,T_1;
[L^{\overline{m}}(\Omega,\lambda_\alpha)]')\times
L^{\overline{\mu}'}(0,T_1;
[L^{\overline{\mu}}(\Gamma_1,\lambda_\beta)]')$,}
\end{alignedat}\right.
\end{equation}
where $\to$,  $\rightharpoonup$ and $\rightharpoonup^*$ respectively
stand for strong, weak and weak $^*$ convergence. Hence, by
\eqref{5.11}--\eqref{5.12}, estimates
\eqref{5.11ripetuta}--\eqref{5.12ripetuta} will be granted for any
choice of $T_2\le T_1$. Since $u^n$ is a weak solution of
\eqref{1.1n} in $[0,T_1]$, by Definition~\ref{definition3.1} we have
\begin{multline}\label{5.23}
\int_0^{T_1}\left[-(\dot{u}^n,\dot{\varphi})_{H^0}+\langle
Au^n,\varphi\rangle_{H^1}+\langle B(\dot{u}^n),\varphi\rangle
_Y\right]\\=\int_0^{T_1}\int_\Omega
\widehat{f^n}(u^n)\varphi+\int_{\Gamma_1}\widehat{g^n}(u^n)\varphi\qquad\text{for
all $\varphi\in C_c^1((0,T_1);H^{1,\infty,\infty})$.}
\end{multline}
We now pass to the limit in  \eqref{5.3} as $n\to\infty$. By
\eqref{convergenze} we immediately get
\begin{multline}\label{5.24}
\lim_n\int_0^{T_1}\left[-(\dot{u}^n,\dot{\varphi})_{H^0}+\langle
Au^n,\varphi\rangle_{H^1}+\langle B(\dot{u}^n),\varphi\rangle
_Y\right]\\=\int_0^{T_1}\left[-(\dot{u},\dot{\varphi})_{H^0}+\langle
Au,\varphi\rangle_{H^1}+\langle \chi,\varphi\rangle _Y\right].
\end{multline}
To pass to the limit in the first term in right--hand side of
\eqref{5.23} we note that by combining \eqref{convergenze} and
\eqref{5.9} with $\rho=p-1$ we have
$\widehat{f^n}(u^n)-\widehat{f^n}(u)\to 0$ in
$C([0,T_1];L^1(\Omega))$, hence a fortiori in
$L^1(0,T_1;L^1(\Omega))$. Next, by \eqref{5.1}---\eqref{5.2} and
(FG1), we have $f^n(\cdot,u)\to f(\cdot,u)$ a.e. in
$(0,T_1)\times\Omega$ and $|f^n(\cdot,u)-f(\cdot,u)|\le
c_p|1-\eta_n(u)|(1+|u|^{p-1})\le c_p(1+|u|^{p-1})\in
L^{m_p'}((0,T_1)\times\Omega)$ by \eqref{mpmuqbound}, hence by
Fubini's and Lebesgue dominated convergence theorem we get
\begin{equation}\label{giugi}
\widehat{f^n}(u)\to\widehat{f}(u)\quad\text{in
$L^{m_p'}(0,T_1;L^{m_p'}(\Omega))$}.
\end{equation}
A fortiori $\widehat{f^n}(u)\to\widehat{f}(u)$ in
$L^1(0,T_1;L^1(\Omega))$ which, combined with previous remark,
yields $\widehat{f^n}(u^n)\to\widehat{f}(u)$ in
$L^1(0,T_1;L^1(\Omega))$. Then, as $\varphi\in
C_c^1((0,T_1);H^{1,\infty,\infty})$,
\begin{equation}\label{5.25}
\int_0^{T_1}\int_\Omega \widehat{f^n}(u^n)\varphi\to
\int_0^{T_1}\int_\Omega \widehat{f}(u)\varphi.
\end{equation}
By similar arguments we get
\begin{equation}\label{5.26}
\int_0^{T_1}\int_{\Gamma_1} \widehat{g^n}(u^n)\varphi\to
\int_0^{T_1}\int_{\Gamma_1} \widehat{g}(u)\varphi.
\end{equation}
Combining \eqref{5.23}--\eqref{5.26} we obtain
\begin{equation}\label{5.27}
\int_0^{T_1}\left[(-\dot{u},\dot{\phi})_{H^0}+\langle
Au,\phi\rangle_{H^1}+\langle \chi,\phi\rangle
_Y\right]=\int_0^{T_1}\int_\Omega
\widehat{f}(u)\phi+\int_{\Gamma_1}\widehat{g}(u)\phi
\end{equation}
for all $\varphi\in C_c^1((0,T_1);H^{1,\infty,\infty})$. Using
Lemma~\ref{lemma5.4} the distribution identity \eqref{5.27} holds
for all $\varphi\in C_c((0,T);H^1)\cap C^1_c((0,T);H^0)\cap Z(0,T)$.
Moreover, denoting $\chi=(\chi_1,\chi_2)$, by the form of the Riesz
isomorphism between $L^{\overline{m}'}(\Omega;\lambda_\alpha)$ and
$[L^{\overline{m}}(\Omega;\lambda_\alpha)]'$, we have $\chi_1=\alpha
\chi_3$ where $\chi_3\in L^{\overline{m}'}(0,T_1;
L^{\overline{m}'}(\Omega;\lambda_\alpha))$ and by the same argument
$\chi_2=\alpha \chi_4$ where $\chi_4\in L^{\overline{\mu}'}(0,T_1;
L^{\overline{\mu}'}(\Gamma_1;\lambda_\beta))$. By the remarks made
before Definition~\ref{definition3.1} then $u$ is a weak solution of
\eqref{L} with $\rho=\overline{m}$, $\theta=\overline{\mu}$,
$\xi=\widehat{f}(u)-\chi_1$, $\eta=\widehat{g}(u)-\chi_2$ and
\eqref{2.7} holds. Hence $U\in C([0,T_1];\cal{H})$.
 To complete the proof we then only have to prove that
$B(\dot{u})=\chi$ in $[0,T_2]$ for a suitable
$T_2=T_2(\|U_0\|_{\cal{H}^{\widetilde{\sigmaomega},\widetilde{\sigmagamma}}})\in
(0,T_1]$,  this one being the main
technical point in the proof.

We claim that there is  such a $T_2$ for which, up to a subsequence,
\begin{equation}\label{CLAIM}
U^n\to U\qquad\text{strongly in $C([0,T_2];\cal{H})$.}
\end{equation}
 To prove it we introduce $w^n=u^n-u$, denoting
$W^n=U^n-U$, and we note that by \eqref{5.23} and \eqref{5.27}
$w^n$ is a weak solution of \eqref{L} with
$\xi=\widehat{f^n}(u^n)-\widehat{f}(u)-\widehat{P}(u^n_t)+\chi_1$
and
$\eta=\widehat{g^n}(u^n)-\widehat{g}(u)-\widehat{Q}({u_{|\Gamma}^n}_t)+\chi_2$.
They verify \eqref{2.7} with $\rho=\overline{m}$,
$\theta=\overline{\mu}$, so by Lemma~\ref{lemma2.1} (as $W^n(0)=0$)
the energy identity
\begin{multline}\label{5.34}
\frac 12\|W^n(t)\|_{\cal{H}}^2+\int_0^t\langle
B(\dot{u}^n)-\chi,\dot{w}^n\rangle_Y=\tfrac 12\|w^n(t)\|_{2,\Gamma_1}^2\\
+\int_0^t \int_\Omega
[\widehat{f^n}(u^n)-\widehat{f}(u)]w^n_t+\int_0^t \int_{\Gamma_1}
[\widehat{g^n}(u^n)-\widehat{g}(u)]{w^n_{|\Gamma}}_t,
\end{multline}
holds for $t\in [0,T_1]$. Consequently, by Lemma~\ref{lemma3.1}--(iii) and the trivial estimate
$\|w^n(t)\|_{2,\Gamma_1}^2\le\|w^n(t)\|_{H^0}^2\le
\int_0^t\|\dot{w}^n\|_{H^0}^2$ (where $T_1\le 1$ was used) we have
\begin{multline}\label{5.28}
\frac 12\|W^n(t)\|_{\cal{H}}^2\le \frac
12\int_0^t\|\dot{w}^n\|_{H^0}^2+ \int_0^{T_1}|\langle
B(\dot{u})-\chi,\dot{w}^n\rangle_Y|\\+\int_0^t \int_\Omega
[\widehat{f^n}(u^n)-\widehat{f}(u)]w^n_t+\int_0^t \int_{\Gamma_1}
[\widehat{g^n}(u^n)-\widehat{g}(u)]{w^n_{|\Gamma}}_t.
\end{multline}
By Lemma~\ref{lemma3.1} and \eqref{convergenze}
\begin{equation}\label{5.29}
a_n(U_0):=\int_0^{T_1}|\langle B(\dot{u})-\chi,\dot{w}^n\rangle_Y|\to
0\qquad\text{as $n\to\infty$}.
\end{equation}
We are now going to estimate the last two terms in the right--hand
side of \eqref{5.28}
\begin{equation}\label{5.30}
I^f_n(t):=\int_0^t \int_\Omega
[\widehat{f^n}(u^n)-\widehat{f}(u)]w^n_t,\quad I^g_n(t):=\int_0^t
\int_{\Gamma_1}
[\widehat{g^n}(u^n)-\widehat{g}(u)]{w^n_{|\Gamma}}_t.
\end{equation}
Trivially $I^f_n(t)=I^{f,1}_n(t)+I^{f,2}_n(t)$ and
$I^g_n(t)=I^{g,1}_n(t)+I^{g,2}_n(t)$, where
\begin{equation}\label{5.31}
\begin{alignedat}2
I^{f,1}_n(t)\!=&\!\int_0^t \!\!\int_\Omega
[\widehat{f^n}(u^n)-\widehat{f^n}(u)]w^n_t,\,\,\,
&&I^{g,1}_n(t)\!=\!\int_0^t\!\! \int_{\Gamma_1}
[\widehat{g^n}(u^n)-\widehat{g^n}(u)]{w^n_{|\Gamma}}_t,\\
I^{f,2}_n(t)\!=&\!\int_0^t \!\!\int_\Omega
[\widehat{f^n}(u)-\widehat{f}(u)]w^n_t,\,\,
&&I^{g,2}_n(t)\!=\!\int_0^t\!\! \int_{\Gamma_1}
[\widehat{g^n}(u)-\widehat{g}(u)]{w^n_{|\Gamma}}_t.
\end{alignedat}
\end{equation}
To estimate $I^{f,2}_n(t)$ we note that by H\"{o}lder inequality  and Fubini's
theorem
$$ I^{f,2}_n(t)\le b_n(U_0):=
\|\widehat{f^n}(u)-\widehat{f}(u)\|_{L^{m_p'}((0,T_1)\times\Omega)}\|w^n_t\|_{L^{m_p}((0,T_1)\times\Omega)}.
$$
By \eqref{ZLMP} the sequence $(w^n_t)_n$ is bounded in
$L^{m_p}((0,T_1)\times\Omega)$. Then by \eqref{giugi}
\begin{equation}\label{5.32}
I^{f,2}_n(t)\le b_n(U_0)\to 0\qquad\text{as $n\to\infty$}.
\end{equation}
By transposing previous arguments from $\Omega$ to $\Gamma_1$ we get
that, as $n\to\infty$,
\begin{equation}\label{5.33}
I^{g,2}_n(t)\le
c_n(U_0):=\|\widehat{g^n}(u)-\widehat{g}(u)\|_{L^{\mu_q'}((0,T_1)\times\Gamma_1)}\|{w^n_{|\Gamma}}_t\|_{L^{\mu_q}((0,T_1)\times\Gamma_1)}\to
0.
\end{equation}
To estimate $I^{f,1}_n(t)$ we shall distinguish between two cases:
\renewcommand{\labelenumi}{{(\roman{enumi})}}
\begin{enumerate}
\item $1+\romega/2<p<1+\romega/m'$,
\item $2\le p\le 1+\romega/2$ or
$1+\romega/2<p=1+\romega/m'$.
\end{enumerate}
In the first one, by Lemmas~\ref{lemma5.1} and \ref{lemma5.3}, H\"{o}lder inequality and \eqref{5.11ripetuta} ,
we get
$$I^{f,1}_n(t)\le d_n(U_0):=
\int_0^{T_1}\|\widehat{f^n}(u^n)-\widehat{f^n}(u)\|_{m'}\|w^n_t\|_{m}\le
K_{15}\int_0^{T_1}\|w^n\|_{(p-1)m'}\|w^n_t\|_{m},$$
where $K_{15}=K_{15}(\|U_0\|_{\cal H})=\overline{c_p'}k_3(1+(1+2\|U_0\|_{\cal H})^{p-2})$.
Since  $(p-1)m'<\romega$ the embedding
$H^1(\Omega)\hookrightarrow L^{(p-1)m'}(\Omega)$ is compact hence,
using \eqref{convergenze} and the Simon's compactness result
recalled above, up to a subsequence we have $w^n\to 0$ strongly in
$C([0,T_1]; L^{(p-1)m'}(\Omega))$ and consequently, being $w^n_t$
bounded in $L^m(0,T_1;L^m(\Omega))$ by assumption (FGQP1),
\begin{equation}\label{5.35}
I_n^{f,1}(t)\le d_n\to 0\qquad\text{as $n\to\infty$}.
\end{equation}
In the second case we are going to apply Lemma~\ref{lemma4.1}, so
let us check its  assumptions.  By (FG12) and Lemma~\ref{lemma5.1}
the functions $f$ and $f^n$ satisfy assumptions (FG12) with constants
$\overline{\boldsymbol{c}_p}'$ independent on $n\in\N$. Hence (F2)
holds when $1+\romega/2<p=1+\romega/m'$. Moreover, setting $R_1:[0,\infty)\to[0,\infty)$ by $R_1(\tau)=\max\{1+2\tau,
\kappa(\tau)\}$, we note that $R_1$ is increasing and
consequently $1+2\|U_0\|_{\cal{H}}\le R_1(\|U_0\|_{\cal{H}})\le
R_1\left(\|U_0\|_{\cal{H}^{\widetilde{\sigmaomega},\widetilde{\sigmagamma}}}\right)$.
Then, since in this case $\widetilde{\sigmaomega}=\somega$ by
\eqref{sigmaomegagammatilda},  setting
$R=R_1\left(\|U_0\|_{\cal{H}^{\widetilde{\sigmaomega},\widetilde{\sigmagamma}}}\right)$,
by Lemma~\ref{lemma5.3} and \eqref{5.11ripetuta}--\eqref{5.12ripetuta}  we have
\begin{equation}\label{5.36}
\|U^n\|_{C([0,T_1];\cal{H})},\,\, \|U\|_{C([0,T_1];\cal{H})},\,\,
\|\dot{u}^n\|_{Z(0,T_1)},\,\,\|\dot{u}\|_{Z(0,T_1)},\,\,\|u_0\|_\somega\le
R.
\end{equation}
Then,  for any $\eps>0$, denoting
$K_{16}=K_{16}(\|U_0\|_{\cal{H}^{\widetilde{\sigmaomega},\widetilde{\sigmagamma}}})=
K_4(R_1\left(\|U_0\|_{\cal{H}^{\widetilde{\sigmaomega},\widetilde{\sigmagamma}}}\right),
\overline{\boldsymbol{c}_p})$,
$K_{17}=K_{17}(\eps,U_0)=K_5\left(\eps,R_1\left(\|U_0\|_{\cal{H}^{\widetilde{\sigmaomega},\widetilde{\sigmagamma}}}\right),
u_0,u_0,\overline{\boldsymbol{c}_p}'\right)$,
by \eqref{4.3}  we get
\begin{equation}\label{5.37}
I_n^{f,1}(t)\le
K_{16}(\eps+t)\|W^n(t)\|_{\cal{H}}^2+K_{17}\int_0^t\left(1+\|u^n_t\|_m+\|u_t\|_m\right)\|W^n(\tau)\|^2_{\cal{H}}\,d\tau.
\end{equation}
Now, by setting $d_n(U_0)=0$ in case (ii)  and
$K_{16}=K_{16}(\|U_0\|_{\cal{H}^{\widetilde{\sigmaomega},\widetilde{\sigmagamma}}})
=K_{17}=K_{17}(\eps,U_0)=1$
in case (i), we  combine \eqref{5.35} and \eqref{5.37} to get
\begin{multline}\label{5.38}
I_n^{f,1}(t)\le d_n(U_0)+
K_{16}(\eps+t)\|W^n(t)\|_{\cal{H}}^2+\\K_{17}\int_0^t\left(1+\|u^n_t\|_{m_p}+\|u_t\|_{m_p}\right)\|W^n(\tau)\|^2_{\cal{H}}\,d\tau,
\end{multline}
where $d_n(U_0)\to 0$ as $n\to\infty$. Transposing previous
arguments to $\Gamma_1$, distinguishing between the two cases
\renewcommand{\labelenumi}{{(\roman{enumi})}}
\begin{enumerate}
\item $1+\rgamma/2<q<1+\rgamma/\mu'$, and
\item $2\le q\le 1+\rgamma/2$ or
$1+\rgamma/2<q=1+\rgamma/\mu'$,
\end{enumerate}
using Lemma~\ref{lemma4.4} instead of Lemma~\ref{lemma4.1} we
estimate $I_n^{g,1}(t)$ as
\begin{multline}\label{5.39}
I_n^{g,1}(t)\le e_n(U_0)+
K_{18}(\eps+t)\|W^n(t)\|_{\cal{H}}^2\\+K_{19}\int_0^t\left(1+\|{u^n_{|\Gamma}}_t\|_{\mu_q,\Gamma_1}+\|{u_{|\Gamma}}_t\|_{\mu_q,\Gamma_1}\right)\|W^n\|^2_{\cal{H}},
\end{multline}
where $e_n(U_0)\to 0$ as $n\to\infty$ and we denote
$K_{18}=K_{18}(\|U_0\|_{\cal{H}^{\widetilde{\sigmaomega},\widetilde{\sigmagamma}}})=
 K_{19}=K_{19}(\eps,U_0)=1$
 and $e_n(U_0)=
\int_0^{T_1}\|\widehat{g^n}(u^n)-\widehat{g^n}(u)\|_{\mu',\Gamma_1}\|(w_
{|\Gamma_1}^n)_t\|_{\mu,\Gamma_1}$ in case (i), while $e_n(U_0)=0$,
$
K_{18}=K_{18}(\|U_0\|_{\cal{H}^{\widetilde{\sigmaomega},\widetilde{\sigmagamma}}})=
K_{12}(R_1\left(\|U_0\|_{\cal{H}^{\widetilde{\sigmaomega},\widetilde{\sigmagamma}}}\right),\boldsymbol{\overline{c_q}})
$
and
$K_{19}=K_{19}(\eps,U_0)=K_{13}(\eps,R,u_0,u_0,\overline{\boldsymbol{c}_q}')$
in case (ii).

Hence, denoting $h_n(U_0)=b_n(U_0)+c_n(U_0)+d_n(U_0)+e_n(U_0)$,
\begin{equation*}\label{5.36six}
K_{20}=K_{20}(\|U_0\|_{\cal{H}^{\widetilde{\sigmaomega},\widetilde{\sigmagamma}}})=
K_{16}(\|U_0\|_{\cal{H}^{\widetilde{\sigmaomega},\widetilde{\sigmagamma}}})+
K_{18}(\|U_0\|_{\cal{H}^{\widetilde{\sigmaomega},\widetilde{\sigmagamma}}})
\end{equation*}
and
$K_{21}=K_{21}(\eps,U_0)=K_{17}(\eps,U_0)+K_{19}(\eps,U_0)$,
by \eqref{5.30}--\eqref{5.33}, \eqref{5.37} and \eqref{5.39} we get
that $h_n(U_0)\to 0$ as $n\to\infty$ and
\begin{multline}\label{4.39bis}
I_n^f(t)+I_n^g(t)\le h_n(U_0)+
K_{20}(\eps+t)\|W^n(t)\|_{\cal{H}}^2\\+
K_{21}\int_0^t\left(1+\|u^n_t\|_{m_p}+\|u_t\|_{m_p}+
\|{u^n_{|\Gamma}}_t\|_{\mu_q,\Gamma_1}+\|{u_{|\Gamma}}_t\|_{\mu_q,\Gamma_1}\right)\|W^n(\tau)\|^2_{\cal{H}}\,d\tau
\end{multline}
for all $t\in [0,T_1]$. Plugging it with \eqref{5.29} in
\eqref{5.28} we finally get
\begin{multline}\label{5.40}
\frac 12\|(W^n(t)\|_{\cal{H}}^2\le \l_n(U_0)+
K_{20}(\eps+t)\|W^n(t)\|_{\cal{H}}^2+\\
K_{22}\int_0^t\left(1+\|u^n_t\|_{m_p}+\|u_t\|_{m_p}+
\|{u^n_{|\Gamma}}_t\|_{\mu_q,\Gamma_1}+\|{u_{|\Gamma}}_t\|_{\mu_q,\Gamma_1}\right)\|W^n(\tau)\|^2_{\cal{H}}\,d\tau
\end{multline}
where
$K_{22}=K_{22}(\eps,U_0)=1+K_{21}(\eps,U_0)$
and
\begin{equation}\label{5.41}
l_n(U_0)=a_n(U_0)+h_n(U_0)\to 0\quad\text{as $n\to\infty$.}
\end{equation}
We now set the function $T_2:[0,\infty)\to(0,1]$ by $T_2(\tau)=\min\{T_1(\tau),1/8K_{20}(\tau)\}$. Hence $T_2\le T_1$, $K_{20}T_2\le
1/8$ and $T_2$ is decreasing. We also choose
$\eps=\eps_2:=1/8K_{20}(\|U_0\|_{\cal{H}^{\widetilde{\sigmaomega},\widetilde{\sigmagamma}}})$ so that, denoting
$K_{23}=K_{23}(U_0)=K_{22}(\eps_2,U_0)$,
by \eqref{5.40}
\begin{multline}\label{5.42}
\frac 14\|(W^n(t)\|_{\cal{H}}^2\le \l_n(U_0)+
K_{23}\int_0^t\Big(1+\|u^n_t\|_{m_p}+\|u_t\|_{m_p}\\+
\|{u^n_{|\Gamma}}_t\|_{\mu_q,\Gamma_1}+\|{u_{|\Gamma}}_t\|_{\mu_q,\Gamma_1}\Big)\|W^n(\tau)\|^2_{\cal{H}}\,d\tau
\end{multline}
for all $t\in
[0,T_2(\|U_0\|_{\cal{H}^{\widetilde{\sigmaomega},\widetilde{\sigmagamma}}})]$,
so by the already recalled Gronwall inequality
\begin{multline}\label{5.43}
\|W^n(t)\|_{\cal{H}}^2\le 4\l_n(U_0)\exp{\Big[
4K_{23}}\int_0^{T_2}\Big(1+\|u^n_t\|_{m_p}+\|u_t\|_{m_p}\\+
\|{u^n_{|\Gamma}}_t\|_{\mu_q,\Gamma_1}+\|{u_{|\Gamma}}_t\|_{\mu_q,\Gamma_1}\Big)\,d\tau
\Big]
\end{multline}
for all $t\in
[0,T_2(\|U_0\|_{\cal{H}^{\widetilde{\sigmaomega},\widetilde{\sigmagamma}}})]$.
Since, by \eqref{ZLMP}, the sequences $(u^n_t)_n$ and
${u_{|\Gamma}^n}_t$ are respectively bounded in
$L^{m_p}((0,T_1)\times\Omega)$ and in
$L^{\mu_q}((0,T_1)\times\Gamma_1)$, by \eqref{5.41} we get that
\eqref{CLAIM} holds, proving our claim.

Using \eqref{CLAIM}, \eqref{4.39bis}, the just used boundedness of
$(u^n_t)_n$, ${u_{|\Gamma}^n}_t$ and  $\lim_n h_n(U_0)=0$  we get
$I_n^f(T_2)+I_n^g(T_2)\to 0$ as $n\to\infty$. Consequently, using
\eqref{CLAIM} again in the energy identity \eqref{5.34} we get that
$\lim_n\int_0^{T_2}\langle B(\dot{u}^n)-\chi,\dot{w}^n\rangle_Y=0$. Since
by Lemma~\ref{lemma3.1} and \cite[Theorem~1.3,~p.40]{barbu} the
operator $B$ is maximal monotone in $Z(0,T_2)$ this fact yields, by
the classical monotonicity argument (see for example
\cite[Lemma~1.3,~p.49]{barbu1993}), that $B(\dot{u})=\chi$ in $(0,T_2)$,
concluding the proof.
\end{proof}
\begin{proof}[Proof of Theorem~\ref{theorem5.1}.] By
Proposition~\ref{proposition5.1} we get the existence of a weak
solution in $[0,T_2]$ when $U_0\in
\cal{H}^{\sigmaomega,\sigmagamma}$.  The existence of a maximal weak
solution $u$ in $[0,T_{\text{max}})$ follows by a standard application of
Zorn's lemma. By \eqref{lsomegagamma} and
Lemma~\ref{lemma3.3}--(i)--(iii) we get \eqref{1.15} and the
energy identity \eqref{energyidentity}, completing the proof of
(i--ii). To prove (iii) we suppose by contradiction that
$T_{\text{max}}<\infty$ and $\varlimsup_{t\to
T^-_{\text{max}}}\|U(t)\|_{H^1\times H^0}<\infty,$ which by
\eqref{1.15} implies $U\in L^\infty(0,T_{\text{max}};\cal{H})$. By
Lemma~\ref{lemma3.3}-(iii--iv) then $u$ is a weak solution in
$[0,T_{\text{max}}]$ and $U\in
C([0,T_{\text{max}}];\cal{H}^{\sigmaomega,\sigmagamma})$. Then,
applying Proposition~\ref{proposition5.1}, problem  \eqref{1.1} with
initial data $U(T_{\text{max}})$ has a weak solution $v$ in
$[0,T_2]$. By Lemma~\ref{lemma3.3}--(ii), $\bar{u}$ defined by
$\bar{u}(t)=u(t)$ for $t\in [0,T_{\text{max}}]$,
$\bar{u}(t)=v(t-T_{\text{max}})$ for $t\in
[T_{\text{max}},T_{\text{max}}+T_2]$ is a weak solution in
$[0,T_{\text{max}}+T_2]$, contradicting the maximality of $u$.
\end{proof}

We now state and prove, for the sake of clearness, some corollaries
of Theorem~\ref{theorem5.1} which generalize
Corollaries~\ref{corollary1.1}--\ref{corollary1.3} in the
introduction. The discussion made there applies here as
well.
\begin{cor} \label{corollary5.1} Under assumptions (PQ1--3), (FG1),
(FGQP1) and \eqref{R'} the conclusions of Theorem~\ref{theorem5.1}
hold and $H^{1,\sigmaomega,\sigmagamma}=H^1$.
\end{cor}
\begin{proof} Since \eqref{R'} can be written also as
$p\not=1+\romega/m'$ when $p>1+\romega/2$ and $q\not=1+\rgamma/\mu'$
when $q>1+\rgamma/2$, clearly assumption (FG2) can be skipped. Moreover, by  \eqref{sigmaomegagamma},  when \eqref{R'} holds
we have $H^{1,\sigmaomega,\sigmagamma}=H^1$. \end{proof}
\begin{rem} Clearly when \eqref{R'} holds the proof of  Proposition~\ref{proposition5.1} can be simplified since Lemmas~\ref{lemma4.1} and \ref{lemma4.4} are not needed.
\end{rem}
\begin{cor}\label{corollary5.2} If assumptions (PQ1--3), (FG1--2),
(FGQP1)  are satisfied and \eqref{R''bis} holds, hence in particular
when $2\le p\le \romega$ and $2\le q\le \rgamma$,
 the conclusions of Theorem~\ref{theorem5.1} hold and
$H^{1,\sigmaomega,\sigmagamma}=H^1$.
\end{cor}
\begin{proof} When \eqref{R''bis} holds by \eqref{sigmaomegagamma} we have $\sigmaomega=\sigmagamma=2$,
hence $H^{1,\sigmaomega,\sigmagamma}=H^1$. \end{proof}

\renewcommand{\labelenumi}{{(\roman{enumi})}}
\begin{cor} \label{corollary5.3} Under assumptions (PQ1--3), (FG1--2),
(FGQP1)  the conclusions of Theorem~\ref{theorem5.1} hold when the
space $H^{1,\sigmaomega,\sigmagamma}$ is replaced by
\begin{enumerate}
\item the space $H^{1,\rho,\theta}_{\alpha,\beta}$, provided \eqref{ranges1}
holds;
\item the space $H^{1,\rho,\theta}$,
provided \eqref{spaziuguali} and \eqref{ranges1} holds;
\item the space $H^{1,\somega,\sgamma}$, provided  also assumption (FG2)$'$ holds.
\end{enumerate}
\end{cor}
\begin{proof} The statement (i) follows by combining
Theorem~\ref{theorem5.1} with Lemma~\ref{lemma3.3}--(iii), (ii)
follows by combining (i)  with Remark~\ref{remark1.1}, while (iii),
by \eqref{somegagammarelazioni2}--\eqref{lsomegagamma}, is a
particular case of (ii)  when (FG2)$'$ holds.\end{proof}

\begin{proof}[Proof of Theorem~\ref{theorem1.1} and
Corollaries~\ref{corollary1.1}--\ref{corollary1.3} in
Section~\ref{intro}.] When
$$P(x,v)=\alpha(x)P_0(v),\quad Q(x,v)=\beta(x)Q_0(v),\quad f(x,u)=f_0(u),\quad
g(x,u)=g_0(u),$$ by Remarks~\ref{remark3.1}--\ref{remark3.3}
assumptions (PQ1--3), (FG1), (FGQP1) reduce to (I--III), while by
Remark~\ref{remark3.6} (FG2) and (FG2)$'$ reduce to (IV)
and (IV)$'$, hence  Theorem~\ref{theorem1.1} and
Corollaries~\ref{corollary1.1}--\ref{corollary1.3} are particular
cases  of Theorem~\ref{theorem5.1} and
Corollaries~\ref{corollary5.1}--\ref{corollary5.3}.
\end{proof}
\section{Uniqueness and local well--posedness}\label{section6}
This section is devoted to our uniqueness and well--posedness
results for problem \eqref{1.1}. To get uniqueness of solutions we
need to restrict to sources satisfying (FG2)$'$ and  $u_0\in
H^{1,\sigmaomega,\sigmagamma}$, as in the last statement of
Corollary~\ref{corollary5.3}.
\begin{thm}[\bf Uniqueness]\label{theorem6.1} Suppose that (PQ1--3), (FG1), (FGQP1) and (FG2)$'$
 hold, let $U_0\in\cal{H}^{\sigmaomega,\sigmagamma}$ and $u$,
 $v$ are maximal solutions of \eqref{1.1}. Then $u=v$.
\end{thm}
\begin{proof} We fix $u,v$, so constants $K_i$ introduced in this proof will depend also on them.
 We denote $\text{dom }u=[0,T^u_{\text{max}})$ and  $\text{dom
}v=[0,T^v_{\text{max}})$.  By Lemma~\ref{lemma3.3} and
\eqref{somegagammarelazioni2}--\eqref{lsomegagamma} we have $U\in
C([0,T^u_{\text{max}});\cal{H}^{\sigmaomega,\sigmagamma})$ and $V\in
C([0,T^v_{\text{max}});\cal{H}^{\sigmaomega,\sigmagamma})$. We set
$\widetilde{T}_{\max}=\min\{T^u_{\text{max}},T^v_{\text{max}}\}$.

We claim that there is a (possibly small)
$\widetilde{T}<\widetilde{T}_{\max}$ such that $u=v$ in
$[0,\widetilde{T}]$. To prove our claim we set
$T'=\min\{1,\widetilde{T}_{\max}/2\}$, $R_2=R_2(u,v)$ by
\begin{equation}\label{6.1}
R_2=\max\left\{\|U\|_{C([0,T'];\cal{H}^{\sigmaomega,\sigmagamma})},\,
\|V\|_{C([0,T'];\cal{H}^{\sigmaomega,\sigmagamma}},\,
\|\dot{u}\|_{Z(0,T')},\,\|\dot{v}\|_{Z(0,T')}\right\},
\end{equation}
and we denote $w=u-v$, $W=U-V$. Clearly $w$ is a weak solution
of \eqref{L} with
$\xi=\widehat{f}(u)-\widehat{f}(v)-\widehat{P}(u_t)+\widehat{P}(v_t)$,
$\eta=\widehat{g}(u)-\widehat{g}(v)-\widehat{Q}(u_t)+\widehat{Q}(v_t)$,
$\xi,\eta$ satisfying \eqref{2.7} with $\rho=\overline{m}$,
$\theta=\overline{\mu}$ and $W(0)=0$. Hence, by
Lemma~\ref{lemma2.1}, the energy identity
\begin{multline*}
\tfrac 12\|W(t)\|_{\cal{H}}^2+\int_0^t\langle
B(\dot{u})-B(\dot{v}),\dot{w}\rangle_Y\\= \tfrac 12\|w(t)\|_{2,\Gamma_1}^2+\int_0^t
\int_\Omega [\widehat{f}(u)-\widehat{f}(v)]w_t+\int_0^t
\int_{\Gamma_1} [\widehat{g}(u)-\widehat{g}(v)]{w_{|\Gamma}}_t,
\end{multline*}
holds for $t\in [0,T']$. Consequently, by Lemma~\ref{lemma3.1}--(iii)
and the trivial estimate
$\|w(t)\|_{2,\Gamma_1}^2\le\|w(t)\|_{H^0}^2\le
\int_0^t\|\dot{w}\|_{H^0}^2$ (where $T'\le 1$ was used) we have, for
$t\in [0,T']$,
$$
\frac 12\|W(t)\|_{\cal{H}}^2\le \frac
12\int_0^t\|\dot{w}\|_{H^0}^2+\int_0^t \int_\Omega
[\widehat{f}(u)-\widehat{f}(v)]w_t+\int_0^t \int_{\Gamma_1}
[\widehat{g}(u)-\widehat{g}(v)]{w_{|\Gamma}}_t. $$ By assumption
(FG2)$'$ we can apply Lemmas~\ref{lemma4.1} and \ref{lemma4.4} with
$R$ given by \eqref{6.1}, hence for any $\eps>0$, denoting
$K_{24}=K_4(R_2,\boldsymbol{c}_p)+K_{12}(R_2),\boldsymbol{c}_q)$,
$$K_{25}=K_{25}(\eps)=K_5(\eps,R_2,u_0,u_0,\boldsymbol{c}_p')+K_{13}(\eps,R_2,u_0,u_0,\boldsymbol{c}_q')+1,$$
\begin{equation}\label{costantivettoriali}
\boldsymbol{c_p}=(c_p,c_p'),\quad \boldsymbol{c_q}=(c_q,c_q'),\quad \boldsymbol{c_p'}=(c_p,c_p',c_p''), \quad \boldsymbol{c_q'}=(c_q,c_q',c_q''),
\end{equation}
plugging \eqref{4.1} and \eqref{4.1bis} in previous estimate we get
\begin{multline*}
\tfrac 12\|W(t)\|_{\cal{H}}^2\le K_{24}(\eps+t)\|W(t)\|_{\cal{H}}^2+
K_{25}\int_0^t\Big(1+\|u_t\|_{m_p}+\|v_t\|_{m_p}\\+
\|{u_{|\Gamma}}_t\|_{\mu_q,\Gamma_1}+\|{v_{|\Gamma}}_t\|_{\mu_q,\Gamma_1}\Big)\|W(\tau)\|^2_{\cal{H}}\,d\tau.
\end{multline*}
Choosing $\eps=\eps_3:=1/8K_{24}$ and
 denoting
$K_{26}=K_{25}(\eps_3)$ we have
\begin{multline*}
\tfrac 14\|W(t)\|_{\cal{H}}^2\le
K_{26}\int_0^t\Big(1+\|u_t\|_{m_p}+\|v_t\|_{m_p}+
\|{u_{|\Gamma}}_t\|_{\mu_q,\Gamma_1}\\+\|{v_{|\Gamma}}_t\|_{\mu_q,\Gamma_1}\Big)\|W(\tau)\|^2_{\cal{H}}\,d\tau,
\end{multline*}
for $t\in [0,\widetilde{T}]$ where $\widetilde{T}:=\min\{T',
1/8K_{24}\}$.  Since by \eqref{ZLMP} we have
$1+\|u_t\|_{m_p}+\|v_t\|_{m_p}+
\|{u_{|\Gamma}}_t\|_{\mu_q,\Gamma_1}+\|{v_{|\Gamma}}_t\|_{\mu_q,\Gamma_1}\in
L^1(0,\widetilde{T})$, by the already recalled Gronwall inequality
we get $\|W(t)\|^2_{\cal{H}}\le 0$ for $t\in [0,\widetilde{T}]$,
proving our claim.

The statement now follows in a standard way, which is described in
the sequel for the reader's convenience. We set $T^*=\sup\{t\in
[0,\widetilde{T}_{\max}): u=v\,\,\text{in}\,\, [0,t]\}$. Clearly
$T^*\le \widetilde{T}_{\max}$ and $U=V$ in $[0,T^*)$. Supposing by
contradiction that $T^*<\widetilde{T}_{\max}$ we then have $U,V\in
C([0,T^*];\cal{H}^{\sigmaomega,\sigmagamma})$, so $U(T^*)=V(T^*)\in
\cal{H}^{\sigmaomega,\sigmagamma}$. Then, since \eqref{1.1}  is
autonomous, $\tilde{u}(t):=u(t+T^*)$ and $\tilde{v}(t):=v(t+T^*)$
are weak solutions of \eqref{1.1}, with initial data
$u(T^*),\dot{u}(T^*)$, in $[0, \widetilde{T}_{\max}-T^*)$, so by our
claim $\widetilde{u}=\widetilde{v}$ in $[0,\tau)$ for some $\tau>0$,
i.e. $u=v$ in $[0,T^*+\tau)$ contradicting the definition of $T^*$.
Hence $T^*=\widetilde{T}_{\max}$ and $U=V$ in
$[0,\widetilde{T}_{\max})$. Finally
$T^u_{\text{max}}=T^v_{\text{max}}$ since  if
$T^u_{\text{max}}<T^v_{\text{max}}$ then $v$ is a proper extension
of $u$, a contradiction.
\end{proof}
Essentially by combining Corollary~\ref{corollary5.3} and
Theorem~\ref{theorem6.1} we get
\begin{thm}[\bf Local existence--uniqueness]\label{theorem6.2} Suppose that (PQ1--3), (FG1), (FGQP1) and (FG2)$'$
 hold. Then all conclusions of Theorem~\ref{theorem1.2} hold
true when problem \eqref{1.2} is generalized to problem \eqref{1.1},
provided the energy identity \eqref{energyidentityreduced} is
generalized to \eqref{energyidentity}.
\end{thm}
\begin{proof} By combining Corollary~\ref{corollary5.3} and
Theorem~\ref{theorem6.1} we immediately get statements (i--ii) and
$\varlimsup_{t\to T^-_{\text{max}}}\|U(t)\|_{H^1\times H^0}=\infty$
when $T_{\text{max}}<\infty$, so we have only to prove that
$\lim_{t\to T^-_{\text{max}}}\|U(t)\|_{H^{1,\somega,\sgamma}\times
H^0}=\infty$ in this case. This fact follows from
Proposition~\ref{proposition5.1} and a standard procedure, described
in the sequel. Since $H^{1,\somega,\sgamma}\hookrightarrow
H^{1,\sigmaomega,\sigmagamma}=H^{1,\widetilde{\sigmaomega},\widetilde{\sigmagamma}}$
we shall prove that $\lim_{t\to
T^-_{\text{max}}}\|U(t)\|_{H^{1,\widetilde{\sigmaomega},\widetilde{\sigmagamma}}\times
H^0}=\infty$. Suppose by contradiction that
$M:=\sup_n\{\|U(t_n)\|_{\cal{H}^{\widetilde{\sigmaomega},\widetilde{\sigmagamma}}}\}<\infty$
for some $t_n\to T^-_{\max}$. Then by
Proposition~\ref{proposition5.1} for each $n\in\N$ problem
\eqref{1.1} with initial data $U(t_n)$ has a weak solution  $v_n$ in
$[0,T_2(M)]$. Hence, by Lemma~\ref{lemma3.3}--(ii), $w_n$ defined by
$w_n(t)=u(t)$ for $t\in [0,t_n]$ and $w_n(t)=v_n(t-t_n)$ for $t\in
[t_n,t_n+T_2(M)]$ is a weak solution of \eqref{1.1} in
$[0,t_n+T_2(M)]$ and, by  Theorem~\ref{theorem6.1}, $w_n=u$ so,
being $u$ maximal, $t_n+T_2(M)<T^-_{\max}$  which,
when $n\to\infty$, gives a contradiction.
\end{proof}
\begin{rem} From the proof of the case (ii) in  Proposition~\ref{proposition5.1} it is clear
that one can give an alternative proof of Theorem~\ref{theorem6.2} without using compactness
arguments.
\end{rem}
We now give a consequence of
Theorem~\ref{theorem6.2} which generalizes Corollary
~\ref{corollary1.4} in the introduction, the discussion made there
applying as well.
\begin{cor} \label{corollary6.1} Under assumptions (PQ1--3), (FG1), (FGQP1) and (FG2)$'$ the
conclusions of Theorem~\ref{theorem6.2} hold when the space
$H^{1,\somega,\sgamma}$ is replaced by
\begin{enumerate}
\item the space $H^{1,\rho,\theta}_{\alpha,\beta}$, provided
\eqref{ranges2} holds, and
\item
the space $H^{1,\rho,\theta}$, provided \eqref{spaziuguali} and \eqref{ranges2} hold.
\end{enumerate}
\end{cor}

We now give our main local Hadamard well--posedness result for
problem \eqref{1.1}, restricting to damping terms satisfying also
assumption (PQ4), to non--bicritical
nonlinearities and to $u_0\in H^{1,s_2,s_2}$ with $s_1,s_2$
satisfying \eqref{ranges2}.

\begin{thm}[\bf Local Hadamard well--posedness I]\label{theorem6.3} Suppose that (PQ1--4), (FG1), (FGQP1),
(FG2)$'$, \eqref{R''} hold and  let $s_1,s_2$ satisfy
\eqref{ranges2ter}. Then all conclusions of Theorem~\ref{theorem1.3}
hold true when problem \eqref{1.2} is generalized to
\eqref{1.1}.

 In particular \eqref{1.1}  is locally well--posed in
$\cal{H}$, under assumptions (PQ1--3), (FG1), (FGQP1) and
(FG2)$'$, when $2\le p<\romega$ and $2\le q< \rgamma$ .
\end{thm}
\begin{proof}
Let $U_{0n}:=(u_{0n},u_{1n})\to U_0:=(u_0,u_1)$ in
$\cal{H}^{s_1,s_2}$ and $U^n\in
C([0,T^n_{\max});\cal{H}^{s_1,s_2})$, $U\in
C([0,T_{\max});\cal{H}^{s_1,s_2})$, $s_1$ and $s_2$ be fixed as in the statement.
\begin{footnote}{functions and constants $K_i$ introduced in this proof will depend also on them.}\end{footnote} Since
$H^{1,s_1,s_2}\hookrightarrow H^{1,\somega,\sgamma}\hookrightarrow
H^{1,\sigmaomega,\sigmagamma}=H^{1,\widetilde{\sigmaomega},\widetilde{\sigmagamma}}$,
\begin{equation}\label{6.2}
M(t)=\max\{\|U\|_{C([0,t];\cal{H}^{s_1,s_2})},\|U\|_{C([0,t];\cal{H}^{\somega,\sgamma})},\|U\|_{C([0,t];\cal{H}^{\widetilde{\sigmaomega},\widetilde{\sigmagamma}})}
\}
\end{equation}
defines an increasing function $M:[0, T_{\max})\to
[0,\infty)$. Hence
\begin{equation}\label{6.3} T_3(t)=T_2(1+M(t)),
\end{equation}
where $T_2$ is the function defined in
Proposition~\ref{proposition5.1}, defines a decreasing function
$T_3:[0,T_{\max})\to [0,1)$,
 with $T_3(t)\le T_1(1+M(t))$. Now let
$T^*\in (0,T_{\max})$. By \eqref{6.2} we have
$\|U_0\|_{\cal{H}^{s_1,s_2}},\|U_0\|_{\cal{H}^{\somega,\sgamma}},\|U_0\|_{\cal{H}^{\widetilde{\sigmaomega},\widetilde{\sigmagamma}}}\le
M(T^*)$ and consequently, since $U_{0n}\to U_0$ in
$\cal{H}^{s_1,s_2}$, there is $n_1=n_1(T^*)\in\N$,  such that for all $n\ge
n_1(T^*)$,
\begin{equation}\label{6.4}
\begin{alignedat}4
&\|U_{0n}\|_{\cal{H}^{s_1,s_2}},
&&\|U_{0n}\|_{\cal{H}^{\somega,\sgamma}},
&&\|U_{0n}\|_{\cal{H}^{\widetilde{\sigmaomega},\widetilde{\sigmagamma}}}\le &&1+M(T^*),\\
&\|U_0\|_{\cal{H}^{s_1,s_2}}, &&\|U_0\|_{\cal{H}^{\somega,\sgamma}},
&&\|U_0\|_{\cal{H}^{\widetilde{\sigmaomega},\widetilde{\sigmagamma}}}
\le &&1+M(T^*).
\end{alignedat}
\end{equation}
Since $u_n$ and $u$ are unique maximal solutions by
\eqref{6.3}--\eqref{6.4} and Proposition~\ref{proposition5.1}
\begin{gather}\label{6.5}
T_3(T^*)<T_{\max},\quad T_3(T^*)<T_{\max}^n,\\
\label{6.6}\begin{alignedat}3
&\|U^n\|_{C([0,T_3(T^*)];\cal{H})},\quad
&&\|U\|_{C([0,T_3(T^*)];\cal{H})}&&\le 1+2\|U_0\|_{\cal{H}},\\
&\|\dot{u}^n\|_{Z(0,T_3(T^*))},\quad &&\|\dot{u}\|_{Z(0,T_3(T^))}&&\le
\kappa(\|U_0\|_{\cal{H}})
\end{alignedat}
\end{gather}
for all $n\ge n_1(T^*)$. Hence, as
$\|U_0\|_{\cal{H}}\le\|U_0\|_{\cal{H}^{\widetilde{\sigmaomega},\widetilde{\sigmagamma}}}$,
setting the increasing function $R_3:[0,T_{\max})\to [0,\infty)$ by
$R_3(\tau)=\max\{3+2M(\tau), \kappa(1+M(\tau))\}$, by \eqref{6.4}
and \eqref{6.6} we have
\begin{equation}\label{6.7}
\begin{alignedat}3
&\|U^n\|_{C([0,T_3(T^*)];\cal{H})},\quad
&&\|U\|_{C([0,T_3(T^*)];\cal{H})}&&\le R_3(T^*),\\
&\|\dot{u}^n\|_{Z(0,T_3(T^*))},\quad &&\|\dot{u}\|_{Z(0,T_3(T^*))}&&\le
R_3(T^*)
\end{alignedat}
\end{equation}
for all $n\ge n_1(T^*)$. Since by  (FG2)$'$ the property (F2) holds when
$p>1+\romega/2$, by \eqref{ranges2ter} we have $\somega<s_1$,
$\sgamma<s_2$ and $1+M(T^*)\le R_3(T^*)$, we can apply the final
parts of  Lemmas~\ref{lemma4.1} and \ref{lemma4.4}. Consequently, keeping the notation \eqref{costantivettoriali} and
denoting $w^n=u^n-u$, $W^n=U^n-U$, $W_{0n}=U_{0n}-U_0$,
$K_{27}=K_{27}(T^*)=K_4(R_3(T^*),\boldsymbol{c}_p)+K_{12}(R_3(T^*),\boldsymbol{c}_q)$ and
$$K_{28}=K_{28}(\eps,T^*)=K_6(\eps,R_3(T^*),\boldsymbol{c}_p')
+K_{14}(\eps,R_3(T^*),\boldsymbol{c}_q'),$$ we have the estimate
\begin{multline}\label{6.9}
\int_0^t \int_\Omega
[\widehat{f}(u^n)-\widehat{f}(u)]w^n_t+\int_0^t \int_{\Gamma_1}
[\widehat{g}(u^n)-\widehat{g}(u)]{w^n_{|\Gamma}}_t\\\le
K_{27}(\eps+t)\|W^n(t)\|_{\cal{H}}^2+
K_{28}\Bigg[\|W_{0n}\|_{\cal{H}}^2+\int_0^t\Big(1+\|u^n_t\|_{m_p}+\|u_t\|_{m_p}\\+
\|{u^n_{|\Gamma}}_t\|_{\mu_q,\Gamma_1}+\|{u_{|\Gamma}}_t\|_{\mu_q,\Gamma_1}\Big)\|W^n\|^2_{\cal{H}}\Bigg]
\qquad\text{for all $t\in [0,T_3(T^*)]$.}
\end{multline}
Since $w^n$ is a weak solution of \eqref{L} with
$\xi=\widehat{f}(u^n)-\widehat{f}(u)-\widehat{P}(u^n_t)+\widehat{P}(u_t)$
and
$\eta=\widehat{g}(u^n)-\widehat{g}(u)-\widehat{Q}({u_{|\Gamma}^n}_t)+\widehat{Q}({u_{|\Gamma}}_t)$
verifying \eqref{2.7} with $\rho=\overline{m}$,
$\theta=\overline{\mu}$,  by Lemma~\ref{lemma2.1} the energy
identity
\begin{multline}\label{6.10}
\frac 12\|W^n(t)\|_{\cal{H}}^2+\int_0^t\langle
\widehat{B}(\dot{u}^n)-\widehat{B}(\dot{u}),\dot{w}^n\rangle_Y=\tfrac
12\|W_{0n}\|_{\cal{H}}^2-\tfrac 12\|w_{0n}\|_{2,\Gamma_1}^2\\+\tfrac
12\|w^n(t)\|_{2,\Gamma_1}^2 +\int_0^t \int_\Omega
[\widehat{f}(u^n)-\widehat{f}(u)]w^n_t+\int_0^t \int_{\Gamma_1}
[\widehat{g}(u^n)-\widehat{g}(u)]{w^n_{|\Gamma}}_t,
\end{multline}
holds for all $t\in [0,T_3(T^*)]$. By  (PQ4) and
\eqref{3.20}--\eqref{3.21} we have
\begin{multline}\label{6.11}
\langle \widehat{B}(\dot{u}^n)-\widehat{B}(\dot{u}),\dot{w}^n\rangle_Y\ge
\widetilde{c_m}''\|[w_t^n]_\alpha\|_{m,\alpha}^m
-c_m'''\|\alpha\|_\infty\|w^n_t\|_2^2\\
+\widetilde{c_\mu}''\|[{w_{|\Gamma}}_t^n]_\beta\|_{\mu,\beta,\Gamma_1}^\mu
-c_\mu'''\|\beta\|_{\infty,\Gamma_1}\|(w^n_{|\Gamma1})_t\|_2^2\\
\ge
\widetilde{c_m}''\|[w_t^n]_\alpha\|_{m,\alpha}^m+\widetilde{c_\mu}''\|[{w_{|\Gamma}}_t^n]_\beta\|_{\mu,\beta,\Gamma_1}^\mu
-k_{17}\|W^n\|_{\cal{H}}^2,
\end{multline}
with $\widetilde{c_m}''>0$ when $p\ge \romega$ and
$\widetilde{c_\mu}''>0$ when $q\ge \rgamma$. Hence, setting
$K_{29}=K_{29}(\eps,T^*)=
2+K_{28}(\eps,T^*)$ and plugging
\eqref{6.9}, \eqref{6.11} and the trivial estimate $\tfrac12
\|w^n(t)\|_{2,\Gamma_1}^2\le\tfrac12 \|w^n(t)\|_{H^0}^2\le
\|W_{0n}\|_{\cal{H}}^2+\int_0^t\|W^n\|_{H^0}^2$ (where $T_3(T^*)\le
1$ was used) in \eqref{6.10} we get
\begin{multline}\label{6.12}
\frac
12\|W^n(t)\|_{\cal{H}}^2+\int_0^t\widetilde{c_m}''\|[w_t^n]_\alpha\|_{m,\alpha}^m+\widetilde{c_\mu}''\|[{w_{|\Gamma}}_t^n]_\beta\|_{\mu,\beta,\Gamma_1}^\mu\\\le
K_{27}(\eps+t)\|W^n(t)\|_{\cal{H}}^2+
K_{29}\Bigg[\|W_{0n}\|_{\cal{H}}^2+\int_0^t\Big(1+\|u^n_t\|_{m_p}+\|u_t\|_{m_p}\\+
\|{u^n_{|\Gamma}}_t\|_{\mu_q,\Gamma_1}+\|{u_{|\Gamma}}_t\|_{\mu_q,\Gamma_1}\Big)\|W^n\|^2_{\cal{H}}\Bigg]
\qquad\text{for all $t\in [0,T_3(T^*)]$.}
\end{multline}
We now set
$T_4(T^*)=\min\{T_3(T^*),1/8K_{27}(T^*)\}$,
so that $T_4\le T_3$ and $K_{27}T_4\le 1/8$. We
also choose $\eps=\eps_4:=1/8K_{27}(T^*)$
so that, setting
$K_{30}=K_{30}(T^*)=4K_{29}(\eps_4,T^*)$,
by \eqref{6.12} we have
\begin{multline}\label{6.13}
\|W^n(t)\|_{\cal{H}}^2+\int_0^t\widetilde{c_m}''\|[w_t^n]_\alpha\|_{m,\alpha}^m+\widetilde{c_\mu}''\|[{w_{|\Gamma}}_t^n]_\beta\|_{\mu,\beta,\Gamma_1}^\mu\\\le
K_{30}\Bigg[\|W_{0n}\|_{\cal{H}}^2+\int_0^t\Big(1+\|u^n_t\|_{m_p}+\|u_t\|_{m_p}+
\|{u^n_{|\Gamma}}_t\|_{\mu_q,\Gamma_1}+\|{u_{|\Gamma}}_t\|_{\mu_q,\Gamma_1}\Big)\|W^n\|^2_{\cal{H}}\Bigg]
\end{multline}
for all $t\in [0,T_4(T^*)]$. By disregarding the second term in the
left--hand side of \eqref{6.13} and applying  Gronwall
inequality  we get
\begin{multline*}
\|(W^n(t)\|_{\cal{H}}^2\le K_{30}\|W_{0n}\|_{\cal {H}}^2\exp{\Bigg[
K_{30}}\int_0^{T_3(T^*)}\Big(1+\|u^n_t\|_{m_p}+\|u_t\|_{m_p}\\+
\|{u^n_{|\Gamma}}_t\|_{\mu_q,\Gamma_1}+\|{u_{|\Gamma}}_t\|_{\mu_q,\Gamma_1}\Big)\,d\tau
\Bigg].
\end{multline*}
Consequently, by H\"{o}lder inequality in time, \eqref{ZLMP} and
\eqref{6.7},
\begin{equation}\label{6.14}
\|W^n(t)\|_{\cal{H}}^2\le K_{30}\|W_{0n}\|_{\cal {H}}^2e^{
K_{30}[1+2R_3(T^*)]}\qquad\text{for all $t\in [0,T_4(T^*)]$,}
\end{equation}
from which we immediately get
\begin{equation}\label{6.15}
W^n\to 0\quad\text{in $C([0,T_4(T^*)];\cal{H})$ \quad as
$n\to\infty$.}
\end{equation}
To get the stronger (when $p\ge \romega$ or $q\ge \rgamma$)
convergence in  $C([0,T_4(T^*)];\cal{H}^{s_1,s_2})$ we now plug
\eqref{6.14} into \eqref{6.13} and use \eqref{ZLMP}, \eqref{6.7} and
H\"{o}lder inequality to get
\begin{multline*}
\int_0^{T_4(T^*)}\widetilde{c_m}''\|[w_t^n]_\alpha\|_{m,\alpha}^m+\widetilde{c_\mu}''\|[{w_{|\Gamma}}_t^n]_\beta\|_{\mu,\beta,\Gamma_1}^\mu\le
K_{30}\|W_{0n}\|_{\cal {H}}^2\Bigg\{1+K_{30}e^{
K_{30}[1+2R_3(T^*)]}\\\times\int_0^{T_3(T^*)}\Big(1+\|u^n_t\|_{m_p}+\|u_t\|_{m_p}+
\|{u^n_{|\Gamma}}_t\|_{\mu_q,\Gamma_1}+\|{u_{|\Gamma}}_t\|_{\mu_q,\Gamma_1}\Big)\,d\tau\Big\}\\\le
K_{30}\|W_{0n}\|_{\cal {H}}^2\Big\{1+K_{30}e^{
K_{30}[1+2R_3(T^*)]}[1+2R_3(T^*)]\Bigg\},
\end{multline*}
from which it immediately follows that
\begin{equation}\label{6.16}
\widetilde{c_m}''\|[w_t^n]_\alpha\|_{L^m(0,T_4(T^*);L^m_\alpha(\Omega)}^m,\widetilde{c_\mu}''\|[{w_{|\Gamma}}_t^n]_{L^\mu(0,T_4(T^*);L^\mu_\beta(\Gamma_1)}^\mu\to
0
\end{equation}
as $n\to\infty$. When $p\ge\romega$ we have $\widetilde{c_m}''>0$ so
by \eqref{6.16} and (FGQP1) it follows $w^n_t\to 0$ in
$L^m(0,T_4(T^*);L^m(\Omega))$. Since by
\eqref{somegagammarelazioni2} and \eqref{ranges2ter} in this case we
have $s_1\le m$, we derive  $w^n_t\to 0$ in
$L^{s_1}(0,T_4(T^*);L^{s_1}(\Omega))$. As $w_{0n}\to 0$ in
$L^{s_1}(\Omega)$ we get by a trivial integration in time that
$w^n\to 0$ in $C([0,T_4(T^*)];L^{s_1}(\Omega))$. Using similar
arguments we get $w^n_{|\Gamma}\to 0$ in
$C([0,T_4(T^*)];L^{s_2}(\Gamma_1))$ when $q\ge\rgamma$.
Consequently, using \eqref{6.15},  Sobolev embeddings and
\eqref{ranges2ter} when $p<\romega$ and $q<\rgamma$ we derive
\begin{equation}\label{6.17}
W^n\to 0\quad\text{in $C([0,T_4(T^*)];\cal{H}^{s_1,s_2})$ \quad as
$n\to\infty$.}
\end{equation}
We then complete the proof by repeating previous arguments a finite
 number of times. More explicitly we set the
function $\kappa_1:(0,T_{\max})\to\N_0$,  by
$\kappa_1(T^*)=\min\{k\in\N_0 : T^*/T_4(T^*)\le k+1\}$, so that
\begin{equation}\label{6.18}
\kappa_1(T^*)T_4(T^*)<T^*\le [\kappa_1(T^*)+1]T_4(T^*).
\end{equation}
If $\kappa_1(T^*)=0$, that is if $T^*\le T_4(T^*)$, by \eqref{6.17}
we have $W^n\to 0$ in $C([0,T^*];\cal{H}^{s_1,s_2})$, that is the
conclusion (ii) in the statement of Theorem~\ref{theorem1.3}.
Moreover in this case by \eqref{6.5} we have
\begin{equation*}
T^*\le T_4(T^*)\le T_3(T^*)<T^n_{\max}\qquad\text{for all $n\ge
n_1(T^*)$.}
\end{equation*}
Now let $\kappa_1(T^*)\ge 1$, that is $T_4(T^*)<T^*$. By \eqref{6.2}
we have $\|U(T_4(T^*))\|_{\cal{H}^{s_1,s_2}}$,
$\|U(T_4(T^*))\|_{\cal{H}^{\somega,\sgamma}}$,
$\|U(T_4(T^*))\|_{\cal{H}^{\widetilde{\sigmaomega},\widetilde{\sigmagamma}}}\le
M(T^*)$, and consequently, by \eqref{6.17}, there is
$n_2=n_2(T^*)\in\N$ such that for all $n\ge
n_2(T^*)$
\begin{equation}\label{6.19}
\begin{alignedat}4
&\|U^n_1\|_{\cal{H}^{s_1,s_2}},\quad
&&\|U^n_1\|_{\cal{H}^{\somega,\sgamma}},\quad
&&\|U^n_1\|_{\cal{H}^{\widetilde{\sigmaomega},\widetilde{\sigmagamma}}}\le &&1+M(T^*),\\
&\|U_1\|_{\cal{H}^{s_1,s_2}}, &&\|U_1\|_{\cal{H}^{\somega,\sgamma}},
&&\|U_1\|_{\cal{H}^{\widetilde{\sigmaomega},\widetilde{\sigmagamma}}}
\le &&1+M(T^*),
\end{alignedat}
\end{equation}
where we denote $U^n_1=U^n(T_4(T^*))$ and $U_1=U(T_4(T^*))$. Since
problem \eqref{1.1} is autonomous, by applying
Lemma~\ref{lemma3.3}--(ii) and Theorem~\ref{theorem6.1}, starting
from \eqref{6.19} we can repeat all arguments from \eqref{6.4} to
\eqref{6.17}, getting in this way that
\begin{equation}\label{6.21}
2T_4(T^*)\le T_3(T^*)+T_4(T^*)<T_{\max},\quad 2T_4(T^*)\le
T_3(T^*)+T_4(T^*)<T_{\max}^n,
\end{equation}
for all $n\ge n_2(T^*)$, and
$$
\begin{aligned}
&\widetilde{c_m}''\|[w_t^n]_\alpha\|_{L^m(T_4(T^*),2T_4(T^*));L^m_\alpha(\Omega)}^m,\quad
\widetilde{c_\mu}''\|[{w_{|\Gamma}}_t^n]_{L^\mu(T_4(T^*),2T_4(T^*));L^\mu_\beta(\Gamma_1)}^\mu\to
0, \\
&W^n\to 0\quad\text{in $C([T_4(T^*),2T_4(T^*)];\cal{H}^{s_1,s_2})$ as
$n\to\infty$,}
\end{aligned}
$$
 which by \eqref{6.17} implies \label{6.22bis}
\begin{equation}
\label{6.22}
\begin{aligned}
&\widetilde{c_m}''\|[w_t^n]_\alpha\|_{L^m(0,2T_4(T^*));L^m_\alpha(\Omega)}^m,\quad
\widetilde{c_\mu}''\|[{w_{|\Gamma}}_t^n]_{L^\mu(0,2T_4(T^*));L^\mu_\beta(\Gamma_1)}^\mu\to
0,\\
& W^n\to 0\quad\text{in $C([0,2T_4(T^*)];\cal{H}^{s_1,s_2})$ \quad
as $n\to\infty$.}
\end{aligned}
\end{equation}
If $\kappa_1(T^*)=1$, that is if $T^*\le 2T_4(T^*)$, by \eqref{6.22}
we have
$$\begin{gathered}
\widetilde{c_m}''\|[w_t^n]_\alpha\|_{L^m(0,T^*);L^m_\alpha(\Omega)}^m+
\widetilde{c_\mu}''\|[{w_{|\Gamma}}_t^n]_{L^\mu(0,T^*);L^\mu_\beta(\Gamma_1)}^\mu\to
0,\\
U^n\to U\quad\text{in $C([0,T^*];\cal{H}^{s_1,s_2})$,}
\end{gathered}
$$
 that is the
conclusion (ii) in the statement of Theorem~\ref{theorem1.3}.
Moreover, in this case by \eqref{6.21} we have
\begin{equation}\label{6.23}
T^*<T^n_{\max}\qquad\text{for all $n\ge n_2(T^*)$.}
\end{equation}
If $\kappa_1(T^*)>1$ we repeat the procedure above $\kappa_1(T^*)-1$
times to get
\begin{equation}\label{6.24}
\begin{aligned}
&T^*\le (\kappa_1(T^*)+1)T_4(T^*)<T^n_{\max}\qquad\text{for all $n\ge
n_{\kappa_1(T^*)+1}(T^*)$},\\
&\widetilde{c_m}''\|[w_t^n]_\alpha\|_{L^m(0,(\kappa_1(T^*)+1)T_4(T^*));L^m_\alpha(\Omega)}^m\to
0,\\
&\widetilde{c_\mu}''\|[{w_{|\Gamma}}_t^n]_{L^\mu(0,(\kappa_1(T^*)+1)T_4(T^*));L^\mu_\beta(\Gamma_1)}^\mu\to
0,
\\
&W^n\to 0\quad\text{in
$C([0,(\kappa_1(T^*)+1)T_4(T^*)];\cal{H}^{s_1,s_2})$ \quad as
$n\to\infty$.}
\end{aligned}
\end{equation}
By \eqref{6.24} and \eqref{6.18} we then get
\begin{equation}\label{6.26}
\begin{aligned}
&\widetilde{c_m}''\|[w_t^n]_\alpha\|_{L^m(0,T^*);L^m_\alpha(\Omega)}^m,\quad
\widetilde{c_\mu}''\|[{w_{|\Gamma}}_t^n]_{L^\mu(0,T^*);L^\mu_\beta(\Gamma_1)}^\mu\to
0,
\\
&W^n\to 0\quad\text{in $C([0,T^*];\cal{H}^{s_1,s_2})$ \quad as
$n\to\infty$,}
\end{aligned}
\end{equation}
and, being $T^*\in (0,T_{\max})$ arbitrary, $T_{\text{max}}\le \varliminf_n T^n_{\text{max}}$.
\end{proof}
The aim of our final main result is  to get well--posedness for  spaces  in Corollary~\ref{corollary5.3}.
\begin{thm}[\bf Local Hadamard well--posedness II]\label{theorem6.4} Suppose that  (PQ1--3), (PQ4)$'$,(FG1), (FGQP1),
(FG2)$'$, \eqref{R''} hold and $\rho,\theta$ satisfy \eqref{ranges3}.
Then problem \eqref{1.1} is locally well--posed in
$H^{1,\rho,\theta}_{\alpha,\beta}\times H^0$, i.e. the conclusions
of Theorem~\ref{theorem1.3} hold true with $H^{1,s_1,s_2}$ replaced
by $H^{1,\rho,\theta}_{\alpha,\beta}$ and problem \eqref{1.2} is
generalized to problem \eqref{1.1}. In particular it is
locally--well posed  in $\cal{H}^{\rho,\theta}$ when
$\rho,\theta$ satisfy \eqref{ranges3}, $\rho\le \romega$  if  $p\le
1+\romega/2$ and $\theta\le \rgamma$ if $q\le 1+\rgamma/2$.
\end{thm}
\begin{proof} At first we remark that, for any $\rho,\theta$
satisfying \eqref{ranges3}, setting
\begin{equation}\label{6.27}
\overline{s_1}(\rho)=
\begin{cases}
2,   \quad&\text{if $p<\romega$},\\
\rho \quad&\text{if $p\ge\romega$},
\end{cases}\qquad
\overline{s_2}(\theta)=
\begin{cases}
2,   \quad&\text{if $q<\rgamma$},\\
\theta \quad&\text{if $q\ge\rgamma$},
\end{cases}
\end{equation}
when $p\ge \romega$ we have $\rho\ge \romega$
so, by (FGQP1),
$L^{2,\rho}_\alpha(\Omega)=L^\rho(\Omega)$, while by the same
arguments when $q\ge \rgamma$ we have
$L^{2,\rho}_\beta(\Gamma_1)=L^\rho(\Gamma_1)$. Hence
$H^{1,\rho,\theta}_{\alpha,\beta}\hookrightarrow
H^{1,\overline{s_1}(\rho),\overline{s_2}(\theta)}$.

Consequently, given a sequence $U_{0n}\to U_0$ in
$H^{1,\rho,\theta}_{\alpha,\beta}\times H^0$, we have $U_{0n}\to
U_0$ in $\cal{H}^{\overline{s_1}(\rho),\overline{s_2}(\theta)}$, hence,
since  $(\overline{s_1}(\rho),\overline{s_2}(\theta))$ satisfies
\eqref{ranges2ter}, by \eqref{6.26} we get
\begin{equation}
\begin{aligned}\label{6.28}
&\widetilde{c_m}''\|[w_t^n]_\alpha\|_{L^m(0,T^*);L^m_\alpha(\Omega)}^m,\quad
\widetilde{c_\mu}''\|[{w_{|\Gamma}}_t^n]_\beta\|_{L^\mu(0,T^*);L^\mu_\beta(\Gamma_1)}^\mu\to
0,
\\&W^n\to 0\quad\text{in $C([0,T^*];\cal{H})$ \quad as $n\to\infty$,}
\end{aligned}
\end{equation}
so to prove that $W^n\to 0$ in
$C([0,T^*];H^{1,\rho,\theta}_{\alpha,\beta}\times H^0)$, by Sobolev
embeddings,  reduces to prove the following two facts: if
$\rho>\romega$ then $w_n\to 0$ in
$C[0,T^*];L^{2,\rho}_\alpha(\Omega))$, and, if $\theta>\rgamma$, then
$w_n\to 0$ in $C[0,T^*];L^{2,\rho}_\beta(\Gamma_1))$. We  prove the
first one. When $\rho>\romega$, by \eqref{ranges3} we also have
$m>\romega$. Then, by  (PQ4)$'$, see Remark~\ref{remark3.8},
we have $\widetilde{c_m}''>0$, so by \eqref{6.28} we get $w^n_t\to
0$ in $L^m(0,T^*;L^{2,m}_\alpha(\Omega))$. Consequently, since
$\rho\le m$ by \eqref{ranges3}, $w^n_t\to 0$ in
$L^\rho(0,T^*;L^{2,\rho}_\alpha(\Omega))$. Since $w_{0n}\to 0$ in
$L^{2,\rho}_\alpha(\Omega)$ then $w_n\to 0$ in
$C[0,T^*];L^{2,\rho}_\alpha(\Omega))$. The proof of the second fact
uses similar arguments and it is omitted. Finally, when
$(\rho,\theta)$ satisfy \eqref{ranges3},  $\rho\le \romega$  if  $p\le
1+\romega/2$ and $\theta\le \rgamma$ if $q\le 1+\rgamma/2$, by Remark~\ref{remark1.1}
we have  $H^{1,\rho,\theta}=H^{1,\rho,\theta}_{\alpha,\beta}$.
\end{proof}
\begin{proof}[Proof of Theorems~\ref{theorem1.2}--\ref{theorem1.4} and
Corollary~\ref{corollary1.4} in Section~\ref{intro}.] When
$$P(x,v)=\alpha(x)P_0(v),\quad Q(x,v)=\beta(x)Q_0(v),\quad f(x,u)=f_0(u),\quad
g(x,u)=g_0(u),$$ by Remarks~\ref{remark3.1}--\ref{remark3.3}
assumptions (PQ1--3), (FG1), (FGQP1) reduce to (I--III). Moreover
by Remark~\ref{remark3.6} (FG2) and (FG2)$'$  reduce to
(IV) and (IV)$'$, while by Remark~\ref{remark3.9} (PQ4) and (PQ4)$'$
reduce to (V) and (V)$'$.  Hence
Theorems~\ref{theorem1.2}--\ref{theorem1.4} and
Corollary~\ref{corollary1.4} are particular cases  of
Theorem~\ref{theorem6.2}--\ref{theorem6.4} and
Corollary~\ref{corollary6.1}.
\end{proof}
\section{Global existence}\label{section7}
In this section we shall prove that when the source parts of the perturbation terms $f$ and $g$   has at most linear growth at infinity, uniformly in the space
variable, or, roughly, it is dominated by the corresponding damping term, then weak solutions of \eqref{1.1} found in Theorem~\ref{theorem5.1} are global in time provided $u_0\in H^{1,p,q}$.

To precise our statement we introduce, the assumption (FG1) being in force, the primitives of the functions $f$ and $g$ by
\begin{equation}\label{fgprimitives}
\mathfrak{F}(x,u)=\int_0^u f(x,s)\,ds,\qquad\text{and}\quad \mathfrak{G}(y,u)=\int_0^u g(y,s)\,ds,
\end{equation}
for a.a. $x\in\Omega$, $y\in\Gamma_1$ and all $u\in\R$. Moreover we shall make the following specific assumption:
\renewcommand{\labelenumi}{{(FGQP2)}}
\begin{enumerate}
\item there are $p_1$ and $q_1$ verifying \eqref{V.2} and constants $C_{p_1},C_{q_1}\ge 0$ such that
$$
\mathfrak{F}(x,u)\le C_{p_1}\left[1+u^2+\alpha(x)|u|^{p_1}\right], \,\,\mathfrak{G}(y,u)\le  C_{q_1}\left[1+u^2+\beta(y)|u|^{q_1}\right]
$$
for a.a. $x\in\Omega$, $y\in\Gamma_1$ and all $u\in\R$.
\end{enumerate}
Since $\mathfrak{F}(\cdot,u)=\int_0^1 f(\cdot,su)u\,ds$ (and similarly $\frak{G}$), assumption
(FGQP2) is a weak version of
 of the following one:
\renewcommand{\labelenumi}{{(FGQP2)$'$}}
\begin{enumerate}
\item there are $p_1$ and $q_1$ verifying \eqref{V.2} and constants $C'_{p_1},C'_{q_1}\ge 0$ such that
$$f(x,u)u\le C'_{p_1}\left[|u|+u^2+\alpha(x)|u|^{p_1}\right], \quad g(y,u)u\le  C'_{q_1}\left[|u|+u^2+\beta(y)|u|^{q_1}\right]
$$
for a.a. $x\in\Omega$, $y\in\Gamma_1$ and all $u\in\R$.
\end{enumerate}
\begin{rem}\label{remark6.3ff}
Assumptions (FG1) and (FGQP2)$'$ hold
 provided
 \begin{equation}\label{centerdotbis}
 f=f^0+f^1+f^2,\qquad g=g^0+g^1+g^2,
 \end{equation}
 where $f^i$, $g^i$ satisfy the following assumptions:
\renewcommand{\labelenumi}{{(\roman{enumi})}}
\begin{enumerate}
\item $f^0$ and $g^0$ are a.e. bounded and independent on $u$;
\item $f^1$ and $g^1$ satisfy (FG1) with exponents $p_1$ and $q_1$ satisfying \eqref{V.2}, and
\begin{enumerate}
\item when $p_1>2$ and $\essinf_\Omega\alpha=0$ there is a constant $\overline{c_{p_1}}\ge 0$ such
\begin{footnote}{that is $\varlimsup_{|u|\to\infty}|f_1(\cdot,u)|/|u|^{p_1-1}\le\overline{c_{p_1}}\alpha$ a.e. uniformly in $\Omega$.}\end{footnote} that
$$|f^1(x,u)|\le \overline{c_{p_1}}\left[1+|u|+\alpha(x)|u|^{p_1-1}\right]$$
for a.a. $x\in\Omega$ and all $u\in\R$;
\item when $q_1>2$ and $\essinf_{\Gamma_1}\beta =0$ there is a constant $\overline{c_{q_1}}\ge 0$ such that
$$|g^1(y,u)|\le \overline{c_{q_1}}\left[1+|u|+\beta(y)|u|^{q_1-1}\right]$$
for a.a. $y\in\Gamma_1$ and all $u\in\R$;
\end{enumerate}
\item $f^2$ and  $g^2$ satisfy (FG1),
$f^2(x,u)u\le 0$ and $g^2(y,u)u\le 0$ for a.a. $x\in\Omega$, $y\in\Gamma_1$ and all $u\in\R$.
\end{enumerate}
Conversely any couple of functions $f$ and $g$ satisfying (FG1) and (FGQP2)$'$ admits a decomposition of the form \eqref{centerdotbis}--(i--iii) with
$f^1$ and $g^1$ being source terms. Indeed one can set
$f^0=f(\cdot,0)$,
$$f^1(\cdot,u)=
\begin{cases}[f(\cdot,u)-f^0]^+ \,\,&\text{if $u>0$,}\\
\,\,\qquad 0\,\,&\text{if $u=0$,}\\
-[f(\cdot,u)-f^0]^- \,\,&\text{if $u<0$,}
\end{cases}\,\, \text{and}\,\,
f^2(\cdot,u)\begin{cases}-[f(\cdot,u)-f^0]^- \,\,&\text{if $u>0$,}\\
\,\,\qquad 0\,\,&\text{if $u=0$,}\\
[f(\cdot,u)-f^0]^+ \,\,&\text{if $u<0$,}
\end{cases}
$$
and define $g^0$, $g^1$, $g^2$ in the analogous way.
\end{rem}
\begin{rem}\label{remark6.4}When dealing with  problem \eqref{1.2} assumption
(FGQP2) reduces to (VI). The function $f\equiv f_2$ defined in \eqref{modellisupplementari} satisfies  (FGQP2) provided
one among the following cases occurs:
\renewcommand{\labelenumi}{{(\roman{enumi})}}
\begin{enumerate}
\item $\gamma_1^+=\gamma_2^+\equiv 0$,
\item $\gamma_2^+\equiv0$, $\gamma_1^+\not\equiv0$,  $\widetilde{p}\le\max\{2,m\}$ and
$\gamma_1\le c'_1\alpha$ a.e. in $\Omega$ when $\widetilde{p}>2$
\item $\gamma_1^+\not=0$, $\gamma_2^+\not\equiv 0$, $p\le\max\{2,m\}$,
$\gamma_1\le c'_1\alpha$ when $\widetilde{p}>2$ and $\gamma_2\le c'_2\alpha$ when $p>2$, a.e. in $\Omega$,
\end{enumerate}
where $c_1', c_2'\ge 0$ denote suitable constants. The analogous cases (j--jjj) occurs when $g\equiv g_2$,
so that $(f_2,g_2)$ satisfies (FGQP2) provided any combination between the cases (i--iii) and (j--jjj) occurs.
In particular then a damping term can be localized provided the corresponding source is equally localized.

Finally when $f\equiv f_3$ and $g\equiv g_3$ as in  \eqref{f3g3}, assumption (FGQP2) holds provided
$f_0$ and $g_0$  satisfy assumption (VI) (where we conventionally take $f_0\equiv 0$ when $\gamma\equiv 0$ and
$g_0\equiv 0$ when $\delta\equiv 0$), $\gamma\le \alpha$ when $p_1>2$ and
$\delta\le \beta$ when $q_1>2$.
\end{rem}

We can now state the main result of this section.
\begin{thm}[\bf Global analysis]  \label{theorem7.1}
The following conclusions hold true.
\renewcommand{\labelenumi}{{(\roman{enumi})}}
\begin{enumerate}
\item {\bf (Global existence)} Suppose that (PQ1--3), (FG1--2) and (FGQP1--2) hold. Then for any $(u_0,u_1)\in \cal{H}^{\lomega,\lgamma}$
the weak maximal solution $u$ of problem \eqref{1.1} found in Theorem~\ref{theorem5.1} is global in time, that is
$T_{\text{max}}=\infty$, and $u\in C([0,T_{\text{max}});H^{1,\lomega,\lgamma})$.

In particular, when \eqref{subSobolevcritical} holds, for any $(u_0,u_1)\in \cal{H}$ problem \eqref{1.1} has a global weak solution.

\item {\bf (Global existence--uniqueness)} Suppose that (PQ1--3), (FG1), (FG2)$'$ and (FGQP1--2) hold. Then for any $(u_0,u_1)\in H^{1,\somega,\sgamma}\times H^0$ the
unique maximal solution of problem \eqref{1.1} found in  Theorem~\ref{theorem6.2} is global in time, that is
$T_{\text{max}}=\infty$, and $u\in C([0,\infty);H^{1,\somega,\sgamma})$.

In particular, when \eqref{subSobolevcritical} holds, for any $(u_0,u_1)\in H^1\times H^0$ problem \eqref{1.1} has a unique global weak solution.

\item {\bf (Global Hadamard well--posedness)}  Suppose that (PQ1--4), (FG1), (FG2)$'$,  (FGQP1--2)
and \eqref{R''} hold. Then problem \eqref{1.1} is  globally well--posed in $H^{1,s_1,s_2}\times H^0$ for
$s_1$ and $s_2$ satisfying \eqref{ranges2ter}, that is $T_{\text{max}}=\infty$ in Theorem~\ref{theorem6.3}.

Consequently the semi--flow generated by problem \eqref{1.2} is a dynamical system in $H^{1,s_1,s_2}\times H^0$.

In particular, when $2\le p<\romega$ and $2\le q<\rgamma$ and under assumptions (PQ1--3), (FG1), (FG2)$'$,  (FGQP1--2),
problem \eqref{1.1} is globally well--posed in $H^1\times H^0$, so the semi--flow generated by \eqref{1.1} is a dynamical system in $H^1\times H^0$.
\end{enumerate}
\end{thm}
To prove Theorem~\ref{theorem7.1} we shall use  following abstract version of the classical chain rule, which proof is given for the reader's convenience.
\begin{lem}\label{chainrule} Let $X_1$ and $Y_1$ be real Banach spaces such that $X_1\hookrightarrow Y_1$ with dense embedding, so that
$Y_1'\hookrightarrow X_1'$, and let $I$ be a bounded real interval.

Then for any $J_1\in C^1(X_1)$ having Fr\`{e}chet derivative $J_1'\in C(X_1;Y_1')$ and
any $w\in W^{1,1}(I;Y_1)\cap C(\overline{I};X_1)$ we have
 $J_1\cdot w\in W^{1,1}(I)$ and $(J_1\cdot w)'=\langle J_1'\cdot w,w'\rangle_{Y_1}$ almost everywhere in $I$, where $\cdot$ denotes the composition product.
\end{lem}
\begin{proof} We first note that, when $w\in C^1(\R;X_1)$, by the chain rule for the
Fr\`{e}chet derivative (see \cite[Proposition~1.4, p. 12]{ambrosettiprodi}), we have $J_1\cdot w\in C^1(\R)$ and
$$(J_1\cdot w)'=\langle J_1'\cdot w,w'\rangle_{X_1}=\langle J_1'\cdot w,w'\rangle_{X_1}\qquad\text{in $\R$.}$$

When $w\in W^{1,1}(I;Y_1)\cap C(\overline{I};X_1)$ we first extend it by reflexion to $w\in W^{1,1}(\R;Y_1)\cap C(\R;X_1)$
as in \cite[Theorem~8.6, p.~209]{brezis2}). Then, denoting by $(\rho_n)_n$ a standard sequence of mollifiers and  by $*$ the standard convolution product in $\R$,
we set $w_n=\rho_n *w$, so $w_n\in C^1(\R;X_1)$ and, as in \cite[Proposition~4.21, p.~108 and proof of Theorem~8.7, p.~211]{brezis2}),
we have ${w_n}_{|I}\to w$ in $W^{1,1}(I;Y_1)\cap C(\overline{I};X_1)$. By previous remark
\begin{equation}\label{chainruleq1}
\int_I J_1\cdot w_n\varphi'=-\int_I\langle J_1'\cdot w_n,w_n'\rangle_{Y_1}\phi\qquad\text{for all $\phi\in C_c^1(I)$ and $n\in\N$.}
\end{equation}
We now claim that $C_0:=\bigcup_{i=0}^\infty w_{k_i}(\overline{I})$ (where we denoted $w_0=w$) is compact in $X_1$.
Indeed, given any sequence $(x_n)_n$ in $C_0$, either there is $N_0\in\N$ such that $x_n\in \bigcup_{i=0}^{N_0} w_{k_i}(\overline{I})$ for all $n\in\N$, and hence
$(x_n)_n$ has a convergent subsequence since this set is compact, or
there are sequences $(t_n)_n$ in $\overline{I}$ and $(k_n)_n$ in $\N_0$ such that
$x_n=w_{k_n}(t_n)$ for all $n\in\N$ and  $k_n\to \infty$.
Then $t_n\to \overline{t}\in\overline{I}$,  up to a subsequence, and consequently
$$\|x_n-w(\overline{t})\|_{X_1}\le\|w_{k_n}(t_n)-w(t_n)\|_{X_1}+\|w(t_n)-w(\overline{t})\|_{X_1}\to 0,$$
proving our claim.

We now pass to the limit in \eqref{chainruleq1}. By the continuity of $J_1$ and our claim  we get that $J_1\cdot w_n\to J_1\cdot w$ in $I$ and that
$J_1(C_0)$ is compact, and hence bounded, in $X_1$. Consequently $(J_1\cdot w_n)_n$ is uniformly bounded in $I$ and we get
$\lim_n \int_I J_1\cdot w_n\varphi'=\int_I J_1\cdot w\,\varphi'$. Moreover, up to a subsequence, there is $\psi\in L^1(I)$ such
that $w_n'\to w'$  and  $\|w_n'\|_{Y_1}\le \psi$ a.e. in $I$. By the continuity of $J_1'$ and our claim  we get that $J_1'\cdot w_n\to J_1'\cdot w$
in $I$ and that $J_1'(C_0)$  is compact in $Y_1'$. Consequently $$M:=\sup\{\|J_1'\cdot w_n(t)\|_{Y_1'}, n\in\N, t\in I\}<\infty.$$
It follows that $\langle J_1'\cdot w_n,w_n'\rangle_{Y_1}\to \langle J_1'\cdot w,w'\rangle_{Y_1}$ a.e in $I$ and that
$\langle J_1'\cdot w_n,w_n'\rangle_{Y_1}\le M\psi$, so  $\lim_n \int_I\langle J_1'\cdot w_n,w_n'\rangle_{Y_1}=\int_I\langle J_1'\cdot w,w'\rangle_{Y_1}$.
\end{proof}

\begin{proof}[Proof of Theorem~\ref{theorem7.1}]
We first remark that, since $\cal{H}^{\somega,\sgamma}\subset\cal{H}^{\lomega,\lgamma}$,  parts (ii) and (iii) simply follow by combining Theorems~\ref{theorem6.2}--\ref{theorem6.3} with part (i),
hence in the sequel we are just going to prove it.
Since $\cal{H}^{\lomega,\lgamma}\subset\cal{H}^{\sigmaomega,\sigmagamma}$ by Theorem~\ref{theorem5.1}
problem \eqref{1.1} has a maximal weak solution $u$ in $[0,T_{\max})$ and
\begin{equation}\label{CCCP1234}
\lim\limits_{t\to T^-_{\text{max}}}
\|u(t)\|_{H^1}+\|u'(t)\|_{H^0}=\infty.
\end{equation}
provided $T_{\max}<\infty$. Moreover, by \eqref{lomegagamma} and \eqref{lsomegagamma}, the couple $(\lomega,\lgamma)$ satisfies \eqref{spaziuguali} and \eqref{ranges1}, so
 $u\in C([0,T_{\text{max}});H^{1,\lomega,\lgamma})$ by Corollary~\ref{corollary5.3}.
We no suppose by contradiction that $T_{\max}<\infty$, so \eqref{CCCP1234} holds.

By  (FG1)  and  Sobolev embedding theorem we can set the potential operator
$J:H^{1,p,q}\to\R$ by
\begin{equation}\label{potentialJ}
J(v)=\int_\Omega \mathfrak{F}(\cdot,v)+\int_{\Gamma_1} \mathfrak{G}(\cdot,v_{|\Gamma})\qquad\text{for all $v\in H^{1,p,q}$,}
\end{equation}
and, using standard results on Nemitskii operators (see \cite[pp. 16--22]{ambrosettiprodi})
one easily gets that $J\in C^1(H^{1,p,q})$, with  Fr\`echet derivative $J'=(\widehat{f},\widehat{g})$. Moreover, by Lemma~\ref{Nemitskii},
$\widehat{f}\in C(L^p(\Omega);L^{p'}(\Omega))$, $\widehat{f}\in C(H^1(\Omega);L^{m_p'}(\Omega))$,
$\widehat{g}\in C(L^q(\Gamma_1);L^{q'}(\Gamma_1))$ and $\widehat{g}\in C(H^1(\Gamma)\cap L^2(\Gamma_1);L^{\mu_q'}(\Gamma_1))$.
Hence, setting the auxiliary exponents
$$\widetilde{m_p}=
\begin{cases}
m_p=2 &\text{if $p\le 1+\romega/2$},\\
m_p=m &\text{if $p>\max\{m,1+\romega/2\}$},\\
p &\text{if $1+\romega/2<p\le m$},
\end{cases}
\widetilde{\mu_q}=\begin{cases}
\mu_q=2 &\text{if $q\le 1+\rgamma/2$},\\
\mu_q=\mu &\text{if $q>\max\{\mu,1+\rgamma/2\}$},\\
q &\text{if $1+\rgamma/2<q\le \mu$},
\end{cases}
$$
we have $J'=(\widehat{f},\widehat{g})\in C(H^{1,p,q};L^{\widetilde{m_p}'}(\Omega)\times L^{\widetilde{\mu_q}'}(\Gamma_1))$.
We also introduce the functional $\cal{I}:H^{1,p,q}\to\R^+_0$ given by
\begin{equation}\label{6.106bis}
\cal{I}(v)=C_{p_1}\int_\Omega \alpha|v|^{p_1}+C_{q_1}\int_{\Gamma_1}\beta|v|^{q_1}.
\end{equation}
Since by \eqref{V.2} we have $p_1\le p$ and $q_1\le q$, the functions
\begin{equation}\label{frakfg}
\mathfrak{f}(x,u)=p_1C_{p_1}\alpha(x)|u|^{p_1-2}u\qquad\text{and}\quad \mathfrak{g}(x,u)=q_1C_{q_1}\beta(x)|u|^{q_1-2}u
\end{equation}
satisfy assumption (FG1) with exponents $p$ and $q$, hence by repeating previous arguments
$\cal{I}\in C^1(H^{1,p,q})$, with  Fr\`echet derivative
$\cal{I}'=(\mathfrak{f},\mathfrak{g})\in C(H^{1,p,q};L^{\widetilde{m_p}'}(\Omega)\times L^{\widetilde{\mu_q}'}(\Gamma_1))$.

We are now going to apply Lemma~\ref{chainrule}, with $X_1=H^{1,p,q}$ and $Y_1=L^{\widetilde{m_p}}(\Omega)\times L^{\widetilde{\mu_q}}(\Gamma_1)$,
to the potential operators $J_1=J$ and $J_1=\cal{I}$, to $w=u$ and to $I=[s,t]\subset [0,T_{\max})$.
By the definition of  $\widetilde{m_p}$, $\widetilde{\mu_q}$ and \cite[Lemma~2.1]{Dresda1} we have
$H^{1,p,q}\hookrightarrow L^{\widetilde{m_p}}(\Omega)\times L^{\widetilde{\mu_q}}(\Gamma_1)$ with dense embedding.
Moreover, since $H^{1,\lomega,\lgamma}\hookrightarrow H^{1,p,q}$, we have $u\in C([0,T_{\text{max}});H^{1,p,q})$.
Next, since by definition $\widetilde{m_p}\le m_p$ and $\widetilde{\mu_q}\le \mu_q$, by \eqref{ZLMP} we have
$u'\in L^1_{\loc}([0,T_{\max});L^{\widetilde{m_p}}(\Omega)\times
L^{\widetilde{\mu_q}}(\Gamma_1))$.

Then, by Lemma~\ref{chainrule}, we have $J\cdot u,\, \cal{I}\cdot u\in W^{1,1}_{\loc}([0,T_{\max}))$  and
\begin{align}\label{potentialderivative}
J(u(t))-J(u(s))=&\int_s^t\left[\int_\Omega f(\cdot,u)u_t+\int_{\Gamma_1} g(\cdot,u_{|\Gamma})(u_{|\Gamma})_t\right],\\\label{potentialderivative1}
\cal{I}(u(t))-\cal{I}(u(s))=&\int_s^t\left[\int_\Omega \mathfrak{f}(\cdot,u)u_t+\int_{\Gamma_1}\mathfrak{g}(\cdot,u_{|\Gamma})(u_{|\Gamma})_t\right],
\end{align}
for all $s,t\in [0,T_{\max})$,
where $\frak{f}$ and $\frak{g}$ are given by \eqref{frakfg}.

We also introduce the energy functional $\cal{E}\in C^1(\cal{H}^{p,q})$
defined for $(v,w)\in\cal{H}^{p,q}$ by
 \begin{equation}\label{energyfunction}
\cal{E}(v,w)=\frac 12 \|w\|_{H^0}^2+\tfrac 12\int_\Omega |\nabla v|^2+\tfrac 12\int_{\Gamma_1}|\nabla_{\Gamma} v|_{\Gamma}^2-J(v),
 \end{equation}
  and the energy function associated to $u$ by
 \begin{equation}\label{energyfunctionbis}
 \cal{E}_{u}(t)=\cal{E}(u(t),u'(t)),\qquad\text{for all $t\in [0,T_{\max})$}.
 \end{equation}
By \eqref{potentialderivative} and \eqref{energyfunction} the energy identity \eqref{energyidentity}
can be rewritten as
\begin{equation}\label{eniddd}
\cal{E}_u(t)-\cal{E}_u(s)+\int_s^t\langle B(u'),u'\rangle_Y=0\quad\text{for all $s,t\in [0,T_{\max})$.}
\end{equation}
Consequently,  by   \eqref{2.5}, \eqref{energyfunction} and \eqref{energyfunctionbis}, for $t\in [0,T_{\max})$ we have
\begin{equation}\label{nnn6.6}
\tfrac 12 \|u'(t)\|_{H^0}^2+\tfrac 12 \|u(t)\|_{H^1}^2\le \cal{E}_u(0)+\tfrac 12 \|u(t)\|_{H^0}^2+J(u(t))
-\int_0^t\langle B(u'),u'\rangle_Y.
\end{equation}

The proof can then be completed, starting from \eqref{nnn6.6}, as in \cite[Proof of Theorem 6.2]{Dresda1}, since
\eqref{nnn6.6} is nothing but (the correct form of) formula \cite[(6.18)]{Dresda1}. For the reader's convenience we repeat in the sequel the arguments used there.

We introduce an auxiliary function associated to $u$ by
\begin{equation}\label{6.105}
\Upsilon(t)=\tfrac 12 \|u'(t)\|_{H^0}^2+\tfrac 12 \|u(t)\|_{H^1}^2+\cal{I}(u(t)),\qquad\text{for all $t\in [0,T_{\max})$}.
\end{equation}
By \eqref{nnn6.6} and \eqref{6.105} we have
\begin{equation}\label{6.106}
\Upsilon(t)\le\cal{E}_u(0)+\tfrac 12 \|u(t)\|_{H^0}^2+J(u(t))+\cal{I}(u(t))-\int_0^t\langle B(u'),u'\rangle_W.
\end{equation}
By \eqref{6.106bis} and assumption (FGQP1)  we get
\begin{equation}\label{new}
J(v)\le \left[C_{p_1}|\Omega|+C_{q_1}\sigma(\Gamma)\right]\, \left(1+\|v\|_{H^0}^2\right)+\cal{I}(v)\qquad\text{for all $v\in H^1$.}
\end{equation}
By \eqref{6.106}-- \eqref{new} we thus obtain
\begin{equation}\label{new2}
\Upsilon(t)\le \cal{E}_u(0)+k_{18}+k_{18}\|u(t)\|_{H^0}^2+2\cal{I}(u(t))-\int_0^t\langle B(u'),u'\rangle_W,
\end{equation}
where $k_{18}=C_{p_1}|\Omega|+C_{q_1}\sigma(\Gamma)+1/2$. Writing $\|u(t)\|_{H^0}^2=\|u_0\|_{H^0}^2+2\int_0^t (u',u)_{H^0}$ in \eqref{new2}
and using  \eqref{6.106bis} and \eqref{potentialderivative1} we get
\begin{equation}\label{6.107}
\begin{aligned}
\Upsilon(t)\le & K_{31}+\int_0^t \left[2k_{18}(u',u)_{H^0}-\langle B(u'),u'\rangle_W\phantom{\int}\right.\\+& \left.2p_1C_{p_1}\int_\Omega \alpha|u|^{p_1-2}uu_t+2q_1C_{q_1}\int_{\Gamma_1}\beta|u|^{q_1-2}u(u_{|\Gamma})_t\right],
\end{aligned}
\end{equation}
where $K_{31}=K_{31}(u_0,u_1)=\cal{E}_u(0)+2\cal{I}(u_0)+k_{18}(1+\|u_0\|_{H^0}^2)$.
Consequently, by assumption (PQ3),  Cauchy--Schwartz  and Young inequalities, we get the preliminary estimate
\begin{equation}\label{6.108}
\begin{aligned}
\Upsilon(t)\le &K_{31}+\int_0^t\!\!\left[- c_m'\|[u_t]_\alpha\|_{m,\alpha}^m\negquad-c_\mu'\|[(u_\Gamma)_t]_\beta\|_{\mu,\beta}^\mu\!\!+k_{18}\left(\|u'\|_{H^0}^2+\|u\|_{H^0}^2\right)\right.\\+&\left.2p_1C_{p_1}\int_\Omega \alpha|u_t| |u|^{p-1}|u_t|+2q_1C_{q_1}\int_{\Gamma_1}\beta |u|^{q-1}|(u_{|\Gamma})_t|
|(u_{|\Gamma})_t|\right]
\end{aligned}
\end{equation}
for all $t\in[0,T_{\max})$.
We now estimate, a.e. in $[0,T_{\max})$, the last three integrands  in the right--hand side of \eqref{6.108}.
By \eqref{6.105} we get
\begin{equation}\label{6.109}
k_{18}\|u'\|_{H^0}^2\le 2k_{18} \Upsilon.
\end{equation}
Moreover, by the embedding $H^1(\Omega;\Gamma)\hookrightarrow L^2(\Omega)\times L^2(\Gamma)$,
\begin{equation}\label{6.110}
 \|u\|_{H^0}^2\le k_{19} \|u\|_{H^1}^2.
\end{equation}
Consequently, by \eqref{6.105},
\begin{equation}\label{6.111}
k_{18}\|u\|_{H^0}^2\le  k_{20}\Upsilon.
\end{equation}
To estimate the addendum $2p_1C_{p_1}\int_\Omega \alpha|u|^{p_1-1}|u_t|$ we now distinguish between the cases $p_1=2$ and $p_1>2$.
When $p_1=2$,  by \eqref{6.105}, \eqref{6.111} and Young inequality,
\begin{equation}\label{6.112}
2p_1C_{p_1}\int_\Omega \alpha|u|^{p-1}|u_t|\le p_1C_{p_1}\|\alpha\|_\infty(\|u\|_{H^0}^2+\|u'\|_{H^0}^2)\le k_{21}\Upsilon,
\end{equation}
where $k_{21}=2p_1C_{p_1}\|\alpha\|_\infty (1+k_{19})$.

When $p_1>2$, for any $\eps\in (0,1]$ to be fixed later, by  weighted Young inequality
\begin{equation}\label{6.114}
2p_1C_{p_1}\int_\Omega \alpha|u|^{p_1-1}|u_t|\le 2(p_1-1)C_{p_1}\eps^{1-p_1'}\int_\Omega \alpha|u|^{p_1}+2\eps C_{p_1}\int_\Omega \alpha |u_t|^{p_1}.
\end{equation}
By \eqref{6.105} we have
\begin{equation}\label{6.114bis}
2(p_1-1)C_{p_1}\eps^{1-p_1'}\int_\Omega \alpha|u|^{p_1}\le 2(p_1-1)\eps^{1-p_1'}\Upsilon.
\end{equation}
Moreover by \eqref{V.2} we have $p_1\le m=\overline{m}$ and consequently  $|u_t|^{p_1}\le 1+|u_t|^m$ a.e. in $\Omega$, which yields
\begin{equation}\label{6.115}
\int_\Omega \alpha |u_t|^{p_1}\le \int_\Omega\alpha+\int_\Omega \alpha |u_t|^m\le \|\alpha\|_\infty |\Omega|+\|[u_t]_\alpha\|_{m,\alpha}^m.
\end{equation}
Plugging \eqref{6.114bis} and \eqref{6.115} in \eqref{6.114} we get, as $\eps\le1$,
\begin{equation}\label{6.116}
2p_1C_{p_1}\int_\Omega \alpha|u|^{p_1-1}|u_t|\le k_{22}\left(\eps^{1-p_1'}\Upsilon+\eps \|[u_t]_\alpha\|_{m,\alpha}^m+1\right).
\end{equation}
Comparing \eqref{6.110} and \eqref{6.116} we get that for $p\ge 2$ we have
\begin{equation}\label{6.117}
2p_1C_{p_1}\int_\Omega \alpha|u|^{p_1-1}|u_t|\le k_{23}\left[(1+\eps^{1-p_1'})\Upsilon+\eps \|[u_t]_\alpha\|_{m,\alpha}^m+1\right].
\end{equation}
We estimate the last integrand in the right--hand side of
\eqref{6.108} by transposing from $\Omega$ to $\Gamma_1$ the arguments used to get \eqref{6.117}. At the end we get
\begin{equation}\label{6.118}
2q_1C_{q_1}\int_{\Gamma_1} \beta|u|^{q_1-1}|(u_{|\Gamma})_t|\le
k_{24}\left[(1+\eps^{1-q_1'})\Upsilon+\eps \|[(u_{|\Gamma})_t]_\beta\|_{\mu,\beta,\Gamma_1}^\mu+1\right].
\end{equation}
Plugging estimates \eqref{6.109}, \eqref{6.111}, \eqref{6.117} and \eqref{6.118} into \eqref{6.108} we get
 \begin{multline}\label{6.119}
   \Upsilon(t)\le K_{31}+\int_0^t \left[(k_{23}\eps-c_m')\|[u_t]_\alpha\|_{m,\alpha}^m+(k_{24}\eps-c_\mu')\|[(u_\Gamma)_t]_\beta\|_{\mu,\beta}^\mu\right]\\+
   k_{25}\int_0^t\left[(1+\eps^{1-p_1'}+\eps^{1-q_1'})\Upsilon+1\right]\qquad\text{for all $t\in [0,T_{\max})$.}
 \end{multline}
 Fixing $\eps=\eps_1$, where $\eps_1=\min\{1, c_m'/k_{23}, c_\mu'/k_{24}\}$, and setting $K_{32}=K_{32}(u_0,u_1)=K_{31}(u_0,u_1)+k_{25}(1+\eps_1^{1-p_1'}+\eps_1^{1-q_1'})$, the estimate
 \eqref{6.119} reads as
 $$\Upsilon(t)\le K_{32}(1+t)+K_{32}\int_0^t \Upsilon(s)\,ds\qquad\text{for all $t\in [0,T_{\max})$.}$$
 Then, since $T_{\max}<\infty$, by Gronwall Lemma (see \cite[Lemma 4.2, p. 179]{showalter}),   $\Upsilon$ is bounded in $[0,T_{\max})$, getting, by
 \eqref{CCCP1234} and \eqref{6.105}, the desired contradiction.
\end{proof}
We now state and prove, for the sake of clearness, two corollaries
of Theorem~\ref{theorem7.1}--(i)  which generalize
Corollaries~\ref{corollary1.5}--\ref{corollary1.6} in the
introduction. The discussion made there applies here as
well.
\begin{cor} \label{corollary7.1} Suppose that  (PQ1--3), (FG1),
(FGQP1--2) and \eqref{R'} hold. Then for any $(u_0,u_1)\in H^{1,p,q}\times H^0$
problem \eqref{1.1} has a global weak solution $u\in C([0,\infty);H^{1,p,q})$.
\end{cor}
\begin{proof} Since \eqref{R'} can be written also as
$p\not=1+\romega/m'$ when $p>1+\romega/2$ and $q\not=1+\rgamma/\mu'$
when $q>1+\rgamma/2$, clearly assumption (FG2) can be skipped. Moreover, by  \eqref{lomegagamma},  when \eqref{R'} holds
we have $H^{1,\sigmaomega,\sigmagamma}=H^{1,p,q}$. \end{proof}
\begin{cor}\label{corollary7.2} Suppose that (PQ1--3), (FG1--2),
(FGQP1--2) and \eqref{R''bis} hold. Then
for any $(u_0,u_1)\in H^{1,p,q}\times H^0$
problem \eqref{1.1} has a global weak solution $u\in C([0,\infty);H^{1,p,q})$.
\end{cor}
\begin{proof} When \eqref{R''bis} holds by \eqref{lomegagamma} we have  $H^{1,\lomega,\lgamma}=H^{1,p,q}$. \end{proof}

We now state, for the reader convenience, the global--in--time version of the more general local analysis made in Corollaries~\ref{corollary5.3},
\ref{corollary6.1} and Theorem~\ref{theorem6.4}, simply obtained
by combining them with Theorem~\ref{theorem7.1}.

\begin{cor}[\bf Global analysis in the scale of spaces]  \label{corollary7.3}
The following conclusions hold true.
\renewcommand{\labelenumi}{{(\roman{enumi})}}
\begin{enumerate}
\item {\bf (Global existence)} Under assumptions (PQ1--3), (FG1--2),
(FGQP1--2) the
main conclusion of Theorem~\ref{theorem7.1}--(i) hold when
$H^{1,\lomega,\lgamma}$ is replaced by
\renewcommand{\labelenumii}{{(\roman{enumi}.\arabic{enumii})}}
\begin{enumerate}
\item  $H^{1,\rho,\theta}_{\alpha,\beta}$, provided $\rho$ and $\theta$ satisfy
\eqref{ranges4}, and by
\item   $H^{1,\rho,\theta}$, provided $\rho$ and $\theta$ satisfy \eqref{spaziuguali} and \eqref{ranges4}.
\end{enumerate}
\item {\bf (Global existence--uniqueness)} Under assumptions (PQ1--3), (FG1), (FG2)$'$,
(FGQP1--2) the main conclusion of Theorem~\ref{theorem7.1}--(ii) holds when the space
$H^{1,\somega,\sgamma}$ is replaced by
\renewcommand{\labelenumii}{{(\roman{enumi}.\arabic{enumii})}}
\begin{enumerate}
\item $H^{1,\rho,\theta}_{\alpha,\beta}$, provided $\rho$ and $\theta$ satisfy \eqref{ranges2}
and by
\item $H^{1,\rho,\theta}$, provided $\rho$ and $\theta$ satisfy \eqref{spaziuguali} and \eqref{ranges2}.
\end{enumerate}
\item {\bf (Global Hadamard well--posedness)}  Under assumptions (PQ1--3), (PQ4)$'$, (FG1), (FGQP1),
(FG2)$'$ and \eqref{R''} problem \eqref{1.1} is locally well--posed
\renewcommand{\labelenumii}{{(\roman{enumi}.\arabic{enumii})}}
\begin{enumerate}
\item in $H^{1,\rho,\theta}_{\alpha,\beta}\times H^0$ when $\rho$ and $\theta$ satisfy \eqref{ranges3} (that is the conclusions
of Theorem~\ref{theorem7.1}--(iii) hold true when $H^{1,s_1,s_2}$ is replaced
by $H^{1,\rho,\theta}_{\alpha,\beta}$), and
\item $H^{1,\rho,\theta}\times H^0$ when $\rho,\theta$ satisfy
\eqref{spaziuguali} and \eqref{ranges3} (that is the conclusions
of Theorem~\ref{theorem7.1}--(iii) hold true when $H^{1,s_1,s_2}$ is replaced
by $H^{1,\rho,\theta}$).
\end{enumerate}
\end{enumerate}
\end{cor}
\begin{proof}[Proof of Theorem~\ref{theorem1.5} and
Corollaries~\ref{corollary1.5}--\ref{corollary1.7} in Section~\ref{intro}.] When
$$P(x,v)=\alpha(x)P_0(v),\quad Q(x,v)=\beta(x)Q_0(v),\quad f(x,u)=f_0(u),\quad
g(x,u)=g_0(u),$$ by Remarks~\ref{remark3.1}--\ref{remark3.3}
assumptions (PQ1--3), (FG1), (FGQP1) reduce to (I--III). Moreover
by Remark~\ref{remark3.6} (FG2) and (FG2)$'$  reduce to
(IV) and (IV)$'$. By Remark~\ref{remark3.9} (PQ4) and (PQ4)$'$
reduce to (V) and (V)$'$, while by Remark~\ref{remark6.4} assumption (FGQP2) reduces to (VI).  Hence
Theorem~\ref{theorem1.5} and
Corollaries~\ref{corollary1.5}--\ref{corollary1.6} are particular cases  of
Theorem~\ref{theorem7.1} and
Corollaries~\ref{corollary7.1}--\ref{corollary7.3}.
\end{proof}

\appendix

\section{Proof of Lemma~\ref{lemma5.4}}\label{appendixB}
To prove Lemma~\ref{lemma5.4} we first recall the following
elementary result, which proof is given only for the reader's
convenience.
\begin{lem}\label{lemmaA1} Let $0<T<\infty$ and $Y_1,Y_2$ be two
Banach spaces with $Y_2$ densely embedded in $Y_1$. Then
$C^1_c((0,T);Y_2)$ is dense in $C^1_c((0,T);Y_1)$ with respect to
the norm of $C^1([0,T];Y_1)$.
\end{lem}
\begin{proof} Let $u\in C^1_c((0,T);Y_1)$ and $\eta\in (0,T/2)$ such
that $\text{supp }u\subset [\eta,T-\eta]$. Since $u$ is uniformly
continuous and  $Y_2$ is dense in $Y_1$ one can easily build a
sequence of piecewise linear functions $(v_n)_n$ in $C_c((0,T);Y_2)$
such that $v_n\to \dot{u}$ in $C([0,T];Y_1)$. Hence setting
$w_n(t)=\int_0^t v_n(\tau)\,d\tau$ for $\tau\in [0,T]$ we have
$w_n\in C^1([0,T];Y_2)$ and, since $w_n(0)=u(0)=0$, $w_n\to u$ in
$C^1([0,T];Y_1)$. Consequently $w_n\to 0$ in $C^1([0,\eta]\cup
[T-\eta,T];Y_1)$.

Taking a standard cut--off function $\xi\in C^\infty_c(\R)$ such
that $0\le \xi\le 1$, $\xi=1$ in $ [\eta,T-\eta]$ and $\xi=0$ in
$(-\infty,\eta/2)\cup (T-\eta/2, \infty)$ and setting $u_n=\xi w_n$
we trivially have $u_n\in C^1([0,T];Y_2)$,  $\text{supp }u_n\subset
[\eta/2,T-\eta/2]$. Moreover  $u_n=w_n\to u$ in
$C^1([\eta,T-\eta];Y_1)$, while $u_n\to 0$ in $C^1([0,\eta]\cup
[T-\eta,T];Y_1)$. Since $u=0$ in $[0,\eta]\cup [T-\eta,T]$ we get
$u_n\to u$ in $C^1([0,T];Y_1)$.
\end{proof}
\begin{proof}[Proof of Lemma~\ref{lemma5.4}.] Given  $\varphi\in C_c((0,T);H^1)\cap C^1_c((0,T);H^0)\cap
Z(0,T)$, trivially extended  to $\varphi\in C_c(\R;H^1)\cap
C^1_c(\R;H^0)\cap Z(-\infty,\infty)$, by standard time
regularization we build a sequence $(\psi_n)_n$ in
$C^1_c((0,T;X)$ such that $\psi_n\to \varphi$ in the norm of
$C([0,T];H^1)\cap C^1([0,T];H^0)\cap Z(0,T)$, where
$X=H^{1,\overline{m},\overline{\mu}}_{\alpha,\beta}$ (see
\eqref{3.5}). Since by
\cite[Lemma~2.1]{Dresda1} $H^{1,\infty,\infty}$ is dense in $X$  it follows from Lemma~\ref{lemmaA1} that
$C^1_c((0,T);H^{1,\infty,\infty})$ is dense in $C^1_c((0,T);X)$ with
respect to the norm of $C^1([0,T];X)$ and then,  since $X=H^1\cap
[L^{2,\overline{m}}_\alpha(\Omega)\times
L^{2,\overline{\mu}}_\beta(\Gamma_1)]$, with respect to the norm of
$C([0,T];H^1)\cap C^1([0,T];H^0)\cap Z(0,T)$.
\end{proof}
\renewcommand{\arraystretch}{1.3}
\setlength{\arrayrulewidth}{0.7pt}
\begin{table}[h]
\begin{center}\label{Table5}
\begin{tabular}{|c|c|}
\multicolumn{2}{>{\normalsize}c}{\sc Table 5. \rm Further results when $N=3$ and $\essinf_\Omega \alpha>0$}\\
\hline
                     &$2\le q<\infty$ \\
\hline
$2\le p\le 4$        &   \\
\cline{1-1}
$4<p<6$        &  well--posedness in $H^{1,\rho,2}$ for $\rho\in [6,6\vee m]$\\
\hline
$6=p<m$           &  existence--uniqueness in $H^{1,\rho,2}$ for $\rho\in [6,m]$\\
\cdashline{2-2}[1pt/1pt]
                  &  well--posedness  in $H^{1,\rho,2}$ for $\rho\in (6,m]$\\
\hline
$6=p=m$           &  existence--uniqueness in $H^1$\\
\hline
$6<p<1+6/m'$      &  local existence in $H^{1,\rho,2}$ for $\rho\in [6,m]$\\
\cdashline{2-2}[1pt/1pt]
                  &  global existence in $H^{1,\rho,2}$ for $\rho\in [p,m]$\\
\cdashline{2-2}[1pt/1pt]
                  &  existence--uniqueness in $H^{1,\rho,2}$ for $\rho\in [3(p-2)/2,m]$\\
\cdashline{2-2}[1pt/1pt]
                  &  well--posedness  in $H^{1,\rho,2}$ for $\rho\in \left(3(p-2)/2,m\right]$\\
\hline
$6<p=1+6/m'$      & existence--uniqueness in $H^{1,\rho,2}$ for $\rho\in [3(p-2)/2,m]$\\
\cdashline{2-2}[1pt/1pt]
                  &  well--posedness  in $H^{1,\rho,2}$ for $\rho\in (3(p-2)/2,m]$\\
\hline
\end{tabular}
\end{center}
\medskip
\end{table}
\setlength{\tabcolsep}{2pt}
\renewcommand{\arraystretch}{0.1}
\setlength{\arrayrulewidth}{0.7pt}
\begin{table}
\begin{center}\label{Table6}
\begin{tabular}{|>{\tiny}l|>{\tiny}c|>{\tiny}c|>{\tiny}c|>{\tiny}c|>{\tiny}c|}
\multicolumn{6}{>{\normalsize}c}{\sc Table 6. \rm Further results when $N=4$ and $\essinf_\Omega \alpha, \essinf_{\Gamma_1}\beta>0$}\\
\multicolumn{6}{>{\normalsize}c}{}\\
\multicolumn{6}{>{\normalsize}c}{}\\
\multicolumn{6}{>{\normalsize}c}{}\\
\hline
                                 &\raisebox{1.5ex}[12pt]{\tiny{$2\le q\le 4$}}\vline\raisebox{1.5ex}[12pt]{\tiny{$4<q<6$}}   &  \raisebox{1.5ex}[12pt]{\tiny{$6=q<\mu$}}                            & \raisebox{1.5ex}[12pt]{\tiny{$6=q=\mu$}}                 & \raisebox{1.5ex}[12pt]{\tiny{$6<q<1+ 6/\mu'$}}                                  & \raisebox{1.5ex}[12pt]{\tiny{$6<q=1+6/\mu'$}}            \\
\hline
                     &                                                  &                                                    &                                  &local                                                     &                                               \\
                     &                                                  &                                                    &                                  &existence                                                 &                                               \\
                     &                                                  &                                                    &                                  &in $H^{1,\rho,\theta}$ for                                &                                               \\
                     &                                                  &                                                    &                                  &$\rho\in   [4,4\vee m]$                                   &                                               \\
                     &                                                  &                                                    &                                  &$\theta\in [6,\mu]$                                       &                                               \\
\cdashline{5-5}[1pt/1pt]
$2\le p\le 3$        &                                                  &                                                    &                                  &global                                                    &                                               \\
                     &                                                  &                                                    &                                  &existence                                                 &                                               \\
                     &                                                  &                                                    &                                  &in $H^{1,\rho,\theta}$ for                                &                                               \\
                     &                                                  &                                                    &                                  &$\rho\in   [4,4\vee m]    $                               &                                               \\
                     &                                                  &                                                    &                                  &$\theta\in [q,\mu]$                                       &                                               \\
\cline{1-1}\cdashline{5-5}[1pt/1pt]
                     &                                                  &existence--                                         &existence--                       &existence--                                               &existence--                                    \\
                     &                                                  &uniqueness                                          &uniqueness                        &uniqueness                                                &uniqueness                                       \\
                     &                                                  &in $H^{1,\rho,\theta}$ for                          &in $H^{1,\rho,2}$ for             &in $H^{1,\rho,\theta}$ for                                &in $H^{1,\rho,\theta}$ for                       \\
                     &                                                  &$\rho\in   [4,4\vee m]$                             &$\rho\in   [4,4\vee m]$           &$\rho\in   [4,4\vee m]$                                   &$\rho\in   [4,4\vee m]$                        \\
                     &                                                  &$\theta\in [6,\mu]$                                 &                                  &$\theta\in \left[\frac 32(q-2),\mu\right]$                &$\theta\in \left[\frac 32(q-2),\mu\right]$    \\
\cdashline{3-3}[1pt/1pt]\cdashline{5-6}[1pt/1pt]
$3<p<4$              & well--                                           & well--                                             &                                  &well--                                                    &well--                                         \\
                     & posedness                                        & posedness                                          &                                  &posedness                                                 &posedness                                        \\
                     &in $H^{1,\rho,\theta}$ for                        & in $H^{1,\rho,\theta}$ for                         &                                  &in $H^{1,\rho,\theta}$ for                                &in $H^{1,\rho,\theta}$ for                    \\
                     &$\rho\in   [4,4\vee m]$                           &$\rho\in   [4,4\vee m]$                             &                                  &$\rho\in   [4,4\vee m]$                                   &$\rho\in   [4,4\vee m]$                        \\
                     &$\theta\in [6,6\vee\mu]$                          &$\theta\in (6,\mu]$                                 &                                  &$\theta\in \left(\frac 32(q-2),\mu\right]$                &$\theta\in \left(\frac 32(q-2),\mu\right]$    \\
\hline
                     &                                                  &                                                    &                                  &local                                                     &                                               \\
                     &                                                  &                                                    &                                  &existence                                                 &                                               \\
                     &                                                  &                                                    &                                  &in $H^{1,\rho,\theta}$ for                                &                                               \\
                     &                                                  &                                                    &                                  &$\rho\in   [4,m]$                                         &                                               \\
                     &                                                  &                                                    &                                  &$\theta\in [6,\mu]$                                       &                                               \\
\cdashline{5-5}[1pt/1pt]
                     &                                                  &                                                    &                                  &global                                                    &                                               \\
                     &                                                  &                                                    &                                  &existence                                                 &                                               \\
                     &                                                  &                                                    &                                  &in $H^{1,\rho,\theta}$ for                                &                                               \\
                     &                                                  &                                                    &                                  &$\rho\in   [4,m]$                                         &                                               \\
                     &                                                  &                                                    &                                  &$\theta\in [q,\mu]$                                       &                                               \\
\cdashline{5-5}[1pt/1pt]
$4=p<m$              &existence--                                       &existence--                                         &existence--                       &existence--                                               &existence--                                    \\
                     &uniqueness                                        &uniqueness                                          &uniqueness                        &uniqueness                                                &uniqueness                                       \\
                     &in $H^{1,\rho,\theta}$ for                        &in $H^{1,\rho,\theta}$ for                          &in $H^{1,\rho,2}$ for             &in $H^{1,\rho,\theta}$ for                                &in $H^{1,\rho,\theta}$ for                       \\
                     &$\rho\in   [4,m]$                                 &$\rho\in   [4,m]$                                   &$\rho\in [4,m]$                   &$\rho\in   [4,m]$                                         &$\rho\in   [4,m]$                            \\
                     &$\theta\in [6,6\vee\mu]$                          &$\theta\in [6,\mu]$                                 &                                  &$\theta\in \left[\frac 32(q-2),\mu\right]$                &$\theta\in \left[\frac 32(q-2),\mu\right]$    \\
\cdashline{2-3}[1pt/1pt]\cdashline{5-6}[1pt/1pt]
                     & well--                                           & well--                                             &                                  &well--                                                    &well--                                         \\
                     & posedness                                        & posedness                                          &                                  &posedness                                                 &posedness                                        \\
                     &in $H^{1,\rho,\theta}$ for                        & in $H^{1,\rho,\theta}$ for                         &                                  &in $H^{1,\rho,\theta}$ for                                &in $H^{1,\rho,\theta}$ for                    \\
                     &$\rho\in   (4, m]$                                &$\rho\in(4,m]$                                      &                                  &$\rho\in   (4,m]$                                         &$\rho\in   (4,m]$                        \\
                     &$\theta\in [6,6\vee\mu]$                          &$\theta\in (6,\mu]$                                 &                                  &$\theta\in \left(\frac 32(q-2),\mu\right]$                &$\theta\in \left(\frac 32(q-2),\mu\right]$    \\
\hline
                     &                                                  &                                                    &                                  &local                                                     &                                               \\
                     &                                                  &                                                    &                                  &existence                                                 &                                               \\
                     &                                                  &                                                    &                                  &in $H^{1,2,\theta}$ for                                   &                                               \\
                     &                                                  &                                                    &                                  &$\theta\in [6,\mu]$                                       &                                               \\
\cdashline{5-5}[1pt/1pt]
                     &                                                  &                                                    &                                  &global                                                    &                                               \\
                     &                                                  &                                                    &                                  &existence                                                 &                                               \\
$4=p=m$              &                                                  &                                                    &                                  &in $H^{1,2,\theta}$ for                                   &                                               \\
                     &                                                  &                                                    &                                  &$\theta\in [q,\mu]$                                       &                                               \\
\cdashline{5-5}[1pt/1pt]
                     &existence--                                       &existence--                                         &existence--                       &existence--                                               &existence--                                    \\
                     &uniqueness                                        &uniqueness                                          &uniqueness                        &uniqueness                                                &uniqueness                                       \\
                     &in $H^{1,2,\theta}$ for                           &in $H^{1,2,\theta}$ for                             &in $H^1$                          &in $H^{1,2,\theta}$ for                                   &in $H^{1,2,\theta}$ for                       \\
                     &$\theta\in [6,6\vee\mu]$                          &$\theta\in [6,\mu]$                                 &                                  &$\theta\in \left[\frac 32(q-2),\mu\right]$                &$\theta\in \left[\frac 32(q-2),\mu\right]$    \\
\hline
                     &local                                             &local                                               &local                             &local                                                     &local                                            \\
                     &existence                                         &existence                                           &existence                         &existence                                                 &existence                                                   \\
                     &in $H^{1,\rho,\theta}$ for                        &in $H^{1,\rho,\theta}$ for                          &in $H^{1,\rho,2}$ for             &in $H^{1,\rho,\theta}$ for                                &in $H^{1,\rho,\theta}$ for                         \\
                     &$\rho\in[4,m]$                                    &$\rho\in[4,m]$                                      &$\rho\in[4,m]$                    &$\rho\in[4,m]$                                            &$\rho\in[4,m]$                            \\
                     &$\theta\in [6,6\vee\mu]$                          &$\theta\in [6,\mu]$                                 &                                  &$\theta\in [6,\mu]$                                       &$\theta\in \left[\frac 32(q-2),\mu\right]$    \\
\cdashline{2-6}[1pt/1pt]
                     &global                                            &global                                              &global                            &global                                                    &global                                           \\
                     &existence                                         &existence                                           &existence                         &existence                                                 &existence                                                   \\
                     &in $H^{1,\rho,\theta}$ for                        &in $H^{1,\rho,\theta}$ for                          &in $H^{1,\rho,2}$ for             &in $H^{1,\rho,\theta}$ for                                &in $H^{1,\rho,\theta}$ for                         \\
                     &$\rho\in[p,m]$                                    &$\rho\in[p,m]$                                      &$\rho\in[p,m]$                    &$\rho\in[p,m]$                                            &$\rho\in[p,m]$                            \\
$4<p<1+4/m'$         &$\theta\in [6,6\vee\mu]$                          &$\theta\in [6,\mu]$                                 &                                  &$\theta\in [q,\mu]$                                       &$\theta\in \left[\frac 32(q-2),\mu\right]$    \\
\cdashline{2-6}[1pt/1pt]
                     &existence--                                       &existence--                                         &existence--                       &existence--                                               &existence--                                    \\
                     &uniqueness                                        &uniqueness                                          &uniqueness                        &uniqueness                                                &uniqueness                                       \\
                     &in $H^{1,\rho,\theta}$ for                        &in $H^{1,\rho,\theta}$ for                          &in $H^{1,\rho,2}$ for             &in $H^{1,\rho,\theta}$ for                                &in $H^{1,\rho,\theta}$ for                       \\
                     &$\rho\in   [2(p-2),m]$                            &$\rho\in   [2(p-2),m]$                              &$\rho\in [2(p-2),m]$              &$\rho\in   [2(p-2),m]$                                    &$\rho\in   [2(p-2),m]$                            \\
                     &$\theta\in [6,6\vee\mu]$                          &$\theta\in [6,\mu]$                                 &                                  &$\theta\in \left[\frac 32(q-2),\mu\right]$                &$\theta\in \left[\frac 32(q-2),\mu\right]$    \\
\cdashline{2-3}[1pt/1pt]\cdashline{5-6}[1pt/1pt]
                     & well--                                           & well--                                             &                                  &well--                                                    &well--                                         \\
                     & posedness                                        & posedness                                          &                                  &posedness                                                 &posedness                                        \\
                     &in $H^{1,\rho,\theta}$ for                        & in $H^{1,\rho,\theta}$ for                         &                                  &in $H^{1,\rho,\theta}$ for                                &in $H^{1,\rho,\theta}$ for                    \\
                     &$\rho\in   (2(p-2),m]$                            &$\rho\in   (2(p-2),m]$                              &                                  &$\rho\in   (2(p-2),m]$                                    &$\rho\in   (2(p-2),m]$                        \\
                     &$\theta\in [6,6\vee\mu]$                          &$\theta\in (6,\mu]$                                 &                                  &$\theta\in \left(\frac 32(q-2),\mu\right]$                &$\theta\in \left(\frac 32(q-2),\mu\right]$    \\
\hline
                     &                                                  &                                                    &                                  &local                                                     &                                               \\
                     &                                                  &                                                    &                                  &existence                                                 &                                               \\
                     &                                                  &                                                    &                                  &in $H^{1,\rho,\theta}$ for                                &                                               \\
                     &                                                  &                                                    &                                  &$\rho\in   [2(p-2),m]$                                    &                                               \\
                     &                                                  &                                                    &                                  &$\theta\in [6,\mu]$                                       &                                               \\
\cdashline{5-5}[1pt/1pt]
                     &                                                  &                                                    &                                  &global                                                    &                                               \\
                     &                                                  &                                                    &                                  &existence                                                 &                                               \\
                     &                                                  &                                                    &                                  &in $H^{1,\rho,\theta}$ for                                &                                               \\
                     &                                                  &                                                    &                                  &$\rho\in   [2(p-2),m]$                                    &                                               \\
$4<p=1+4/m'$         &                                                  &                                                    &                                  &$\theta\in [q,\mu]$                                       &                                               \\
\cdashline{5-5}[1pt/1pt]
                     &existence--                                       &existence--                                         &existence--                       &existence--                                               &existence--                                    \\
                     &uniqueness                                        &uniqueness                                          &uniqueness                        &uniqueness                                                &uniqueness                                       \\
                     &in $H^{1,\rho,\theta}$ for                        &in $H^{1,\rho,\theta}$ for                          &in $H^{1,\rho,2}$ for             &in $H^{1,\rho,\theta}$ for                                &in $H^{1,\rho,\theta}$ for                       \\
                     &$\rho\in   [2(p-2),m]$                            &$\rho\in   [2(p-2),m]$                              &$\rho\in [2(p-2),m]$              &$\rho\in   [2(p-2),m]$                                    &$\rho\in   [2(p-2),m]$                            \\
                     &$\theta\in [6,6\vee\mu]$                          &$\theta\in [6,\mu]$                                 &                                  &$\theta\in \left[\frac 32(q-2),\mu\right]$                &$\theta\in \left[\frac 32(q-2),\mu\right]$    \\
\cdashline{2-3}[1pt/1pt]\cdashline{5-6}[1pt/1pt]
                     & well--                                           & well--                                             &                                  &well--                                                    &well--                                         \\
                     & posedness                                        & posedness                                          &                                  &posedness                                                 &posedness                                        \\
                     &in $H^{1,\rho,\theta}$ for                        & in $H^{1,\rho,\theta}$ for                         &                                  &in $H^{1,\rho,\theta}$ for                                &in $H^{1,\rho,\theta}$ for                    \\
                     &$\rho\in   (2(p-2),m]$                            &$\rho\in   (2(p-2),m]$                              &                                  &$\rho\in   (2(p-2),m]$                                    &$\rho\in   (2(p-2),m]$                        \\
                     &$\theta\in [6,6\vee\mu]$                          &$\theta\in (6,\mu]$                                 &                                  &$\theta\in \left(\frac 32(q-2),\mu\right]$                &$\theta\in \left(\frac 32(q-2),\mu\right]$    \\
\hline
\end{tabular}
\end{center}
\end{table}
\setlength{\tabcolsep}{1pt}
\renewcommand{\arraystretch}{0.1}
\setlength{\arrayrulewidth}{0.7pt}
\begin{table}
\begin{center}\label{Tables7-8}
\begin{tabular}{|>{\tiny}l|>{\tiny}c|>{\tiny}c|>{\tiny}c|>{\tiny}c|>{\tiny}c|}
\multicolumn{6}{>{\normalsize}c}{\sc Table 7. \rm Further results when $N=5$  and $\essinf_\Omega \alpha, \essinf_{\Gamma_1}\beta>0$}\\
\multicolumn{6}{>{\normalsize}c}{}\\
\multicolumn{6}{>{\normalsize}c}{}\\
\multicolumn{6}{>{\normalsize}c}{}\\
\hline
                                 &\raisebox{1.5ex}[12pt]{\tiny{$2\le q<4$}}\,\vline\,\raisebox{1.5ex}[12pt]{\tiny{$2\le q<4$}}   &  \raisebox{1.5ex}[12pt]{\tiny{$4=q<\mu$}}                            & \raisebox{1.5ex}[12pt]{\tiny{$4=q=\mu$}}                 & \raisebox{1.5ex}[12pt]{\tiny{$4<q<1+ 4/\mu'$}}                                  & \raisebox{1.5ex}[12pt]{\tiny{$4<q=1+4/\mu'$}}            \\
\hline
                     &                                                  &                                                    &                                  &local                                                     &                                               \\
                     &                                                  &                                                    &                                  &existence                                                 &                                               \\
                     &                                                  &                                                    &                                  &in $H^{1,\rho,\theta}$ for                                &                                               \\
                     &                                                  &                                                    &                                  &$\rho\in\left[\frac {10}3,\frac{10}3\vee m\right]$        &                                               \\
                     &                                                  &                                                    &                                  &$\theta\in [4,\mu]$                                       &                                               \\
\cdashline{5-5}[1pt/1pt]
                     &                                                  &                                                    &                                  &global                                                    &                                               \\
                     &                                                  &                                                    &                                  &existence                                                 &                                               \\
                     &                                                  &                                                    &                                  &in $H^{1,\rho,\theta}$ for                                &                                               \\
                     &                                                  &                                                    &                                  &$\rho\in\left[\frac {10}3,\frac{10}3\vee m\right]$        &                                               \\
                     &                                                  &                                                    &                                  &$\theta\in [q,\mu]$                                       &                                               \\
\cdashline{5-5}[1pt/1pt]
$2\le p\le \frac 83$ &                                                   &existence--                                         &existence--                       &existence--                                               &existence--                                    \\
                     &                                                  &uniqueness                                          &uniqueness                        &uniqueness                                                &uniqueness                                       \\
                     &                                                  &in $H^{1,\rho,\theta}$ for                          &in $H^{1,\rho,2}$ for             &in $H^{1,\rho,\theta}$ for                                &in $H^{1,\rho,\theta}$ for                       \\
                     &                                                  &$\rho\in\left[\frac {10}3,\frac{10}3\vee m\right]$ &$\rho\in\left[\frac {10}3,\frac{10}3\vee m\right]$&$\rho\in\left[\frac {10}3,\frac{10}3\vee m\right]$&$\rho\in\left[\frac {10}3,\frac{10}3\vee m\right]$                        \\
                     &                                                  &$\theta\in [4,\mu]$                                 &                                  &$\theta\in [2(q-2),\mu]$                                  &$\theta\in [2(q-2),\mu]$    \\
\cdashline{3-3}[1pt/1pt]\cdashline{5-6}[1pt/1pt]
                     & well--                                           & well--                                             &                                  &well--                                                    &well--                                                 \\
                     & posedness                                        & posedness                                          &                                  &posedness                                                 &posedness                                               \\
                     &in $H^{1,\rho,\theta}$ for                        & in $H^{1,\rho,\theta}$ for                         &                                  &in $H^{1,\rho,\theta}$ for                                &in $H^{1,\rho,\theta}$ for                             \\
                     &$\rho\in   \left[\frac {10}3,\frac{10}3\vee m\right]$&$\rho\in\left[\frac {10}3,\frac{10}3\vee m\right]$&                                  &$\rho\in\left[\frac {10}3,\frac{10}3\vee m\right]$       &$\rho\in   \left[\frac {10}3,\frac{10}3\vee m\right]$                        \\
                     &$\theta\in [4,4\vee\mu]$                          &$\theta\in (4,\mu]$                                 &                                  &$\theta\in (2(q-2),\mu]$                                  &$\theta\in (2(q-2),\mu]$    \\
\hline
                     &                                                  &                                                    &                                  &local                                                     &                  \\
                     &existence                                         &existence                                           &  existence                       &existence                                                 &                 existence                                             \\
                     &in $H^{1,\rho,\theta}$ for                        &in $H^{1,\rho,\theta}$ for                          &in $H^{1,\rho,2}$ for             &in $H^{1,\rho,\theta}$ for                                &in $H^{1,\rho,\theta}$ for                                       \\
                     &$\rho\in\left[\frac {10}3,\frac{10}3\vee m\right]$&$\rho\in\left[\frac {10}3,\frac{10}3\vee m\right]$  &$\rho\in\left[\frac {10}3,\frac{10}3\vee m\right]$&$\rho\in\left[\frac {10}3,\frac{10}3\vee m\right]$    &     $\rho\in   \left[\frac {10}3,\frac{10}3\vee m\right]$\\
$\frac83<p<\frac{10}3$         &$\theta\in [4,4\vee\mu]$                          &$\theta\in [4,\mu]$                                 &                                  &$\theta\in [4,\mu]$                                       &                 $\theta\in [2(q-2),\mu]$                                                                                    \\
\cdashline{5-5}[1pt/1pt]
                     &                                                  &                                                    &                                  &global                                                    &                                               \\
                     &                                                  &                                                    &                                  &existence                                                 &                                               \\
                     &                                                  &                                                    &                                  &in $H^{1,\rho,\theta}$ for                                &                                               \\
                     &                                                  &                                                    &                                  &$\rho\in\left[\frac {10}3,\frac{10}3\vee m\right]$        &                                               \\
                     &                                                  &                                                    &                                  &$\theta\in [q,\mu]$                                       &                                               \\
\hline
                     &                                                  &                                                    &                                  &local                                                     &                                               \\
                     &existence                                         &existence                                           &  existence                       &existence                                                 &                 existence                                             \\
                     &in $H^{1,\rho,\theta}$ for                        &in $H^{1,\rho,\theta}$ for                          &in $H^{1,\rho,2}$ for             &in $H^{1,\rho,\theta}$ for                                &in $H^{1,\rho,\theta}$ for                                       \\
                     &$\rho\in\left[\frac {10}3, m\right]$              &$\rho\in\left[\frac {10}3,m\right]$                 &$\rho\in\left[\frac {10}3,m\right]$&$\rho\in\left[\frac {10}3,m\right]$                      &$\rho\in\left[\frac {10}3, m\right]$\\
$\frac{10}3=p\le m$  &$\theta\in [4,4\vee\mu]$                          &$\theta\in [4,\mu]$                                 &                                  &$\theta\in [4,\mu]$                                       &$\theta\in [2(q-2),\mu]$\\
\cdashline{5-5}[1pt/1pt]
                     &                                                  &                                                    &                                  &global                                                    &                                               \\
                     &                                                  &                                                    &                                  &existence                                                 &                                               \\
                     &                                                  &                                                    &                                  &in $H^{1,\rho,\theta}$ for                                &                                               \\
                     &                                                  &                                                    &                                  &$\rho\in\left[\frac {10}3, m\right]$                      &                                               \\
                     &                                                  &                                                    &                                  &$\theta\in [q,\mu]$                                       &                                               \\
\hline
                     &local                                             &local                                               &local                             &local                                                     &local                                            \\
                     &existence                                         &existence                                           &existence                         &existence                                                 &existence                                                   \\
                     &in $H^{1,\rho,\theta}$ for                        &in $H^{1,\rho,\theta}$ for                          &in $H^{1,\rho,2}$ for             &in $H^{1,\rho,\theta}$ for                                &in $H^{1,\rho,\theta}$ for                         \\
                     &$\rho\in\left[\frac{10}3,m\right]$                &$\rho\in\left[\frac{10}3,m\right]$                  &$\rho\in\left[\frac{10}3,m\right]$&$\rho\in\left[\frac{10}3,m\right]$                        &$\rho\in\left[\frac{10}3,m\right]$       \\
$\frac{10}3<p<1+\frac{10}{3m'}$    &$\theta\in [4,4\vee\mu]$                          &$\theta\in [4,\mu]$                                 &                                  &$\theta\in [4,\mu]$                                       &$\theta\in [2(q-2),\mu]$    \\
\cdashline{2-6}[1pt/1pt]
                     &global                                            &global                                              &global                            &global                                                    &global                                           \\
                     &existence                                         &existence                                           &existence                         &existence                                                 &existence                                                   \\
                     &in $H^{1,\rho,\theta}$ for                        &in $H^{1,\rho,\theta}$ for                          &in $H^{1,\rho,2}$ for             &in $H^{1,\rho,\theta}$ for                                &in $H^{1,\rho,\theta}$ for                         \\
                     &$\rho\in[p,m]$                                    &$\rho\in[p,m]$                                      &$\rho\in[p,m]$                    &$\rho\in[p,m]$                                            &$\rho\in[p,m]$                            \\
                    &$\theta\in [4,4\vee\mu]$                           &$\theta\in [4,\mu]$                                 &                                  &$\theta\in [q,\mu]$                                       &$\theta\in [2(q-2),\mu]$    \\
\hline
&\multicolumn{5}{>{\tiny}c|}{} \\
$\frac{10}3<p=1+\frac{10}{3m'}$&\multicolumn{5}{>{\tiny}c|}{no results} \\
&\multicolumn{5}{>{\tiny}c|}{} \\
\hline
\end{tabular}

\bigskip

\setlength{\tabcolsep}{5.8pt}
\renewcommand{\arraystretch}{0.7}
\setlength{\arrayrulewidth}{0.7pt}
\begin{tabular}{|>{\tiny}c|>{\tiny}c|>{\tiny}c|>{\tiny}c|>{\tiny}c|}
\multicolumn{5}{>{\normalsize}c}{\sc Table 8. \rm Further results when $N\ge 6$  and $\essinf_\Omega \alpha, \essinf_{\Gamma_1}\beta>0$}\\
\multicolumn{5}{>{\normalsize}c}{}\\
\hline
                                 &\raisebox{1.5ex}[12pt]{\tiny{$2\le q\le 1+\rgamma/2$}}   &  \raisebox{1.5ex}[12pt]{\tiny{$1+\rgamma/2<q\le\rgamma$}}                            & \raisebox{1.5ex}[12pt]{\tiny{$\rgamma<q<1+\rgamma/\mu'$}}                 & \raisebox{1.5ex}[12pt]{\tiny{$\rgamma<q=1+\rgamma/\mu'$}}                                            \\
\hline
                                 &well--posedness                                   &                                         & local  existence                                                    &                                                                                \\
$2\le p\le 1+\romega/2$          &in $H^{1,\rho,\theta}$                            &                                         & in $H^{1,\rho,\theta}$                                        &                                                                                \\
                                 &$\rho\in[\romega,\romega\vee m]$                  &                                         & $\rho\in[\romega,\romega\vee m]$       &                                                                                \\
                                 &$\theta\in[\rgamma,\rgamma\vee\mu]$               &                                         & $\theta\in[\rgamma,\mu]$                &                                                                                \\
\cline{1-2} \cdashline{4-4}[1pt/1pt]
                                 &\multicolumn{2}{>{\tiny}c|}{\qquad \qquad\qquad \qquad existence}                           &global existence                                  &                                                                                     \\
$1+\romega/2<p\le\romega$        &\multicolumn{2}{>{\tiny}c|}{\qquad \qquad\qquad \qquad in $H^{1,\rho,\theta}$}              & in $H^{1,\rho,\theta}$                           &                                                                                    \\
                                 &\multicolumn{2}{>{\tiny}c|} {\qquad \qquad\qquad \qquad $\rho\in[\romega,\romega\vee m]$}   &$\rho\in[\romega,\romega\vee m]$&                                                                        \\
                                 & \multicolumn{2}{>{\tiny}c|}{\qquad \qquad\qquad \qquad $\theta\in[\rgamma,\rgamma\vee\mu]$}&$\theta\in[q,\mu]$        &                                                                           \\
\cline{1-4}
                                 &  \multicolumn{2}{>{\tiny}c|}{local existence}                                               &local existence                   &                                                               \\
                                 &  \multicolumn{2}{>{\tiny}c|}{in $H^{1,\rho,\theta}$}                                        & in $H^{1,\rho,\theta}$                         &                              \\
                                 &\multicolumn{2}{>{\tiny}c|} {$\rho\in[\romega,m]$}                                           &$\rho\in[\romega,m]$&                                                                        \\
$\romega<p<1+\romega/m'$         & \multicolumn{2}{>{\tiny}c|}{$\theta\in[\rgamma,\rgamma\vee\mu]$}                            &$\theta\in[\rgamma,\mu]$        &                                                                           \\
\cdashline{2-4}[1pt/1pt]
                                 &\multicolumn{2}{>{\tiny}c|}{global existence}                                                &global   existence                         &          \\
                                 &\multicolumn{2}{>{\tiny}c|}{in $H^{1,\rho,\theta}$ }                                         & in $H^{1,\rho,\theta}$          &                                      \\
                                 &\multicolumn{2}{>{\tiny}c|} {$\rho\in[p,m]$}                                                 &$\rho\in[p,m]$&                                                                        \\
                                 & \multicolumn{2}{>{\tiny}c|}{$\theta\in[\rgamma,\rgamma\vee\mu]$}                            &$\theta\in[q,\mu]$        &                                                                           \\
\cline{1-4}
&\multicolumn{4}{>{\tiny}c|}{} \\
$\romega<p=1+\romega/m'$&\multicolumn{4}{>{\tiny}c|}{\qquad\qquad \qquad\qquad \qquad\qquad \qquad\qquad \qquad\qquad \qquad\qquad\qquad\qquad no results} \\
&\multicolumn{4}{>{\tiny}c|}{} \\
\hline
\end{tabular}

\end{center}
\end{table}
\def\cprime{$'$}
\providecommand{\bysame}{\leavevmode\hbox to3em{\hrulefill}\thinspace}
\providecommand{\MR}{\relax\ifhmode\unskip\space\fi MR }
\providecommand{\MRhref}[2]{%
  \href{http://www.ams.org/mathscinet-getitem?mr=#1}{#2}
}
\providecommand{\href}[2]{#2}

\end{document}